\documentclass[11pt,reqno]{amsart}

\usepackage[T1]{fontenc}
\usepackage{amsmath,amssymb,amsthm}
\usepackage{mathrsfs}
\usepackage{mathtools}
\usepackage[margin=1in]{geometry}
\usepackage{enumitem}
\usepackage{array}
\usepackage{longtable}
\usepackage{booktabs}
\usepackage[hidelinks]{hyperref}

\newtheoremstyle{rr}%
  {\medskipamount}{\medskipamount}%
  {\itshape}{}%
  {\bfseries}{.}{ }%
  {\thmname{#1}\thmnote{ #3}}
\newtheoremstyle{rrup}%
  {\medskipamount}{\medskipamount}%
  {\normalfont}{}%
  {\bfseries}{.}{ }%
  {\thmname{#1}\thmnote{ #3}}

\theoremstyle{rr}
\newtheorem*{rthm}{Theorem}
\newtheorem*{rlem}{Lemma}
\newtheorem*{rprop}{Proposition}
\newtheorem*{rcor}{Corollary}

\theoremstyle{rrup}
\newtheorem*{rdef}{Definition}
\newtheorem*{rrem}{Remark}

\newcommand{\Sin}[1]{\Sigma^{\mathrm{in}}_{#1}}
\newcommand{\Pin}[1]{\Pi^{\mathrm{in}}_{#1}}

\DeclareMathOperator{\cSR}{cSR}
\newcommand{\SRp}{\mathrm{SR}_{\mathrm{p}}}
\newcommand{\srS}{\mathrm{sr}_{\mathrm{Sacks}}}
\newcommand{\SRsym}{\mathrm{SR}^{\sim}}
\newcommand{\SRAK}{\mathrm{SR}^{\mathrm{AK}}}
\newcommand{\AKR}{\mathrm{R}}
\DeclareMathOperator{\SC}{SC}
\DeclareMathOperator{\tr}{tr}
\DeclareMathOperator{\qr}{qr}
\DeclareMathOperator{\rank}{rank}
\DeclareMathOperator{\Sp}{Sp}
\DeclareMathOperator{\Adm}{Adm}
\DeclareMathOperator{\Th}{Th}
\DeclareMathOperator{\tp}{tp}
\DeclareMathOperator{\Aut}{Aut}
\DeclareMathOperator{\Mod}{Mod}
\DeclareMathOperator{\trace}{trace}
\DeclareMathOperator{\Spec}{Spec}
\DeclareMathOperator{\Taut}{Taut}
\DeclareMathOperator{\Sent}{Sent}
\DeclareMathOperator{\tc}{tc}
\DeclareMathOperator{\ot}{ot}
\DeclareMathOperator{\Def}{Def}

\newcommand{\Ccl}{\mathcal{C}}
\newcommand{\Ffix}{\mathfrak{F}}
\newcommand{\Szero}{S^{\circ}}

\newcommand{\woneCK}{\omega_1^{\mathrm{CK}}}
\newcommand{\Lom}{L_{\omega_1\omega}}

\newcommand{\rhole}{\rho^{\leq}}
\newcommand{\rhosym}{\rho^{\sim}}

\newcommand{\iso}{\cong}
\newcommand{\restr}{\!\restriction\!}

\newcommand{\dq}{\textup{(RIG)}}           
\newcommand{\bff}{back-and-forth}

\newcommand{\AKbridge}{\textup{(AK-$\sim$-bridge)}}
\newcommand{\REALB}{\textup{(REAL-B)}}
\newcommand{\NSK}{\textup{(NS-K)}}
\newcommand{\NSKp}{\ensuremath{\textup{(NS-K)}^{+}}}
\newcommand{\BRctree}{\textup{(BR-c$^{\mathrm{tree}}$)}}
\newcommand{\BRcuns}{\textup{(BR-c$^{\mathrm{tree/uns}}$)}}

\allowdisplaybreaks
\numberwithin{equation}{section}

\begin{document}
\title[$\omega_1$-anchored labels in minimal counterexamples to Vaught's conjecture]{$\omega_1$-anchored labels in minimal counterexamples to Vaught's conjecture: a per-witness trichotomy and an unconditional stationary dichotomy}

\author{Mohammad Assem Mahmoud}
\address{University of Toronto Mississauga}
\email{mo.mahmoud@utoronto.ca}
\subjclass[2020]{Primary 03C75; Secondary 03C70, 03D60, 03E15}
\keywords{Vaught's conjecture, Scott rank, admissible ordinals, back-and-forth relations, weakly scattered theories, Harrison linear order}
\date{August 15, 2026}

\begin{abstract}
Let $\varphi$ be a minimal counterexample to Vaught's conjecture in the sense of
[Mon13, Def.~3.1]; such a $\varphi$ exists if Vaught's conjecture fails, by
Steel [Ste78] and Harnik--Makkai [HM77]. We combine two bodies of work on the
models of such a $\varphi$ that have not previously been brought into contact at
the level of statements: the analysis of Gonzalez--Rossegger--Turetsky [GRT],
who show that at every countable level $\beta$ exactly one $\equiv_\beta$-class
$\Ccl_\beta$ of models of $\varphi$ is uncountable and that at fixed points of an
associated function this class has a distinguished member of least Scott rank
(its \emph{label}); and the supply of models with prescribed $\omega_1^A$ coming
from higher recursion theory, namely Montalb\'an's Gandy-basis lemma [Mon13,
Lemma~3.4] and Sacks' $\Sigma_1$-hull club in \emph{Bounds on weak scattering}
[Sac07, Thm.~5.3]. All results are theorems of ZFC under the standing hypotheses
(H0)--(H3) of \S2.12. Our main results are:
\begin{enumerate}[label=\textup{(\roman*)},leftmargin=2.2em]
\item a \emph{per-witness trichotomy} (Thm.~5.1): for every limit $\lambda$ in the
fixed-point club $\Ffix$ above $\qr(\varphi)$, every model $A\models\varphi$ with
$\omega_1^A=\lambda$ and $\cSR(A)\ge\lambda$ either is the level-$\lambda$ label
$K_\lambda$ (which forces $\omega_1^{K_\lambda}=\lambda$), or lies outside
$\Ccl_{\lambda+1}$ (which forces the existence of two non-isomorphic models of
$\cSR=\lambda+1$), or is the level-$(\lambda+1)$ label $K_{\lambda+1}$ (which
forces $\omega_1^{K_{\lambda+1}}=\lambda$, so that $K_{\lambda+1}$ attains the
Nadel bound);
\item \emph{coordinate identities} (Thm.~5.2): the first and third branch
conditions are, level by level, equivalent to computations of $\omega_1$ of the
labels themselves ($\omega_1^{K_\lambda}=\lambda$, respectively
$\omega_1^{K_{\lambda+1}}=\lambda$);
\item a \emph{seeding theorem} (Thm.~5.3): on a club $\Szero$ of limit ordinals,
Sacks' construction supplies at every level a model of top rank with prescribed
$\omega_1$, and his atomic chain consists of the labels $K_\lambda$;
\item an unconditional \emph{stationary dichotomy} (Thm.~5.4(ii)): on $\Szero$,
either stationarily many successor levels $\lambda+1$ carry at least two
non-isomorphic models of Scott rank $\lambda+1$, or stationarily many labels
$K_{\lambda+1}$ attain the Nadel bound with $\omega_1^{K_{\lambda+1}}=\lambda$.
\end{enumerate}
On the set of models with $\omega_1=\lambda$ at a branch level (the \emph{fiber})
we then develop a further layer of structure: a uniqueness theorem for the
node-saturated model (Thm.~5.6), a complete isomorphism invariant --- the
\emph{trace} (Lemmas 5.7, 5.8) --- with a countable spectrum (\S5.11), the fact
that the fiber is a single $\equiv_\lambda$-class (Lemma~5.9), a localization
result for supported types (\S5.15), and two lower bounds on the types realized
by the labels (Lemmas 5.13, 5.14). We prove nothing bearing on Vaught's
conjecture itself; \S1.3 records explicitly that neither limb of the route to
Vaught's conjecture described at [GRT, p.~3] is approached.
\end{abstract}

\maketitle

\section*{Conventions and terminology}

\noindent\textbf{Scott ranks.} Several notions of Scott rank occur below and are
kept rigorously distinct (\S\S2.2--2.5):
\begin{itemize}[leftmargin=1.6em,itemsep=1pt]
\item the \emph{categoricity rank} $\cSR$ of [Mon15] (\S2.2). This is the rank
written $\mathrm{SR}$ in [GRT]: $\text{GRT-}\mathrm{SR}=\cSR$ with no discrepancy.
Unless the contrary is said, every rank statement below is in $\cSR$;
\item the \emph{parameterized rank} $\SRp$ (\S2.2), defined by
$\Sin{\alpha+2}$ Scott sentences. Ranks are unparameterized unless $\SRp$ is
written;
\item the \emph{Sacks fragment rank} $\srS$ (\S2.4), from [Sac07, \S2, (2.5)];
\item Montalb\'an's \emph{symmetric rank} $\SRsym=\sup\{\rhosym(\bar a)+1\}$
([Mon15, p.~5432]) (\S2.4);
\item the \emph{Ash--Knight rank} $\AKR$ of [AK00, \S6.7], which is the rank
written $\mathrm{SR}$ in [Mon13] (\S2.3).
\end{itemize}
The lowercase rank $\mathrm{sr}(A)=\sup\rhosym(\bar a)$ of [Mon15, p.~5432] is yet
another rank and is never used here; we mention it only because the notation is
easily confused with $\srS$ (see \S2.4).

\smallskip
\noindent\textbf{Rank transport.} A step in a proof that passes between the
Sacks fragment rank $\srS$ and the categoricity rank $\cSR$ will be called a
\emph{rank transport}. Every rank transport in this paper is routed through
Lemma~TV of \S2.5, which also records the complete list of such steps:
step~(3) in the proof of Lemma~5.9; Corollary~5.10 and Lemma~AL(vii-2)
(lower-bound form only); and Lemma~S(iv)--(v), hence Theorem~5.3(i)--(ii)
(Theorem~(TV-b) at the two values). This is worth isolating because the
comparison of $\srS$ with $\cSR$ is not available from the literature in the
form needed here; see \S2.5.

\smallskip
\noindent\textbf{Terms introduced in this paper.} The following are not standard
terminology and are defined where indicated: \emph{label} (\S3.2, following
[GRT, Def.~3.1, p.~13]); \emph{crowded class} (\S3.1); \emph{high witness} and
the sets $W_\lambda$, $W^{\mathrm{high}}_\lambda$, $\mathrm{HIGH}(\lambda)$
(\S3.4); \emph{glue} and \emph{split} (\S3.6); \emph{fiber}, \emph{trace}, and
the \emph{trace spectrum} $\Spec_\lambda$ (Definitions 5.5 and 5.11); \emph{node} and
\emph{node theory} (\S3.7, following [Sac07, \S4]); \emph{node-saturated}
(Definition~5.5); \emph{supported} type (\S5.15, following [Mon15, Def.~3.1]); and
\emph{rank transport} (above). Where we say that a structure \emph{attains the
Nadel bound} (equivalently, is \emph{Nadel-maximal}) we mean that
$\cSR(A)=\omega_1^A+1$, the largest value permitted by \S2.7.

\smallskip
\noindent\textbf{Citations and standing assumptions.} All statements about
$\varphi$ are conditional on the failure of Vaught's conjecture. The base theory
is ZFC throughout; no determinacy hypothesis and no large cardinal is used in
any proof below (see Appendix~A). Citations to Sacks, \emph{Bounds on weak
scattering} [Sac07], follow the numbering of the December~2004 preprint; a
concordance with the published edition is given in Appendix~A.7. The numbering
of [GRT] collides at $3.1$ (Proposition~3.1 is on p.~12, Definition~3.1 on
p.~13), so every citation to a ``3.1'' of [GRT] carries its page.

\section{Introduction}

\subsection*{1.1. Setting and what is proved}

A counterexample to Vaught's conjecture [Vau61], if one exists, admits a minimal refinement $\varphi$: an $\Lom$ sentence with exactly $\aleph_1$ countable models, scattered in the model-level sense that for each $\alpha<\omega_1$ there are at most countably many $\equiv_\alpha$-classes among its models ([Mon13, Def.~1.4]), and such that every $\Lom$ sentence decides $\varphi$ up to countably many models ([Mon13, Def.~3.1, p.~6]; existence from [Ste78, Thm.~1.5.11] via [Mon13, p.~6], and [GRT, Def.~3.3, p.~13] citing [HM77]---both external, by reference, attribution only). Two structure theories for the models of such a $\varphi$ exist in the literature, in disjoint vocabularies.

On the \bff\ side, Gonzalez--Rossegger--Turetsky prove ([GRT, Thm.~3.6, p.~14], and the first paragraph of its proof) that at every countable level $\beta$ exactly one $\equiv_\beta$-class $\Ccl_\beta$ of models of $\varphi$ is uncountable, that these classes are nested and limit-coherent, and that at fixed points of the induced function $f$ the class $\Ccl_\alpha$ has a unique member $K_\alpha$ of minimal Scott rank---the \emph{label}. On the recursion-theoretic side, Montalb\'an ([Mon13, Lemma~3.4, p.~7], Gandy basis) supplies, at every admissible $\lambda$ (relativized: $\lambda\in\Adm(t_\varphi)$, the countable ordinals admissible relative to a real code $t_\varphi$ of $\varphi$; see \S2.9), a model $A\models\varphi$ with $\omega_1^A=\lambda$ and Scott rank in $\{\lambda,\lambda+1\}$; and Sacks ([Sac07, Prop.~5.2 and Thm.~5.3, p.~14]; numbering per the December 2004 preprint) supplies, on a club $C_{5.3}$ of $\Sigma_1$-hull ordinals, both an exact-rank atomic chain $A_\lambda$ and a $\lambda$-saturated top-rank witness $B_\lambda$ with $\omega_1^{B_\lambda}=\lambda$. Neither side mentions the other's objects: GRT contains no $\omega_1^A$ apparatus, and [Sac07] contains no $\equiv$-class, label, uniqueness, or trichotomy content.

This paper's contribution is the interaction. The trichotomy (Thm.~5.1) classifies each pinned witness against the labels; the coordinate identities (Thm.~5.2) convert the first and third branch conditions into $\omega_1$-computations of the labels themselves; the seeding theorem (Thm.~5.3) shows on a club that $W_\lambda^{\mathrm{high}}$ is non-empty, and identifies the atomic chain of [Sac07, Thm.~5.3] with the labels there; and the stationary dichotomy (Thm.~5.4(ii)) extracts an unconditional either/or with genuinely different model-theoretic content on the two sides. (``Unconditional'' here and in the title means: conditional on nothing beyond
the standing hypotheses (H0)--(H3) of \S2.12 --- in particular not on any of
the open questions of \S6, such as $\dq$ --- \emph{together with} the cited
literature as itemized in A.1 and A.5, three statements of which are consumed
at statement level without a printed proof in their sources: the two [Mon15,
\S3.1] coincidence clauses behind TV-a(i)/TV-c(ii), and the [Sac07, p.~14]
pin behind Lemma~S; each is flagged at its point of use.) A second layer of structure then develops on the $\omega_1=\lambda$ fiber of the branch node: the trace invariant classifies the fiber completely (Lemmas 5.7 and 5.8) with countable spectrum (\S5.11, (C1)); the fiber is a single $\equiv_\lambda$-class (Lemma~5.9; one of the rank transports listed in A.5); the fragment-to-$\Pin{\lambda}$ correspondence localizes at supported cells (Lemma~5.15); and the labels acquire realization floors under the glue coordinate (Lemmas 5.13, 5.14). The remaining open questions are stated precisely (\S6): the questions there fall into two groups, one about which types the labels realize ($\NSKp$, Question~2 of \S6) and one about the relation between fragment types and $\Pin{\lambda}$-types ($\BRcuns$, Question~3 of \S6), and both turn on the same obstruction.

Where a proof passes between Sacks' fragment rank and the categoricity rank,
it does so through the single transport result of \S2.5. The middle link of that
transport is, as far as we have been able to determine, not stated anywhere in
the literature; we prove it here as Theorem~(TV-b). The complete list of
transport steps is recorded at Lemma~TV in \S2.5.

\subsection*{1.2. Attribution}

Almost all of the apparatus this paper stands on is due to others, and we
record the attributions precisely. The minimality dichotomy, the uncountable
$\equiv_\alpha$-class at every level together with its coherence at limits, the
function $f$, and the uniqueness of the model of least rank at fixed points of
$f$ are due to Gonzalez--Rossegger--Turetsky ([GRT, Thm.~3.6, p.~14, proof
paragraphs~1--2]); the club/$\Sin{<\alpha}$ form of the same phenomenon is in
Montalb\'an ([Mon13, Lemma~3.3, pp.~6--7]). Existence, uniqueness and
universality of labels are [GRT, Lemma~2.3, p.~7], [GRT, Def.~3.1, p.~13], [GRT,
Prop.~3.1 and Cor.~3.2, p.~12], and [GRT, Cor.~3.4/3.5, p.~13]; the
collapse mechanism used in branch~(III) is [GRT, Prop.~3.1 and Cor.~3.2, p.~12] (compare also [GRT, Prop.~3.13, p.~17]). The supply
of models with $\omega_1^A$ equal to a given admissible ordinal, with rank in
$\{\alpha,\alpha+1\}$, is [Mon13, Lemma~3.4, p.~7]. The club supply of top-rank
models with prescribed $\omega_1$ and the exact-rank atomic chain are [Sac07,
Prop.~5.2 and Thm.~5.3, p.~14]; the necessity direction is [Sac07, Cor.~6.2 and Cor.~6.4],
quoted at [Mon13, p.~4]. Counting results by rank are in [Mon13, p.~4] and
[GRT, Thm.~2.1 and Cor.~2.11, pp.~7, 11--12], and, in the analytic setting, in
Larson [Lar] and Larson--Shelah [LS]. Our conventions on ranks follow [Mon15,
pp.~5432--5433] together with [Sac07, \S2] for $\srS$.

At the level of technique, the implication ``same realized types
$\Rightarrow$ isomorphic'' in Lemma~BF has a precedent in Sacks' proof of
[Sac07, Thm.~8.1] (pp.~23--24, steps (8.20)--(8.22)), where equality of realized
type sets is established between a structure and a companion built from it, and
homogeneity then yields an isomorphism. That theorem carries no saturation
hypothesis, asserts no uniqueness, and compares a structure only with a
companion drawn from its own canonical tower; what we take from it is the
mechanism, not the statement (see the remark following Lemma~BF).

What we claim as new is the combination of the supply of models with prescribed
$\omega_1^A$ from [Mon13] with the analysis of labels in [GRT]: the identity
$\omega_1^{K_\lambda}=\lambda$ in branch~(I); the $\equiv_{\lambda+1}$-trichotomy
against $K_{\lambda+1}$; branch~(III), in which the label attains the Nadel
bound; the stationary pigeonhole of \S5.4; and the uniqueness statement of
Theorem~5.6. Neither limb of the route to Vaught's conjecture recorded at
[GRT, p.~3] --- (a) at least two models of every unparameterized rank, or (b)
exactly one model of each parameterized rank --- is implied by any of this; see
\S1.3.

\subsection*{1.2b. What is new}

We have looked for each of the following in the literature cited above and have
not found it; we state them as claims of novelty in that limited sense.
\begin{enumerate}[label=\textup{(\arabic*)},leftmargin=2.4em]
\item \emph{Branch~(I): a label with $\omega_1^{K_\lambda}=\lambda$.} That
$K_\lambda$ exists and has rank exactly $\lambda$ is [GRT, Thm.~3.6]; the
computation of $\omega_1^{K_\lambda}$ is new, and [GRT] contains no
$\omega_1^A$ apparatus at all.
\item \emph{Branch~(II): $(\lambda+1)$-splitting implies $\lambda+1\notin\Ffix$.}
We have found no statement of failure of the fixed-point property, no model
with prescribed $\omega_1$, and no comparable statement in the literature (the
implication is recorded at Theorem~5.1(II)).
\item \emph{Branch~(III): a label attaining the Nadel bound.} The collapse
mechanism is known (the pattern of [GRT, Prop.~3.13] and [GRT, Cor.~3.2]); a
label whose $\omega_1$ is computed and which attains the Nadel bound does not
appear, and the model $B_\alpha$ of [Sac07, Thm.~5.3] is nowhere identified with
a label, nor is any uniqueness asserted for it.
\item \emph{The stationary dichotomy of Theorem~5.4.} We have found nothing of
this shape. The nearest statement in the literature is Larson--Shelah [LS], in a
different setting and with different content. Becker [Bec94, \S\S2, 4] proves
club and stationary per-level orbit-uniqueness dichotomies for minimal
counterexamples to the topological Vaught conjecture, expressed in terms of
Borel rank; there is no overlap with the statements below.
\item \emph{Theorem~5.6.} This strengthens ``an $\alpha$-saturated model'' in
[Sac07, Thm.~5.3] to a uniqueness statement about models of the node theory.
Contrast the precedent in [Sac07, Thm.~8.1] discussed above: no saturation
hypothesis, no uniqueness, and a comparison confined to a single tower.
\item \emph{The question $\BRcuns$ of \S6, isolated as a definite gap.}
[Sac07, \S2] states no relation, in either direction, between equality of
tower types and $\equiv_\delta$- or $\Pin{\delta}$-types; there is no
back-and-forth or Karp-style material in that paper at all.
\item \emph{Proposition (AK-$\sim$-bridge) of \S2.5.} [AK00, \S6.7] states,
without proof, a value identity of the same shape for its own pair
$(\SRAK,\AKR)$; we have found no comparison of $\SRAK$ with $\SRsym$
anywhere. The proposition is consumed in no proof below.
\end{enumerate}

\subsection*{1.3. Relation to the two-limb route of \textup{[GRT, p.~3]}}

Gonzalez--Rossegger--Turetsky ([GRT, p.~3]) record a route to Vaught's
conjecture by contradiction, through either of two strengthenings of their own
results: limb~(a), that counterexamples have at least two models of every
unparameterized Scott rank; limb~(b), that some counterexample has exactly one
model of each parameterized Scott rank. Either limb would prove Vaught's
conjecture outright, against their Cor.~2.11 and Thm.~3.6 respectively.

The same page, and again p.~12, describes the earlier result of Sacks [Sac83] as
giving multiple models of every ``$\Sigma^1_2$-admissible Scott rank''
(superscript~$1$, subscript~$2$). This appears to combine two distinct but
compatible facts, both of which are in [Sac07]. The \emph{bound} is one of
$\Sigma_2$-admissibility: [Sac07, \S1] restates the result of [Sac83] as saying
that every countable model of $T$ has a countable copy in $L(\beta,T)$ for some
$\beta<\sigma^T_2$, where $\sigma^T_2$ is the least $\alpha$ with $L(\alpha,T)$
$\Sigma_2$-admissible (a subscript, at the level of KP; $\sigma^T_2$ recurs at
[Sac07, (6.1)]). Separately, $\Sigma^1_2$ is the \emph{complexity of the
predicate} ``Vaught's conjecture holds for $T$'': [Sac07, \S5, Prop.~5.1] shows
it to be $\Sigma_1$ over $L(\omega_1^{L(T)},T)$, hence $\Sigma^1_2$ (a
superscript, in the analytical hierarchy). The two facts do not conflict, and
neither is used in any proof below; we record the point only so that the reader
is not misled by the phrase when consulting [GRT].

\emph{Neither limb is approached below.} The splitting side of Theorem~5.4(ii)
produces at least two models of rank $\lambda+1$ at stationarily many successor
levels only --- not at every rank --- so limb~(a) is untouched. Nothing below
constrains the counting of models by parameterized rank, so limb~(b) is
untouched; and [GRT, Cor.~2.11] already forbids (b) for $\Pin{2}$ theories.
Every statement below about the comparison of rank notions is a statement about
a single structure, and none of them carries an implication toward Vaught's
conjecture.

\subsection*{1.4. Related work}

None of the following is used in any proof below.
\begin{itemize}[leftmargin=1.4em]
\item \textbf{Harrison-Trainor}, \emph{Scott ranks of models of a theory}
[Har18]: a classification of Scott spectra, proved in ZFC together with
projective determinacy (so noted in that paper); the examples produced there
have continuum many models, as [GRT, p.~7] notes, and there is no counting of
models of a counterexample.
\item \textbf{Larson}, \emph{Scott processes} [Lar]: an approach through
$\Pi_1$-over-$\mathrm{HC}$ statements and the L\'evy collapse, giving models of
every limit rank in $(\qr(\varphi),\omega_2)$, and a production of minimal
counterexamples (Remark~10.9 there).
\item \textbf{Larson--Shelah} [LS], in the setting of the analytic Vaught
conjecture: a transfer of an exact count to club-many ranks (Thm.~0.2 there).
This is the closest statement in the literature to the branch analysis below;
it involves no anchoring at admissible ordinals, no $\equiv_{\lambda+1}$
trichotomy, and no labels.
\item \textbf{Becker} [Bec94]: necessary and sufficient conditions for
minimality in the setting of Polish group actions, some of them proved under
determinacy hypotheses (so noted in that paper).
\item \textbf{Sacks} [Sac83]: the precursor discussed in \S1.3.
\end{itemize}

\section{Preliminaries}

\subsection*{2.1. Back-and-forth and truth flow}
Formulas are those of $\Lom$; $\Pin{\alpha}$ and $\Sin{\alpha}$ denote the standard infinitary hierarchy over which the ranks below are defined (see [MonP2, Ch.~II] or [GRT, \S1]). The relations $\le_\alpha,\equiv_\alpha$ are those of [GRT, Def.~1.2, p.~6]. Karp orientation [GRT, Thm.~1.1, p.~6]: $(M,\bar a)\le_\alpha (N,\bar b)$ iff every $\Pin{\alpha}$ formula true of $\bar a$ in $M$ is true of $\bar b$ in $N$, iff every $\Sin{\alpha}$ formula true of $\bar b$ in $N$ is true of $\bar a$ in $M$ (clauses (2),(3); $\alpha\ge1$). Thus \emph{$X\le_\alpha Y$ moves $\Pi_\alpha$-truths of $X$ into $Y$}. Monotonicity: $\gamma\le\gamma'\Rightarrow\equiv_{\gamma'}\subseteq\equiv_\gamma$. At a limit $\lambda$: $\equiv_\lambda=\bigcap_{\gamma<\lambda}\equiv_\gamma$ ([GM23, p.~5], citing [MonP2, Def.~II.3.2]). Sentence level: $M\equiv_\alpha N$ iff $M,N$ satisfy the same $\Pin{\alpha}$ (equivalently $\Sin{\alpha}$) sentences.

\subsection*{2.2. The categoricity rank $\cSR$ (primary; ``$\mathrm{SR}$'' in GRT)}
$\cSR(M):=$ the least $\alpha$ such that $M$ has a $\Pin{\alpha+1}$ Scott sentence (Scott sentences in the sense going back to [Sco65]), as at [GRT, p.~5], following [Mon15]. Thus $\text{GRT-}\mathrm{SR}=\cSR$ exactly. $\SRp(M):=$ the least $\alpha$ such that $M$ has a $\Sin{\alpha+2}$ Scott sentence [GRT, p.~5]. Unless the contrary is said, every rank statement below is in $\cSR$, and ranks are unparameterized unless $\SRp$ is written. From [Mon15, Thm.~1.1] we use three of the equivalent conditions listed there, in the numbering of that paper, each read at a level $\alpha$ for a countable structure $M$: \textup{(U1)} every automorphism orbit of a tuple of $M$ is definable by a $\Sin{\alpha}$ formula without parameters; \textup{(U2)} $M$ has a $\Pin{\alpha+1}$ Scott sentence; \textup{(U5)} every $\Pin{\alpha}$-type realized in $M$ is $\Sin{\alpha}$-supported within $M$ in the sense of [Mon15, Def.~3.1] (recalled in \S5.15). These labels are used in Lemmas 5.13 and 5.14 and in \S7.H.

\subsection*{2.3. The Ash--Knight rank $\AKR$}
[Mon13]'s $\mathrm{SR}(A)=\sup\{\rhole(\bar a)+1\}$ equals the rank $\AKR$ of [AK00, \S6.7]; the identification is stated at [Mon13, p.~3]. We quote the defining clause of $\rhole$: $\rhole(\bar a)$ is the least ordinal $\alpha$ such that if all $\Pin{\alpha}$ formulas true of $\bar a$ are true of another tuple $\bar b$, then all $\Lom$ formulas true of $\bar a$ are true of $\bar b$. The comparison of $\AKR$ with $\cSR$ is [Mon15, p.~5433]: they are equal or differ by $1$; they coincide at limit values; and on computable structures they coincide at $\woneCK$ and $\woneCK+1$.

\subsection*{2.4. The Sacks fragment rank $\srS$}
The \emph{canonical tower} $T^A_\delta$, $L^A_\delta$ of a countable structure $A$ is that of [Sac07, \S2, clauses (1)--(4)]: $L^A_0$ is the finitary language of $A$'s signature (clause~(1)); $T^A_\delta:=\Th_{L^A_\delta}(A)$, so each level is a truth theory of $A$ (clause~(3)); $L^A_{\delta+1}$ is the least fragment containing $L^A_\delta$ together with $\bigwedge p$ for every non-principal complete $L^A_\delta$-type $p$ realized in $A$ (clause~(4)); at limits one takes unions (clause~(2)).
Following [Sac07, \S2, (2.5)]: $\srS(A):=$ the least $\delta$ such that $A$ is
the atomic model of $T^A_\delta$. (The bound variable in (2.5) is stated as
$\alpha$ where $\delta$ is meant; this is a typographical slip with no effect on
the definition.) The comparison of $\srS$ with another rank is given at
[Mon15, p.~5432]: ``This rank and the previous Scott rank, $\mathrm{SR}$, coincide
at the multiples of $\omega^2$. On computable structures, they also agree at
$\woneCK$ and $\woneCK+1$.'' The rank called $\mathrm{SR}$ there is the
symmetric one, $\SRsym=\sup\{\rhosym(\bar a)+1\}$, built from the symmetric
back-and-forth relations $\sim_\alpha$ (equality of atomic type at $0$;
symmetric one-element extensions at successors; the intersection of the earlier
relations at limits). It is \emph{not} the Ash--Knight rank $\AKR$; this
distinction is the reason for \S2.5.

\begin{rrem}[A misprint at {[Mon15, p.~5432]}]
The limit clause of the definition of $\sim_\alpha$ is printed in the
published edition as ``for \emph{some} $\beta<\alpha$''
[Mon15, \S3.1, clause~(3), p.~5432]. This is a
misprint: the intersection reading (``for all $\beta<\alpha$'') is the one
intended, and it is the one used throughout this paper. Two independent grounds
support the reading. First, clause~(2) makes $\sim_\alpha$ decreasing in
$\alpha$, so at a limit $\lambda$ the union reading would collapse $\sim_\lambda$
to $\sim_0$ and trivialize both $\mathrm{sr}$ and $\SRsym$. Second, the limit
clause is stated in ``for all'' form in the neighbouring literature: at
[AK00, \S6.7, p.~98] and at [Alv21, p.~1712], among others. In both of those
sources the relations are cumulative in the level and extend by tuples, so they
are not the relations of [Mon15]; see the qualification in \S2.5. What the two
sources establish is that the limit clause of a decreasing symmetric family is
an intersection wherever it is written out. No statement below depends on the misprint,
and every citation to [Mon15] in this paper is to the published edition.
\end{rrem}

Two notational cautions. First, the lowercase rank $\mathrm{sr}(A)=
\sup\rhosym(\bar a)$ of [Mon15, p.~5432] is a different rank from both $\srS$ and
$\SRsym$, and is never used here. Second, the presentation of the Sacks tower at
[Mon15, p.~5432] closes each successor stage under $\vee,\wedge,\neg,\exists,\forall$
and adds $\bigwedge p$ for each realized $L_\alpha(A)$-type, whereas clause~(4)
of [Sac07, p.~5] adds $\bigwedge p$ only for realized non-principal $n$-types;
the two generate the same rank. We emphasize that the coincidence quoted above
is at the level of ranks only: no correspondence between the two hierarchies
level by level is asserted anywhere in the literature we cite, and this is
exactly what is missing at the open question $\BRcuns$ (Question~3 of \S6).

\subsection*{2.5. Comparing $\srS$ with $\cSR$: Theorem \textup{(TV-b)}}

It is tempting to record, as a single convention, that ``at the values
$\omega_1^A$ and $\omega_1^A+1$ the ranks $\srS$, $\AKR$ and $\cSR$ agree''. That
is not available from the literature. The rank with which [Mon15, p.~5432] compares
$\srS$ is its own symmetric rank $\SRsym$, not the Ash--Knight rank $\AKR$: the
rank $\AKR=\sup\{\rhole(\bar a)+1\}$ of [Mon13, p.~3] is built from the one-sided
relations, [Mon15, p.~5433] compares $\AKR$ with two \emph{other} ranks and offers no
comparison between $\SRsym$ and $\AKR$, and the two suprema are taken over
different families of tuple relations. We have not found the identification of
$\SRsym$ with $\AKR$ stated anywhere. The comparison therefore has to be
decomposed into three steps, of which the first and third are in the literature
and the second is proved here.

\begin{rlem}[TV-a ($\srS\leftrightarrow\SRsym$)]
For countable $A$: \textup{(i)} for every value $v$ that is a multiple of $\omega^2$: $\srS(A)=v \iff \SRsym(A)=v$ \textup{[unrelativized]}; \textup{(ii)} for $v\in\{\omega_1^A,\omega_1^A+1\}$: $\srS(A)=v \iff \SRsym(A)=v$.
\end{rlem}
\noindent Clause~(i) is the statement quoted from [Mon15, p.~5432]. Two
provenance facts about that quotation, for parity with the disclosure made for
[AK00, \S6.7] below: it occurs in [Mon15]'s historical review (\S3.1 there,
marked as not needed for the rest of that paper), and it is printed without
proof; the same holds of the TV-c(ii) sentence at [Mon15, p.~5433]. Both are
consumed at statement level. Clause~(ii) is obtained from clause~(i) by a
routine relativization, which we write out once here: fix $X\in\Sp(A)$ with $\omega_1^X=\omega_1^A$ (the min is attained, [Mon13, p.~4]); $A$ has an $X$-computable copy; both ranks are isomorphism-invariant ($\srS$: [Sac07, p.~5]; $\SRsym$: the $\sim$-clauses are isomorphism-invariant); ``computable'' enters the clause of [Mon15] only through the anchor $\woneCK=\omega_1^{\emptyset}$, which relativizes to $\omega_1^X$ uniformly. Note that at $v=\omega_1^A$ the two clauses overlap: $\omega_1^A$ is admissible, hence closed under primitive recursive ordinal functions, hence a multiple of $\omega^2$. So at the limit value only clause~(i) is needed, and the relativization is required only at the successor value.

\begin{rlem}[TV-c ($\AKR\leftrightarrow\cSR$)]
\textup{(i)} $\AKR$ is the rank written $\mathrm{SR}$ in [Mon13] \textup{([Mon13, p.~3])}. \textup{(ii)} For every countable $A$: $\AKR(A)$ and $\cSR(A)$ are equal or differ by $1$; they are equal at limit values; and for $v\in\{\omega_1^A,\omega_1^A+1\}$: $\AKR(A)=v \iff \cSR(A)=v$ \textup{([Mon15, p.~5433]; the clause about computable structures is relativized as in TV-a(ii), and the clause about limit values needs no relativization)}. \textup{(iii)} \emph{Lower-bound corollary:} at a limit $\lambda$, $\cSR(A)<\lambda \Rightarrow \AKR(A)\le\cSR(A)+1<\lambda$; contrapositively $\AKR(A)\ge\lambda \Rightarrow \cSR(A)\ge\lambda$.
\end{rlem}
\noindent Clauses (i) and (ii) are as stated in the sources cited; (iii) is
immediate from (ii).

\begin{rthm}[(TV-b): the middle step]
For countable $A$ and $v\in\{\omega_1^A,\omega_1^A+1\}$:
\[ \SRsym(A)=v \iff \AKR(A)=v. \]
\textbf{Corollary (TV-b$^{-}$)} (lower bound): $\SRsym(A)\ge\omega_1^A \iff \AKR(A)\ge\omega_1^A$; below we use only the direction from left to right.\\
\textbf{Corollary (SR$^{\sim}$-Nadel)} (upper bound): $\SRsym(A)\le\omega_1^A+1$, and the bound is attained at the Harrison linear order (sharpness only; see \S7.H; no proof below uses the attainment).
\end{rthm}
\noindent This is the one place where we prove, rather than cite, a comparison
between rank notions. The proof below assembles it from statements in the
literature; we have not found the result itself stated anywhere, and we return
to that point in the remark after the proof.

\begin{proof}[Proof of \textup{(TV-b)} and its corollaries]
Write $\lambda:=\omega_1^A$, which is admissible ([Mon13, p.~4]), so that
$\omega\cdot\lambda=\lambda$. From the definitions cited --- the $\sim$-clauses of
[Mon15, p.~5432], read at limits as an intersection (\S2.4), together with
[GRT, p.~6, Def.~1.2 and Thm.~1.1], the latter being Karp's theorem [Kar65] ---
one derives the following two comparisons, tuple by tuple:
\begin{itemize}[leftmargin=2.1em]
\item[(T1)] $\rhosym(\bar a)\le\omega\cdot\max(\rhole(\bar a),1)$. \emph{Proof.}
First, $\sim_{\omega\cdot\alpha}\subseteq{\le_\alpha}\cap{\ge_\alpha}$ for every
$\alpha$, by induction. At $\alpha=0$: $\sim_0$ is equality of atomic type,
which is the two inclusions of $\Pin{0}$-types (the first printed line of the
$\le$-definition at $0$). At $\alpha>0$, using the second printed line of the
$\le$-definition: given $\beta<\alpha$ and $\bar d$ with $|\bar d|=n$,
monotonicity gives $\bar a\sim_{\omega\cdot\beta+n}\bar b$ (as
$\omega\cdot\beta+n<\omega\cdot(\beta+1)\le\omega\cdot\alpha$), $n$
applications of the one-element clause answer $\bar d$ with some $\bar c$,
$\bar a\bar c\sim_{\omega\cdot\beta}\bar b\bar d$, and the inductive
hypothesis gives $\bar a\bar c\ge_\beta\bar b\bar d$; the symmetric
requirement holds because $\sim$ is symmetric. Now at the level
$\omega\cdot\max(\rhole(\bar a),1)$: $\sim$-equivalence to $\bar a$ gives both
$\le_{\rhole(\bar a)}$-inclusions of types, hence (by the printed clause
defining $\rhole$, both ways) equality of $\Lom$-types, hence automorphy ---
the Scott clause at [Mon15, p.~5432]; countability of $A$ is used exactly
here. So $\rhosym(\bar a)\le\omega\cdot\max(\rhole(\bar a),1)$. \qed
\item[(T2)] $\rhole(\bar a)\le h(\rhosym(\bar a))$, $h(0)=1$,
$h(\beta+1)=h(\beta)+2$, $h(\text{limit})=\sup$. \emph{Proof.} First,
$\le_{h(\beta)}\subseteq\sim_\beta$, by induction on $\beta$, using twice per
step the flip $\le_{\gamma+1}\subseteq\ge_\gamma$ (instantiate the second
printed line at $\gamma$ with the empty extension). Base: $\bar a\le_1\bar b$
gives inclusion of $\Pin{1}$-types, and atomic formulas together with their
negations are $\Pin{1}$, so the atomic types are equal: $\bar a\sim_0\bar b$.
Successor: from $\bar a\le_{h(\beta)+2}\bar b$, the flip gives
$\bar b\le_{h(\beta)+1}\bar a$; for the forth clause, given $c$, the
$\le_{h(\beta)+1}$-condition on $\bar b$ at the extension $c$ of $\bar a$
yields $d$ with $\bar b d\ge_{h(\beta)}\bar a c$, i.e.\
$\bar a c\le_{h(\beta)}\bar b d$, and the inductive hypothesis gives
$\bar a c\sim_\beta\bar b d$; for the back clause, given $d$, the
$\le_{h(\beta)+2}$-condition on $\bar a$ yields $c$ with
$\bar a c\ge_{h(\beta)+1}\bar b d$, and one further flip lands in
$\le_{h(\beta)}$ and the inductive hypothesis as before. Limits: both sides
are intersections of the earlier stages. Now at the level
$h(\rhosym(\bar a))$: $\le$-comparability with $\bar a$ gives
$\sim_{\rhosym(\bar a)}$-equivalence, hence automorphy, hence equality of
$\Lom$-types, which is the defining condition of $\rhole$ at that level. \qed
\emph{(The two-sided form of the absorption, with $+3$ at each successor, follows by one more flip.)}
\end{itemize}
The biconditionals now follow by transferring (T1)/(T2) between the two suprema; the same sup-transfer pattern is written out in full in the proof of Proposition~(AK-$\sim$-bridge) below, with (B1)/(B2) in place of (T1)/(T2), and is summarized here clause by clause:
\begin{itemize}[leftmargin=2.4em]
\item[(floor)] the floor biconditional is the sup-transfer of (T1)/(T2) below $\lambda$ ($\omega\cdot\lambda=\lambda$, resp.\ $\lambda$ limit);
\item[(at $\lambda$)] both floors $+$ contrapositive transports;
\item[(at $\lambda+1$)] the forward direction is successor-sup attainment $+$ the (T1)-contrapositive $+$ the Nadel bound on $\AKR$ ([Mon13, p.~4]); the backward direction is the (T2)-contrapositive at the attained $\rhole=\lambda$ (lower half) $+$ the cap;
\item[(cap)] the upper bound follows in two lines from the Nadel bound on $\AKR$, (T1), and $\omega\cdot\lambda=\lambda$, and is attained at the Harrison order (\S7.H).
\end{itemize}
Two side conditions: every appeal to Karp's theorem and to the reversal is at level $\ge1$, with the base cases routed through Karp at $\alpha=1$, so that the enumeration convention for $\le_0$ is never needed; and the limit clause of \S2.4 is used only for monotonicity of $\sim$ and in the limit step of the absorption argument.
\end{proof}

\begin{rrem}[Scope of Theorem \textup{(TV-b)}]
The identification is asserted at the two values $\omega_1^A$ and $\omega_1^A+1$
only. Away from those values $\SRsym$ and $\AKR$ are distinct ranks and nothing
above compares them. The side condition used in the proof is
$\omega\cdot\lambda=\lambda$, which holds at $\lambda=\omega_1^A$ because that
ordinal is admissible; multiplicative closure above $\omega$ suffices.
\end{rrem}

\begin{rrem}[Relation to {[AK00, \S6.7]}]
Theorem~(TV-b) identifies Montalb\'an's symmetric rank $\SRsym$ [Mon15, p.~5432]
with the Ash--Knight rank $\AKR$ at the two values $\{\omega_1^A,\omega_1^A+1\}$.
It is proved above from the following statements in the literature: the
$\sim$-relations and the clause defining $\rhosym$ at [Mon15, p.~5432]; the
back-and-forth relations and Karp's theorem at [GRT, p.~6]; and the clause
defining $\rhole$ together with the Nadel bound at [Mon13, pp.~3--4]. We record
here what the nearest published statement does and does not give.

The source for $\AKR$ itself is [AK00, \S6.7, p.~98], which defines $\AKR(A)$ to
be the least ordinal greater than $\rho(\bar a)$ for all tuples $\bar a$ --- the
clause and supremum form of [Mon13, p.~3] --- so the identification asserted at
[Mon13, p.~3] is confirmed at both ends. Moreover \S6.7 states, without proof, a
value identity of the same shape as (TV-b), but for the symmetric rank
$\SRAK$ of that book:
\begin{quote}
``Neither $\mathrm{SR}$ nor $\AKR$ can take value $0$. If one of these ranks has
value $1$, or a limit ordinal, or the successor of a limit ordinal, then the
other has the same value, and the difference in value is never infinite'';
\end{quote}
instantiated at $v\in\{\omega_1^A,\omega_1^A+1\}$, this is the biconditional of
(TV-b) with $\SRAK$ in place of $\SRsym$. The corresponding upper bound
$\SRAK(A)\le\omega_1^A+1$ follows in one line from the inequality
$\mathrm{SR}\le\AKR$ stated at \S6.7 together with the Nadel bound of
[Mon13, p.~4].

This does not supply a citation for Theorem~(TV-b), because $\SRAK$ and $\SRsym$
are suprema over different families of symmetric relations: the relation
$\sim_\beta$ of [AK00] is cumulative in $\gamma$ (``for all $\gamma<\beta$'',
uniformly at successors and limits; see the footnote on p.~99, where either
tuple may be extended at any move), whereas $\rhosym$ of [Mon15] is defined at
exactly one level, with one-element extensions. The two ranks genuinely differ
in general: on $(\omega,<)$ one has $\SRAK=2$ while $\SRsym=\omega$. Whether
$\SRAK=\SRsym$ at the two values $\{\omega_1^A,\omega_1^A+1\}$ --- we refer to
this statement as $\AKbridge$ --- is likewise not available from the
literature; since the comparison with \S6.7 raises the question, we settle it
in the proposition following this remark. The proposition is consumed nowhere
in this paper: Theorem~(TV-b) is proved directly, and no proof below mentions
$\SRAK$. We note also that nothing in \S6.7 contradicts Theorem~(TV-b) or the
comparisons (T1), (T2); the value shapes listed in the sentence quoted above
positively exclude a separation for the Ash--Knight variant, and the example
$\SRAK=2\ne 3=\AKR$ given at [AK00, Example~4] falls outside that list.
\end{rrem}

\begin{rprop}[(AK-$\sim$-bridge)]
For countable $A$ and $v\in\{\omega_1^A,\omega_1^A+1\}$:
\[ \SRAK(A)=v \iff \SRsym(A)=v. \]
\emph{Floor form:} $\SRAK(A)\ge\omega_1^A \iff \SRsym(A)\ge\omega_1^A$. Both
ranks are unparameterized. The identification is asserted at the two values
only; away from them the ranks genuinely differ, by the example $(\omega,<)$
above.
\end{rprop}
\begin{proof}
Write $\lambda:=\omega_1^A$, admissible ([Mon13, p.~4]), so
$\omega\cdot\lambda=\lambda$ and $\lambda$ is a limit. Write
$\sim^{\mathrm{AK}}_\beta$ for the relations of [AK00, \S6.7] as read in the
preceding remark (level $0$: equality of open type; for $\beta>0$, cumulatively
in $\gamma<\beta$, two-sided extension of either tuple by an arbitrary finite
tuple, matched at $\sim^{\mathrm{AK}}_\gamma$; the footnote on p.~99), and
$\sim_\alpha$ for the relations of [Mon15, p.~5432] with the intersection reading
of \S2.4. Both families are decreasing in the level; by that monotonicity the
cumulative AK clause reduces to its top level ($\sim^{\mathrm{AK}}_{\gamma+1}$
is the two-sided tuple-extension condition at $\gamma$ alone), and at limits
both families are the intersections of the earlier levels. At level $0$ the two
clauses agree: open formulas are finite Boolean combinations of atomic ones, so
equality of open type is equality of atomic type. On a countable $A$ the two
tuple ranks have the same terminal condition: $\rhosym$'s clause targets
automorphy and the AK clause targets equality of $\Lom$-types, and these
coincide for tuples of a countable structure (automorphisms preserve all
formulas; conversely equality of $\Lom$-types implies automorphy --- the Scott
clause, exactly as in (T1); countability of $A$ is used exactly here). Two
comparisons, by induction on the level:
\begin{itemize}[leftmargin=2.4em]
\item[(B1)] $\sim^{\mathrm{AK}}_\alpha\subseteq\sim_\alpha$ for every $\alpha$.
The [Mon15] successor step is the length-one instance of the AK tuple clause at
the top level $\gamma=\alpha$; limits are intersections on both sides.
\item[(B2)] $\sim_{\omega\cdot\alpha}\subseteq\sim^{\mathrm{AK}}_\alpha$ for
every $\alpha$. First, a finite-block simulation: if
$\bar a\sim_{\beta+n}\bar b$ and $|\bar c|=n$, then $n$ applications of the
one-element clause produce $\bar d$ with $\bar a\bar c\sim_\beta\bar b\bar d$
(two-sidedly, by the symmetry of the successor clause). At the successor:
$\omega\cdot(\alpha+1)=\omega\cdot\alpha+\omega$, so a tuple move of length $n$
at the top level $\alpha$ is answered by monotonicity down to
$\sim_{\omega\cdot\alpha+n}$, the simulation, and the inductive hypothesis at
$\sim_{\omega\cdot\alpha}$; limits by sup-continuity of $\omega\cdot(-)$. The
factor $\omega$ is exact already at $\alpha=1$: on $(\omega,<)$ no finite level
of the one-element family is contained in $\sim^{\mathrm{AK}}_1$, which is the
source of the divergence $\SRAK=2$, $\SRsym=\omega$ recorded above; $s(\alpha)
=\omega\cdot\alpha$ is the minimal solution of
$s(\alpha)=\sup_{\gamma<\alpha}(s(\gamma)+\omega)$, the recursion of (T1).
\end{itemize}
Since the terminal conditions agree, (B1) and (B2) give, tuple by tuple, the
sandwich
\[ \rho^{\mathrm{AK}}(\bar a)\ \le\ \rhosym(\bar a)\ \le\
\omega\cdot\rho^{\mathrm{AK}}(\bar a), \]
where $\rho^{\mathrm{AK}}$ is the tuple rank of the AK family (each inequality
states that the defining condition of the left rank is met at the level the
right side names). Caps: $\SRAK(A)\le\lambda+1$ is the upper bound recorded in the remark above
($\mathrm{SR}\le\AKR$ at \S6.7 together with the Nadel bound on $\AKR$,
\S2.7), so every $\rho^{\mathrm{AK}}(\bar a)\le\lambda$; and then by the second
sandwich inequality $\rhosym(\bar a)\le\omega\cdot\rho^{\mathrm{AK}}(\bar a)
\le\omega\cdot\lambda=\lambda$, so $\SRsym(A)\le\lambda+1$ as well. Both caps
thus come from the family clauses and \S2.7 alone --- in particular the
proposition re-derives Corollary~(SR$^\sim$-Nadel) at $A$ without (T1) --- and
every tuple rank of either family is $\le\lambda$. The value
transfer: at $v=\lambda+1$ the sup is attained (a successor sup has a top
element), and an attained $\rho^{\mathrm{AK}}(\bar a)=\lambda$ forces
$\rhosym(\bar a)=\lambda$ by the first inequality and the cap, while an
attained $\rhosym(\bar a)=\lambda$ forces $\rho^{\mathrm{AK}}(\bar a)=\lambda$,
since $\rho^{\mathrm{AK}}(\bar a)<\lambda$ would give
$\rhosym(\bar a)\le\omega\cdot\rho^{\mathrm{AK}}(\bar a)<\lambda$ by
multiplicative closure. At $v=\lambda$ the sandwich maps each family's tuple
ranks into the other's below $\lambda$ ($\omega\cdot\xi<\lambda$ for
$\xi<\lambda$), preserving unboundedness in both directions --- for the AK
direction, given $\xi<\lambda$ pick $\bar a$ with
$\rhosym(\bar a)>\omega\cdot\xi$; then $\rho^{\mathrm{AK}}(\bar a)>\xi$. The
floor form is the same computation read below $\lambda$.
\end{proof}
\begin{rrem}
Composing the proposition with the value identity quoted above from
[AK00, \S6.7] (which is stated there without proof) would re-derive
Theorem~(TV-b) at the two values. We rely on this in neither direction:
Theorem~(TV-b) is proved directly above, and the proposition is proved from the
family clauses alone. Like Theorem~(TV-b) and the club form of Lemma~F, the
proposition is assembled from statements in the literature rather than cited;
see A.5. It is consumed in no proof in this paper.
\end{rrem}

\begin{rlem}[TV (the composite comparison)]
At $v\in\{\omega_1^A,\omega_1^A+1\}$: $\srS(A)=v \iff \cSR(A)=v$. \emph{Floor form:} $\srS(A)\ge\omega_1^A \Rightarrow \cSR(A)\ge\omega_1^A$.
\end{rlem}
\noindent This is the composite of TV-a, (TV-b) and TV-c; the lower-bound form
is TV-a at the two values together with Corollary~(TV-b$^-$) and TV-c(iii). It is
the only route by which any proof below passes between $\srS$ and $\cSR$, and it
is used at exactly the following places. Lemma~5.9, the identity
$\text{fiber}=W_\lambda$ of Corollary~5.10, and Lemma~AL(vii-2) use the lower-bound form
only. The clause of Lemma~S(iv) concerning $A_\lambda$, and hence Lemma~S(v), use
Theorem~(TV-b) at the limit value. The clause of Lemma~S(iv) concerning
$B_\lambda$, and hence Theorem~5.3(i) and (ii), use Theorem~(TV-b) at the
successor value. No other proof below passes between the two ranks.

\subsection*{2.6. The invariant $\omega_1^A$}
$\Sp(A):=\{X\in2^\omega:X\text{ computes a copy of }A\}$ is the spectrum of $A$; $\omega_1^A:=\min\{\omega_1^X:X\in\Sp(A)\}$ ([Mon13, p.~4]); the min is attained. Two descriptions of the same invariant will be used, and their
agreement is a (short) fact, not a convention:
\begin{rlem}[ID (the two descriptions of $\omega_1^A$)]
For countable $A$: $\omega_1^A$ equals the least ordinal $\mu$ such that some
$\Sigma_1$-admissible set of height $\mu$ contains a code of $A$.
\end{rlem}
\begin{proof}
For $X\in\Sp(A)$, $L_{\omega_1^X}[X]$ is admissible of height $\omega_1^X$ and
contains $X$; so the least such $\mu$ is $\le\omega_1^A$. Conversely, if
$\mathfrak A$ is admissible of height $\mu$ and $X\in\mathfrak A$ is a code of
$A$, then $\mu$ is admissible relative to $X$ --- this is the argument of
Lemma~CC(a) below, read over $\mathfrak A$ in place of $L(\alpha,T)$ (classical;
[Bar75, Ch.~II]) --- so $\omega_1^X\le\mu$, whence $\omega_1^A\le\mu$.
\end{proof}
\begin{rrem}[reading of {[Sac07]}'s $\omega_1$-invariants]
[Sac07] nowhere prints a definition of $\omega_1^A$ for structures in the
pages consumed here; throughout this paper every $\omega_1$-invariant imported
from [Sac07] ($\omega_1^A$, $\omega_1^B$, $\omega_1^{T,M}$) is read as the
invariant of this subsection --- in the relativized case,
$\omega_1^{T,M}:=\min\{\omega_1^X: X\in\Sp(M),\,X\ge_T t\}$ as at
[Mon13, p.~4], equivalently by Lemma~ID the least height of an admissible set
containing a code of $\langle T,M\rangle$. This reading is consistent with the
presupposition form of the one printed pin ([Sac07, p.~14]: ``$\omega_1^A=
\alpha$'' with $L(\alpha,T)$ $\Sigma_1$-admissible) and with [Mon13, p.~4]'s
rendering of the same invariant. Like the reading of ``$n$-type of a node''
(\S3.7), the flag is carried, not smoothed: it is a declared reading of
[Sac07]'s convention, consumed statement-level.
\end{rrem} Comparisons of rank at the values relevant below pass exclusively through Lemma~TV. We use repeatedly that admissible ordinals are multiples of $\omega^2$: if $L(\lambda,T)\models\mathrm{KP}$ then $\lambda$ is closed under primitive recursive ordinal functions, hence multiplicatively closed.

\subsection*{2.7. Nadel bound}
$\AKR(A)\le\omega_1^A+1$ (Nadel [Nad74]; stated at [Mon13, p.~4]). Hence $\cSR(A)\le\omega_1^A+1$, by the coincidence at the two top values in TV-c(ii). A structure is said to have \emph{high Scott rank} when $\AKR(A)\ge\omega_1^A$ ([Mon13, p.~4]); we say that $A$ \emph{attains the Nadel bound} when $\cSR(A)=\omega_1^A+1$.
In Sacks' convention likewise $\srS(A)\le\omega_1^A+1$, stated at
[Sac07, p.~5]; the four proofs below that cite this subsection for a bound on
$\srS$ use this clause and this pin.

\subsection*{2.8. Quantifier rank}
$\qr(S):=$ the least $\beta$ with $S\in\Pin{\beta}$ or $\Sin{\beta}$ ([GM23, p.~2, fn.~2]). Then $\varphi\in\Pin{\beta+1}$ for every $\beta\ge\qr(\varphi)$.

\subsection*{2.9. Codes and relative admissibility}
$t_\varphi\in 2^\omega$ is a fixed real coding $\varphi$; $\Adm(t_\varphi):=\{\lambda<\omega_1:\lambda>\omega,\ L_\lambda[t_\varphi]\models\mathrm{KP}\}$. \emph{Code-dependence caveat:} $\Adm(t_\varphi)$ is $\equiv_T$-invariant in the code but genuinely degree-dependent; codes of the same $\varphi$ can differ widely in degree, by the upward closure of the degree spectrum of $(\tc(\{\varphi\}),\in)$
($\tc$ denotes transitive closure throughout) --- a classical, external fact of
computable structure theory, taken by reference and not reproved here. (This fact is used only negatively, at Lemma~S(iii-b).) Every theorem below holds for each fixed choice of $t_\varphi$, and each existence lemma names the set of ordinals over which it applies. If $\varphi\in L^c_{\omega_1\omega}$ (the computable infinitary sentences, as in
[Mon13, \S1.1]), take $t_\varphi$ recursive; then $\Adm(t_\varphi)=\Adm$.

\subsection*{2.10. Three notions of smallness, kept distinct}
\textup{(a)} [GRT, Def.~2.2, p.~7]: $\Sin{\alpha}$-small ${}={}$ countably many $\Sin{\alpha}$-types realized (tuple level). \textup{(b)} [Mon13, Def.~1.4, p.~4]: scattered ${}={}$ at most countably many $\equiv_\alpha$-classes among the models, each $\alpha$ (model level; definitional in (H1)). \textup{(c)} [Sac07, p.~3]: Morley's notion [Mor70], combining (a) and (b) at the level of fragment types. These are not identified with one another anywhere below.

\subsection*{2.11. Admissibility relative to $T$}
Throughout, $L(\alpha,T)$ denotes the $\alpha$-th level of the constructible hierarchy relativized to $T$, with $T$ carried as an amenable predicate, as in [Sac07]. $A(T):=\{\alpha>\omega:L(\alpha,T)\text{ is }\Sigma_1\text{-admissible}\}$, in the sense used throughout [Sac07] (see p.~14 there): KP with the predicate for $T$ amenable. The relation of $A(T)$ to $\Adm(t_\varphi)$ is settled in Lemma~CC of \S4.

\subsection*{2.12. Standing hypotheses}
\textbf{(H0)} Vaught's conjecture fails. (Nothing below asserts or implies Vaught's conjecture; see \S1.3.) \textbf{(H1)} $\varphi$ is minimal in the sense of [Mon13, Def.~3.1, p.~6]. \textbf{(H2)} $\varphi$ is $\Sin{\alpha}$-small for every $\alpha<\omega_1$; this is [GRT, Prop.~2.2, p.~7]. It is consumed once, at the weak-scattering remark of \S3.7.
\begin{rrem}
We read [GRT, Prop.~2.2] under the perfect-set formulation of Vaught's conjecture: the proof given there passes from ``continuum many models'' to the conclusion that Vaught's conjecture holds, which under CH requires that formulation. Wherever only smallness at the level of models is needed, the scatteredness built into [Mon13, Def.~1.4] supplies it by definition.
\end{rrem}
\noindent\textbf{(H3)} $t_\varphi,\qr(\varphi)$ as in \S\S2.8--2.9.

\section{Definitions}

\subsection*{3.1. Crowded classes}
By Lemma~B below, for each $\beta<\omega_1$ exactly one $\equiv_\beta$-class of
models of $\varphi$ is uncountable. We call it the \emph{crowded class} at level
$\beta$ and write $\Ccl_\beta$ for it.

\subsection*{3.2. Labels}
$K_\beta:=$ the \emph{label} of $\Ccl_\beta$ in the sense of [GRT, Def.~3.1,
p.~13]: the unique member of $\Ccl_\beta$ of least categoricity rank. By
Lemma~E it exists, and $\cSR(K_\beta)=\beta$ exactly, for every
$\beta\ge\qr(\varphi)$. That $K_{\lambda+1}\in\Ccl_{\lambda+1}$ holds by
definition. (Recall from the Conventions that Proposition~3.1 of [GRT] is on
p.~12 and Definition~3.1 on p.~13.)

\subsection*{3.3. The function $f$ and the fixed-point club}
$f(\alpha):=\min\{\beta:(\forall M\models\varphi)(M\notin\Ccl_\alpha\Rightarrow
\cSR(M)<\beta)\}$, as in the proof of [GRT, Thm.~3.6, p.~14]. We put
$\Ffix:=\{\alpha:f(\alpha)\le\alpha\}$. The condition used throughout is
$f(\alpha)\le\alpha$ rather than the equality $f(\alpha)=\alpha$ appearing in
[GRT]; the difference is discussed at Lemma~F(iii).

\subsection*{3.4. Witnesses and the high split}
For a limit $\lambda$ put
\[ W_\lambda:=\{A\models\varphi:\omega_1^A=\lambda\text{ and }\cSR(A)\ge\lambda\},
\qquad \mathrm{HIGH}(\lambda):=\{A\models\varphi:\cSR(A)\ge\lambda+1\}, \]
and let $H_=(\lambda)$, $H_>(\lambda)$ be the parts of $\mathrm{HIGH}(\lambda)$
with $\omega_1^M=\lambda$ and with $\omega_1^M>\lambda$ respectively. Finally
$W_\lambda^{\mathrm{high}}:=W_\lambda\cap\{\cSR=\lambda+1\}$. Members of
$W_\lambda$ will be called \emph{witnesses at $\lambda$}, and members of
$W_\lambda^{\mathrm{high}}$ \emph{high witnesses}. All the model classes below ($W_\lambda$, $\mathrm{HIGH}(\lambda)$, the fiber of Definition~5.5, and so on) are counted up to isomorphism: ``countable'' means ``countably many isomorphism types''.

\subsection*{3.5. The two base sets}
The trichotomy of Thm.~5.1 requires only a limit
$\lambda\in\Ffix\cap(\qr(\varphi),\omega_1)$ together with a witness
$A\in W_\lambda$. Two sets of such $\lambda$ are supplied below, by two different
existence lemmas:
\[ S:=\Ffix\cap\Adm(t_\varphi)\cap(\qr(\varphi),\omega_1) \]
(supplied by Lemma~W; stationary by Thm.~5.4(i); depends on the choice of code),
and
\[ \Szero:=\Ffix\cap C_{5.3}\cap(\qr(\varphi),\omega_1) \]
(supplied by Lemma~S; a club; independent of the choice of code). Every member
of $\Szero$ lies in $A(T)$ (Lemma~S(iii-a)), hence is a limit ordinal, so the
standing requirement that $\lambda$ be a limit is automatic on $\Szero$. The
inclusion $\Szero\subseteq S$ holds when $\varphi$ has a recursive code (Lemma~CC(d)) but
fails for general codes (Lemma~S(iii-b)); the intersection $S\cap C_{5.3}$ is
stationary (Thm.~5.4(v)).

\subsection*{3.6. Four conditions on $\lambda$}
For a limit $\lambda\in\Ffix\cap(\qr(\varphi),\omega_1)$ define
\begin{align*}
x(\lambda) &:\iff \text{some }A\in W_\lambda\text{ has }\cSR(A)=\lambda;\\
y_g(\lambda) &:\iff \text{some }A\in W_\lambda^{\mathrm{high}}\text{ lies in }\Ccl_{\lambda+1};\\
y_s(\lambda) &:\iff \text{some }A\in W_\lambda^{\mathrm{high}}\text{ lies outside }\Ccl_{\lambda+1}.
\end{align*}
We say that a high witness \emph{glues} if it lies in $\Ccl_{\lambda+1}$, and
that it \emph{splits} otherwise; the subscripts $g$ and $s$ refer to these two
cases. By Theorem~5.1 a high witness that glues is isomorphic to
$K_{\lambda+1}$, while one that splits produces a second model of rank
$\lambda+1$.

For $\lambda\in\Szero$ we also set $y_B(\lambda):\iff B_\lambda\in
\Ccl_{\lambda+1}$, where $B_\lambda$ is the $\lambda$-saturated model of \S3.7.
We shall prove: $y_B(\lambda)\Rightarrow y_g(\lambda)$ and $\neg y_B(\lambda)
\Rightarrow y_s(\lambda)$, since $B_\lambda\in W_\lambda^{\mathrm{high}}$ always
(Lemma~S(iv), which uses Theorem~(TV-b)); and $y_B(\lambda)\iff B_\lambda\iso
K_{\lambda+1}\iff B_\lambda\equiv_{\lambda+1}K_{\lambda+1}$ (Lemma~A(vi)
together with Lemma~L at $\lambda+1$). Whether $y_g(\lambda)$ implies
$y_B(\lambda)$ is open; we refer to that implication as $\dq$ and return to it
in \S5.12 and \S6.

\subsection*{3.7. The hull club and the tree of nodes}
We apply the results of [Sac07] to $T:=T(\varphi)$, the theory $\{\varphi\}$
padded with the tautologies of the fragment $L_0$ generated by $\varphi$, so
that $T$ mentions every formula of $L_0$ as required by the convention at
[Sac07, p.~8]. The padding does not change the class of models:
$\Mod(T)=\Mod(\varphi)$.

\emph{The tree of nodes} ([Sac07, p.~8]): the \emph{nodes} of $\mathrm{TR}(T)$ are theories $T'$ finitarily consistent and $\omega$-complete in fragments $L_{T'}$ with $T\subseteq T'$, $L_0\subseteq L_{T'}$; level $0=$ such extensions of $T$ in $L_0$; level $\delta+1=$ such extensions of a node $S$ in $L'_S$ (the least fragment adding the conjunctions of $S$'s non-principal types; p.~9, eq.~(4.1)); a limit level $=$ unions along chains (p.~9). As stated at [Sac07, p.~8], each $T'$ has an atomic model, and the class of all such models is the class of all countable models of $T$. In particular every model of every node --- so in particular $A_\lambda$ and $B_\lambda$ below --- is a model of $\varphi$.

$C_{5.3}:=$ the club of [Sac07, Thm.~5.3, p.~14]: $C_{5.3}=\{c_\delta:\delta<\omega_1\}$, $c_\delta=\omega_1^{L(\beta_\delta,T)}$ for the transitive collapses $L(\beta_\delta,T)$ of the $\Sigma_1$-hulls ((5.2)/(5.3), pp.~14--15). $A_\lambda,B_\lambda:=$ the atomic and the $\lambda$-saturated models supplied by [Sac07, Thm.~5.3] at $\lambda\in C_{5.3}$; at $c_\delta$ the proof there takes $T_{c_\delta}$ to be the restriction of $T_{\omega_1}$ to level $c_\delta$, which is $\Delta_1^{L(c_\delta,T)}$ via the branch parameter $p$ (p.~15). We quote the definition of saturation from [Sac07, p.~14]: ``Suppose $L(\alpha,T)$ is $\Sigma_1$ admissible, $A$ is a countable model of $T$, and $\omega_1^A=\alpha$. According to (2.6), $A$ is a homogeneous model of $T^A_\alpha$; $A$ is said to be $\alpha$-saturated if every $n$-type ($n\ge1$) of $T^A_\alpha$ is realized in $A$.''

We shall also use the following, with numbering as in the December 2004 preprint; more generally, bare references of the form ``p.~$n$'', ``(2.6)'', ``(4.1)'' or ``clause~($k$)'' occurring inside proofs that work with the tree of nodes or the canonical tower are to [Sac07] (this paper has no numbered equations of its own). \textup{(a)} $\omega$-completeness, [Sac07, p.~8]: (1)~completeness---every $F\in L'$ has $F\in T'$ or $\neg F\in T'$; (2)~the disjunction property---$\bigvee_i F_i\in T'\Rightarrow F_i\in T'$ for some $i$; nodes are complete theories in their fragments. \textup{(b)} the tree analysis of a single structure, [Sac07, p.~10]: $T(\delta,A)$ (clauses (1)--(5)) with $L_{T(\delta+1,A)}=L'_{T(\delta,A)}$ (eq.~(4.1)) and the tree rank $\tr(A)$ (4.3); Prop.~4.5 (p.~10): $\tr(A)\le\srS(A)$, whose proof establishes $L^A_\delta\subseteq L_{T(\delta,A)}$ by induction; and Prop.~4.6 (p.~10). \textup{(c)} [Sac07, p.~15]: $B$ is constructed as a model of $T_{c_\delta}$ realizing all the types in $N$, where $N$ is the set of non-principal types of the node $T_{c_\delta}$, and $\omega_1^B=c_\delta$. Note that the construction on p.~15 is carried out in terms of the node theory, whereas the statement on p.~14 is in terms of the canonical tower of $B$ itself; Lemma~AL(iv) below supplies the passage between the two.

\begin{rrem}[on $n$-types of a node]
No definition of ``$n$-type of $S$'' is given in [Sac07]. Throughout this paper an $n$-type of a node $S$ means a maximal set of $L_S(\bar x)$-formulas finitarily consistent with $S$. Every use made of the notion below --- complete types of tuples realized in a model, members of the set $N$, principal types --- is insensitive to the choice among the reasonable alternatives.
\end{rrem}

\begin{rrem}[on weak scattering]
Weak scattering in the sense of [Sac07] enters at exactly two places: here, where the minimality hypothesis (H1) is used to obtain clauses (a) and (b) of that notion by way of [Sac07, Prop.~4.3]; and in the ($\Leftarrow$) direction of Lemma~JP, where the localized form at [Sac07, p.~23] is obtained through Cor.~3.2 there. The passage from (H1) to those clauses is an inference from [Sac07, Prop.~4.3] and not a quotation; we note it again at each of the two places, and at step (C1) of \S5.11, which uses the same inference.

The type-count content behind that inference is on the page: by (H2), only
countably many $\Sin{\alpha}$-types are realized in models of $\varphi$ at
each $\alpha<\omega_1$, and by Lemma~AL(vii-1) every tree fragment
$L_{T_\delta}$ lies within quantifier rank $c(\delta)<\omega_1$, so the
\emph{realized} type-sets at every node are countable. The residual content
of the [Sac07, Prop.~4.3] crossing --- consumed as stated there and not
re-derived here --- is the passage from this realized count to the type-sets
counted by the printed clauses. (H2) is consumed at this remark and nowhere
else.
\end{rrem}

\section{Lemmas}

\begin{rrem}
GRT's numbering collides at $3.1$: Proposition~3.1 is on p.~12, Definition~3.1 on p.~13, separate counters; every GRT-$3.1$ citation below carries its page.
\end{rrem}

\begin{rlem}[B (crowded class at every level)]
For every $\beta<\omega_1$: \textup{(i)} exactly one $\equiv_\beta$-class $\Ccl_\beta$ of models of $\varphi$ is uncountable, and every other class is countable; \textup{(ii)} $\beta\le\beta'\Rightarrow\Ccl_{\beta'}\subseteq\Ccl_\beta$; \textup{(iii)} at a limit $\lambda$, $\Ccl_\lambda=\bigcap_{\beta<\lambda}\Ccl_\beta$; \textup{(iv)} $\{M\models\varphi:M\notin\Ccl_\beta\}/{\iso}$ is countable, so $f$ is well-defined and $f(\beta)<\omega_1$.
\end{rlem}
\noindent This is due to Gonzalez--Rossegger--Turetsky, [GRT, Thm.~3.6, proof
paragraph~1, p.~14]; the same dichotomy is in the proof of [Mon13, Lemma~3.3,
p.~7]. We write out the argument.
\begin{proof}
There are countably many $\equiv_\beta$-classes ((H1), model-level; [GRT, p.~14, paragraph~1] derives the count from counterexample-hood); $\aleph_1$ models force some class uncountable (pigeonhole; implicit in both sources, made explicit here). Uniqueness: two uncountable classes are disjoint $\Pin{\beta+1}$-axiomatizable subclasses ([GRT, Lemma~3.3, p.~13]), so an axiomatization $\psi$ of one splits $\varphi$ into two uncountable halves, contradicting (H1) ([GRT, p.~14, paragraph~1] mechanism; [Mon13, p.~7] same dichotomy). Part~(ii) is monotonicity; (iii) is [GRT, p.~14, paragraph~1] as stated; (iv) holds because there are countably many classes, each non-crowded one countable, and countable structures have countable $\cSR$.
\end{proof}

\begin{rlem}[L (label uniqueness, mixed ranks)]
Fix $\beta$ and an $\equiv_\beta$-class $C$. Then $C$ contains at most one model of $\cSR\le\beta$ up to ${\iso}$; and if $K\in C$ has $\cSR(K)=\gamma<\beta$ then $C=\{K\}/{\iso}$.
\end{rlem}
\noindent This is [GRT, Prop.~3.1 and Cor.~3.2, p.~12], together with Karp's
theorem [Kar65] as stated at [GRT, Thm.~1.1, p.~6].
\begin{proof}
Suppose $K\equiv_\gamma K'$ with $\cSR(K)\le\cSR(K')=\gamma\le\beta$; [GRT, Prop.~3.1] at $\gamma$ gives $K\ge_{\gamma+1}K'$; Karp moves $K'$'s $\Pin{\gamma+1}$ Scott sentence into $K$; so $K\iso K'$. Equal ranks is [GRT, Cor.~3.2] as stated. Second clause: $M\equiv_\beta K\Rightarrow M\equiv_{\gamma+1}K\Rightarrow M\models$ the Scott sentence of $K\Rightarrow M\iso K$.
\end{proof}

\begin{rlem}[E (label existence with exact rank)]
For every $\beta\ge\qr(\varphi)$: $\Ccl_\beta$ is labeled ([GRT, Def.~3.1, p.~13]); $\cSR(K_\beta)=\beta$ exactly; and $K_\beta$ is, up to ${\iso}$, the unique model of $\cSR\le\beta$ in $\Ccl_\beta$.
\end{rlem}
\noindent We assemble this from the sources cited at each step of the proof.
\begin{proof}
\emph{(1) Axiomatization.} (H2) [GRT, Prop.~2.2, p.~7] gives $\varphi$ $\Sin{\beta}$-small. [GRT, Lemma~3.3, p.~13] at any $M\in\Ccl_\beta$: $\Ccl_\beta=E(M,\beta)=\{N\models\varphi:N\equiv_\beta M\}$ is $\Pin{\beta+1}$, via the $\bigwedge_{i>0}\neg\psi_i$ description ($\psi_i\in\Pin{\beta}\cup\Sin{\beta}$ separating the $i$-th of the countably many $\equiv_\beta$-classes from $\Ccl_\beta$; countability from smallness/(H1)). Let $\theta_\beta$ be that $\Pin{\beta+1}$ sentence; put $T_\beta:=\varphi\wedge\theta_\beta$.
\emph{(2) The $\qr(\varphi)$-tail (a hypothesis left implicit at [GRT, Cor.~3.4], stated here at the point of use).} [GRT, Cor.~3.4, p.~13] derives labeledness ``directly from Lemma~3.3 and Lemma~2.3''; Lemma~2.3's hypothesis is on the theory's complexity, and the one-line proof is silent on $\varphi$'s own. Here $\varphi\in\Pin{\beta+1}\iff\beta\ge\qr(\varphi)$ (\S2.8), whence $T_\beta\in\Pin{\beta+1}$ (finite $\wedge$) exactly on the stated tail; this is the source of the $\qr(\varphi)$-tail in every base set.
\emph{(3) Existence.} $T_\beta$ is consistent ($\Ccl_\beta\ne\emptyset$---indeed uncountable, Lemma~B(i)); $\Sin{\beta}$-small ($\Mod(T_\beta)\subseteq\Mod(\varphi)$, and $\Sin{\beta}$-smallness is inherited by subclasses: the set of realized $\Sin{\beta}$-types, in the sense of [GRT, Def.~2.2], can only shrink); $T_\beta\in\Pin{\beta+1}$ by~(2). [GRT, Lemma~2.3, p.~7]: some $K\models T_\beta$ has $\cSR(K)\le\beta$. $K\models\theta_\beta$ places $K\in\Ccl_\beta$: the class is labeled; put $K_\beta:=K$.
\emph{(4) Uniqueness.} Lemma~L at $\Ccl_\beta$ ([GRT, Prop.~3.1 $+$ Cor.~3.2, p.~12]; Karp): at most one model of $\cSR\le\beta$ up to ${\iso}$.
\emph{(5) Exactness (both halves).} ``$\le\beta$'' is step~(3). ``$\ge\beta$'': if $\cSR(K_\beta)=\gamma<\beta$, then every $N\in\Ccl_\beta$ has $N\equiv_\beta K_\beta\Rightarrow N\equiv_{\gamma+1}K_\beta$ (monotonicity, $\gamma+1\le\beta$) $\Rightarrow N\models$ the $\Pin{\gamma+1}$ Scott sentence of $K_\beta$ (Karp) $\Rightarrow N\iso K_\beta$---collapsing $\Ccl_\beta$ to a single isomorphism type against Lemma~B(i).
\end{proof}

\begin{rlem}[F (the fixed-point club, and the connection with {[Mon13]})]
With $f(\alpha):=\min\{\beta:(\forall M\models\varphi)(M\notin\Ccl_\alpha\Rightarrow\cSR(M)<\beta)\}$ ([GRT, Thm.~3.6 proof paragraph~2, p.~14], with $\le$ in place of the $=$ of the fixed-point case there; see the proof of (iii)) and $\Ffix:=\{\alpha<\omega_1:f(\alpha)\le\alpha\}$: \textup{(i)} $f$ is well-defined and monotone; \textup{(ii)} $\Ffix$ is closed; \textup{(iii)} $\Ffix$ is unbounded---hence a club; \textup{(iv)} for $\alpha\in\Ffix$ with $\alpha\ge\qr(\varphi)$: $\{M\models\varphi:\cSR(M)\ge\alpha\}=\Ccl_\alpha=\{K_\alpha\}/{\iso}\sqcup\{M\in\Ccl_\alpha:\cSR(M)>\alpha\}$; \textup{(v)} at a limit $\lambda$: $\lambda\in\Ffix\iff\lambda\in C$, where $C:=\{\alpha:\text{any two models of }\varphi\text{ of }\AKR\text{-rank}\ge\alpha\text{ are }\Sin{<\alpha}\text{-elementarily equivalent}\}$---the property set of [Mon13, Lemma~3.3, p.~6], which states that $C$ contains a club.
\end{rlem}
\noindent The proof is routine; we give it in full, since the club form of
$\Ffix$ is not stated in [GRT] (see the remark inside part~(iii)).
\begin{proof}
\emph{(i)} Well-defined: $\{M\models\varphi:M\notin\Ccl_\alpha\}/{\iso}$ is countable (Lemma~B(iv)); countable structures have countable $\cSR$; $f(\alpha)=\sup\{\cSR(M)+1:M\notin\Ccl_\alpha\}<\omega_1$. Monotone: $\alpha\le\alpha'\Rightarrow\Ccl_{\alpha'}\subseteq\Ccl_\alpha$ (B(ii)) $\Rightarrow$ the outside set grows $\Rightarrow f(\alpha)\le f(\alpha')$.
\emph{(ii)} Closed: $\lambda:=\sup_i\alpha_i$ with $\alpha_i\in\Ffix$ increasing; if attained, trivial; else $\lambda$ is a limit and $M\notin\Ccl_\lambda=\bigcap_{\beta<\lambda}\Ccl_\beta$ (B(iii)) gives $M\notin\Ccl_\beta$ for some $\beta<\lambda$, hence $M\notin\Ccl_{\alpha_i}$ for any $\alpha_i\ge\beta$ (B(ii)); then $\cSR(M)<f(\alpha_i)\le\alpha_i<\lambda$. So $f(\lambda)\le\lambda$.
\emph{(iii) Unbounded.} Given $\alpha$: if $f(\alpha)\le\alpha$ then $\alpha\in\Ffix$. Else iterate; either the orbit reaches $f^{k+1}(\alpha)\le f^k(\alpha)$---then $\beta:=f^k(\alpha)\in\Ffix$ with $\beta>\alpha$---or it is strictly increasing and $\beta:=\sup_n f^n(\alpha)$ satisfies $f(\beta)\le\beta$ by the sup-case computation ([GRT, p.~14, paragraph~2]: every $N\notin\Ccl_\beta$ exits some $\Ccl_{f^n(\alpha)}$, so $\cSR(N)<f^{n+1}(\alpha)<\beta$). The fixed-point case stated in [GRT] reads $f(\alpha)=\alpha$; the operative condition throughout this paper is the weaker $f(\alpha)\le\alpha$. We note explicitly: [GRT] states and proves only the per-$\alpha$ conclusion of Theorem~3.6---$\aleph_1$-many values $\{G(\alpha)\}$ (in the notation there) at which exactly one model lives---and \emph{no} club is stated or claimed in [GRT]. The club form of $\Ffix$ is the routine closure strengthening (ii)$+$(iii), assembled here from the ingredients stated in paragraphs~1--2 of the proof there (Lemma~B(ii)(iii)(iv) $+$ the sup-case computation).
\emph{(iv)} ``$\subseteq$'' of $\{\cSR\ge\alpha\}=\Ccl_\alpha$: $M\notin\Ccl_\alpha\Rightarrow\cSR(M)<f(\alpha)\le\alpha$. ``$\supseteq$'': no member of $\Ccl_\alpha$ has $\cSR<\alpha$---else its $\Pin{\gamma+1}$ Scott sentence ($\gamma+1\le\alpha$) transfers across $\equiv_\alpha$ to all of $\Ccl_\alpha$ (Karp), collapsing the class against B(i). (The stated ``by Corollary~3.2, $C_\alpha$ has at most one model $K$ with $\mathrm{SR}(K)\le\alpha$'' uses this mixed-rank $\le$-version; Cor.~3.2 as stated is equal-ranks---the repair routes through Prop.~3.1 $+$ Scott-sentence transfer $=$ Lemma~L.) Decomposition: Lemma~E at $\alpha\ge\qr(\varphi)$ supplies $K_\alpha$ with $\cSR=\alpha$ exactly, unique at rank $\le\alpha$ (Lemma~L); the remainder of $\Ccl_\alpha$ has $\cSR>\alpha$.
\emph{(v)} ($\Leftarrow$) $\lambda\in C$, $\lambda$ limit. Per-rank counting: models of $\cSR=\beta$ inject into the $\le\aleph_0$ many $\equiv_\beta$-classes ((H1), model-level), at most one per class (Cor.~3.2); so $\le\aleph_0$ models of $\cSR<\lambda$, and $\aleph_1$-many models have $\cSR\ge\lambda$, equivalently $\AKR\ge\lambda$. (Here $\AKR$ and $\cSR$ differ by at most $1$ and coincide at limit values, by TV-c(ii); so at a limit $\lambda$ the condition ``$\ge\lambda$'' is the same for both ranks. Only TV-c is used, so this is not a rank transport.) By $C$-membership these are pairwise $\Sin{<\lambda}$-elementarily equivalent $=$ pairwise $\equiv_\lambda$ (limit-$\equiv$, \S2.1); the $\cSR\ge\lambda$ mass is one $\equiv_\lambda$-class, uncountable, hence $=\Ccl_\lambda$ (B(i)); every $M\notin\Ccl_\lambda$ has $\cSR<\lambda$: $f(\lambda)\le\lambda$. ($\Rightarrow$) $\lambda\in\Ffix$ limit: $\{\cSR\ge\lambda\}=\Ccl_\lambda$ [``$\subseteq$'' from $f(\lambda)\le\lambda$; ``$\supseteq$'' from (iv)'s collapse argument, which needs no label and no tail]; so any two models of $\AKR\ge\lambda$ ($=\cSR\ge\lambda$ at the limit, as just noted) lie in one $\equiv_\lambda$-class, hence are $\equiv_\lambda$, hence $\Sin{<\lambda}$-elementarily equivalent: $\lambda\in C$.
\end{proof}
\noindent\emph{Remark.} Part (iv) asserts uniqueness at the ranks in $\Ffix$ only, which is the conclusion of [GRT, Thm.~3.6] itself; it says nothing about every rank, and the caveat at [GRT, p.~12] already places failures of uniqueness at $\omega_1$-many ranks. Nothing here concerns $\SRp$. So neither limb of \S1.3 is in play.

\begin{rlem}[CC (coding comparison)]
Put $\rho_\varphi:=\rank(\tc(\{\varphi\}))<\omega_1$ and $\gamma_\varphi:=\omega\cdot(\rho_\varphi+\omega)$. Then:
\textup{(a)} if $\alpha>\omega$ is a limit, $L(\alpha,T)\models\mathrm{KP}$ ($T$ amenable), and $x\in L(\alpha,T)$ is a real, then $L_\alpha[x]\models\mathrm{KP}$, i.e.\ $\alpha\in\Adm(x)$;
\textup{(b)} for every code $t_\varphi$: $T\in L_{\gamma_\varphi}[t_\varphi]$, and for every limit $\alpha>\gamma_\varphi$ with $L_\alpha[t_\varphi]\models\mathrm{KP}$ also $L(\alpha,T)\models\mathrm{KP}$; hence $\Adm(t_\varphi)\setminus\gamma_\varphi\subseteq A(T)$ for every code, and $S\setminus\gamma_\varphi\subseteq\Ffix\cap A(T)\cap(\qr(\varphi),\omega_1)$;
\textup{(c)} if some code $t_\varphi$ of $\varphi$ lies in $L(c_0,T)$ ($c_0=\min C_{5.3}$), then $C_{5.3}\subseteq\Adm(t_\varphi)$ for that code;
\textup{(d)} if $\varphi\in L^c_{\omega_1\omega}$ and $t_\varphi$ is recursive, then $A(T)\subseteq\Adm(t_\varphi)=\Adm$ and, with Lemma~S(iii-a), $C_{5.3}\subseteq\Adm$.
\end{rlem}
\noindent Neither (a) nor (b) is stated in the sources cited elsewhere in this paper, so we give the proofs in full; they are classical admissibility theory (for the background see [Bar75, Ch.~I \S\S6--8; Ch.~II \S\S1, 5--6]; no numbered external statement is used). For the record, the decode recursion
behind the bound $\gamma_\varphi$ in (b) runs as follows: $t_\varphi$ codes
$(\tc(\{\varphi\}),\in)$ on a field $\subseteq\omega$; define
$F_\xi:=\{(n,x_n):\mathrm{rk}_{E_t}(n)<\xi\}$ by recursion, where
$x_n:=F_\xi''\{m:m\,E_t\,n\}$; each successor step is a $\Sigma_0$-separation
and a union over sets already present, costing finitely many levels of
$L[t_\varphi]$, and each limit step is a union, so by induction
$F_\xi\in L_{\omega\cdot\xi+\omega}[t_\varphi]$; taking $\xi=\rho_\varphi+1$
recovers $\tc(\{\varphi\})$, hence $\varphi$ and $T\cap L_{(\cdot)}$, inside
$L_{\omega\cdot(\rho_\varphi+\omega)}[t_\varphi]=L_{\gamma_\varphi}[t_\varphi]$.
The antecedent of (c) is open (Question~6 of \S6). The transfer engine for (a) and (b) is the following.

\smallskip
\noindent\textbf{(CC-core) (the transfer engine).} Let $A$ be a transitive set with $A\models\mathrm{KP}$, of limit height $\alpha:=A\cap\mathrm{Ord}>\omega$, and let $s\in A$. Then the $s$-relativized constructible hierarchy up to $\alpha$ is $\Sigma_1$-total in $A$, each level is an element of $A$, and $L_\alpha[s]:=\bigcup_{\beta<\alpha}L_\beta[s]\models\mathrm{KP}$.
\begin{proof}
\emph{(Totality.)} The successor clause $\ell\mapsto\Def(\ell)$ is $\Delta_1$ in KP (set-satisfaction is $\Delta_1$); by $\Sigma_1$-recursion (a KP theorem scheme) $\beta\mapsto L_\beta[s]$ is $\Sigma_1$-definable and total on $\alpha$ in $A$, each $L_\beta[s]\in A$; also $L_\beta[s]\in L_{\beta+1}[s]$, so levels are elements of the union. $L_\alpha[s]$ is transitive. \emph{(Extensionality, foundation)} are inherited from transitivity in $V$. \emph{(Pair, union):} for $a,b\in L_\beta[s]$, $\{a,b\}$ and $\bigcup a$ are definable over $L_\beta[s]$, so lie in $L_{\beta+1}[s]$. \emph{(Infinity):} $\omega\in L_{\omega+2}[s]$ and $\alpha>\omega$. \emph{($\Delta_0$-separation):} $a,\bar p\in L_\beta[s]$, $\psi\in\Delta_0$: $\{v\in a:\psi(v,\bar p)\}=\{v\in a:\psi^{L_\beta[s]}(v,\bar p)\}$ ($\Delta_0$-absoluteness between transitive sets) is definable over $L_\beta[s]$, so lies in $L_{\beta+1}[s]$. \emph{($\Sigma_1$-collection---the essential clause):} let $a\in L_\alpha[s]$ and $\phi(u,v)=\exists w\,\delta_0(u,v,w)$, $\delta_0\in\Delta_0$, parameters in $L_\alpha[s]$, with $L_\alpha[s]\models\forall u\in a\,\exists v\,\phi(u,v)$. $\Sigma_1$-persistence along the transitive end-extension chain gives, for $u,v$ in the union: $L_\alpha[s]\models\phi(u,v)\iff\exists\gamma<\alpha\,(v\in L_\gamma[s]\wedge L_\gamma[s]\models\phi(u,v))$. Since the level sequence and set-satisfaction are $\Sigma_1/\Delta_1$ in $A$, the matrix is $\Sigma_1$ over $A$; $A\models\Sigma_1$-collection bounds the levels by some $\beta^*<\alpha$ uniformly in $u\in a$; all witnesses lie in $L_{\beta^*}[s]\in L_\alpha[s]$, and upward persistence returns the facts.
\end{proof}
\noindent\emph{Proof of \textup{(a)}.} Immediate from (CC-core) at $A:=L(\alpha,T)$, $s:=x$.

\noindent\emph{Proof of \textup{(b)}.} \emph{(Decode, KP-free.)} $t_\varphi$ codes a well-founded extensional $E\subseteq\omega\times\omega$ with a distinguished point whose Mostowski collapse is $(\tc(\{\varphi\}),\varphi)$; $E\in L_{\omega+2}[t_\varphi]$. Collapse by $E$-rank: the stage-$\gamma$ partial collapse is definable over any level containing the stage-$<\gamma$ map together with $E$ and the (definable) $E$-rank data; each stage costs finitely many levels, the recursion has length $\le\rho_\varphi+1$, and $\omega$ levels of bookkeeping per stage more than suffice: $\tc(\{\varphi\}),\varphi\in L_{\gamma'}[t_\varphi]$ for some $\gamma'<\gamma_\varphi$. The padded theory $T=\{\varphi\}\cup\Taut(L_0)$ (\S3.7): $L_0=\bigcup_n S_n$ with $S_0=$ the subformulas of $\varphi$ and $S_{n+1}$ the one-step finitary closure, each $S_n$ definable over the previous level---$\omega$ more levels; per-sentence tautologyhood over finitely many maximal non-Boolean constituents is a finite truth-table check once $\Sent(L_0)$ is present. So $T\in L_{\gamma_\varphi}[t_\varphi]$. \emph{(Transfer.)} For a limit $\alpha>\gamma_\varphi$ with $L_\alpha[t_\varphi]\models\mathrm{KP}$: (CC-core) at $A:=L_\alpha[t_\varphi]$, $s:=T$---with the hierarchy clause now the $L(\cdot,T)$-operator, the identical $\Sigma_1$-recursion/$\Sigma_1$-collection argument gives $L(\alpha,T)\models\mathrm{KP}$.

\noindent\emph{Proof of \textup{(c)}.} $t_\varphi\in L(c_0,T)\subseteq L(c_\delta,T)$ for all $\delta$; Lemma~S(iii-a) gives $L(c_\delta,T)\models\mathrm{KP}$; apply (a). Whether the antecedent holds is open: it amounts to asking whether the hull construction of [Sac07, \S5] places $\tc(\{\varphi\})$ below $c_0$, which the text there does not settle. We record it as a property the construction may or may not have, and do not assume it.

\noindent\emph{Proof of \textup{(d)}.} $t_\varphi\in L_{\omega+1}\subseteq L(\alpha,T)$ for every $\alpha>\omega$, so (a) and (b) apply in both directions: $A(T)\subseteq\Adm(t_\varphi)=\Adm$; with Lemma~S(iii-a), $C_{5.3}\subseteq\Adm$.

\begin{rlem}[W (Gandy supply, relativized)]
For every $\lambda\in\Adm(t_\varphi)$: $W_\lambda\ne\emptyset$. Every $A\in W_\lambda$ has $\cSR(A)\in\{\lambda,\lambda+1\}$; and $W_\lambda$ is countable.
\end{rlem}
\noindent The unrelativized statement is [Mon13, Lemma~3.4, p.~7]; the
relativization is routine and follows the pattern of [Mon13, Cor.~3.5].
\begin{proof}
The cited lemma is stated in [Mon13] for $L^c$ sentences and plain admissibles; for general $\varphi$ relativize to $t_\varphi$ ($\lambda=\omega_1^X$, $X\ge_T t_\varphi$, Sacks' characterization as quoted at [Mon13, p.~4], relativized---routine); the Gandy-basis proof runs over $L^{c,X}$; the conclusion pins $\omega_1^A=\lambda$, $\AKR(A)\in\{\lambda,\lambda+1\}$; transport to $\cSR$ at these values is the TV-c side (limit value: TV-c(ii); successor: the same clause); $\srS$ is not involved. The $\{\lambda,\lambda+1\}$ pin for \emph{every} member of $W_\lambda$ is the Nadel bound of \S2.7 $+$ the definition (no admissibility used there); countability is the Lemma~L injection. This is the only lemma below that uses $\lambda\in\Adm(t_\varphi)$, and hence the only one restricted to the set $S$.
\end{proof}

\begin{rlem}[S (the models of {[Sac07, Thm.~5.3]} on the hull club)]
\textup{(i)} \textup{([Sac07, Prop.~5.2, p.~14])} $T(\varphi)$ admits a $\Delta_1^{L(\omega_1,T)}$ branch $T_{\omega_1}$ of $\mathrm{TR}(T)$ such that for all countable $\beta$, $T_\beta$ has an atomic model of $\srS$ exactly $\beta$. \textup{(ii)} \textup{([Sac07, Thm.~5.3, p.~14])} There are such a branch and a club $C_{5.3}\subseteq\omega_1$ with, for all $\lambda\in C_{5.3}$: an atomic $A_\lambda\models T_\lambda$, $\srS(A_\lambda)=\lambda$; a $\lambda$-saturated $B_\lambda\models T_\lambda$, $\srS(B_\lambda)=\lambda+1$; the $A$'s an expanding chain, $A_\beta\subset A_\gamma$ elementary w.r.t.\ the language of $T_\beta$. Clauses (i) and (ii) are consumed
statement-level throughout: the branch construction of [Sac07, Prop.~5.2 and
Thm.~5.3] is cited, not reproved here. The pin $\omega_1^{B_\lambda}=\lambda$ is statement-level (p.~14's $\alpha$-saturation definition presupposes $\omega_1^A=\alpha$ with $L(\alpha,T)$ $\Sigma_1$-admissible) and explicit in the proof (p.~15: ``$\omega_1^B=c_\delta$''). The pin is read
at the invariant of \S2.6, per the reading remark there.
\textup{(iii-a)} $C_{5.3}\subseteq A(T)$.
\textup{(iii-b)} For any code $t$ of $\varphi$ with $\omega_1^t>c_0$ (and such codes exist), $C_{5.3}\not\subseteq\Adm(t)$; the inclusion does hold for any code $t_\varphi\in L(c_0,T)$ (Lemma~CC(c), whose antecedent is open) and when $\varphi$ has a recursive code (Lemma~CC(d)).
\textup{(iii-c)} $\Szero:=\Ffix\cap C_{5.3}\cap(\qr(\varphi),\omega_1)$ is a club.
\textup{(iv)} For $\lambda\in\Szero$: $\cSR(A_\lambda)=\lambda$, and $B_\lambda\in W_\lambda^{\mathrm{high}}$ ($\cSR(B_\lambda)=\lambda+1$, $\omega_1^{B_\lambda}=\lambda$); the clause about $A_\lambda$ uses Theorem~\textup{(TV-b)} at the limit value, the clause about $B_\lambda$ at the successor value.
\textup{(v)} $A_\lambda\iso K_\lambda$ for every $\lambda\in\Szero$, and the labels $\{K_\lambda:\lambda\in\Szero\}$ inherit the elementary-chain structure of [Sac07, Thm.~5.3] relative to the branch languages.
\end{rlem}
\noindent\emph{Interface note for (i):} the proof is ``By Proposition~4.8''; [Sac07, Prop.~4.8, p.~11] requires $L(\alpha,T)$ $\Sigma_2$-admissible, $T$ scattered below $\alpha$, models of arbitrarily high $\mathrm{SR}<\alpha$---applied at $\alpha=\omega_1$, where $\Sigma_2$-admissibility of $L(\omega_1,T)$ is automatic by regularity and scattered-below-$\omega_1$ is (H1) via [Sac07, Prop.~4.3, p.~9]; this is why [Sac07, Prop.~5.2] carries no admissibility hypothesis.

\smallskip
\noindent\emph{Proof of \textup{(iii-a)}.} Two independent grounds. First,
$B_\lambda$ is $\lambda$-saturated, and $\alpha$-saturation is defined at
[Sac07, p.~14] only under the standing supposition that $L(\alpha,T)$ is
$\Sigma_1$-admissible. Second, [Sac07, p.~15] states: ``the structure $L[c_\delta,T;T_{c_\delta},N]\ldots$ is $\Sigma_1$ admissible because no subset of $c_\delta$ in $L(\beta_\delta,T)$ can define a counting of $\omega_1^{L(\beta_\delta,T)}$.'' $\Sigma_1$-admissibility of the expansion implies that of the reduct $L(c_\delta,T)$ (one line). Hence $c_\delta\in A(T)$ for every $\delta$.

\noindent\emph{Proof of \textup{(iii-b)}.} For a code $t$ with $\omega_1^t>c_0$: no ordinal in $(\omega,\omega_1^t)$ is $t$-admissible, so $c_0\in C_{5.3}\setminus\Adm(t)$. Codes of $\varphi$ of arbitrarily high degree exist, by the upward closure
recalled in \S2.9, and nothing bounds the degree of $t_\varphi$. The reason the inclusion cannot hold
in general: any real code $X$ of a countable structure of ordinal height $\beta$ yields an $X$-arithmetic well-ordering of $\omega$ of type $\beta$, so $\omega_1^X>\beta$---the internal $\omega_1$ of the hull collapse and the $\omega_1$ of a code of it are never equal.

\noindent\emph{Proof of \textup{(iii-c)}.} $\Ffix$ is a club (Lemma~F), $C_{5.3}$ is a club ([Sac07, p.~15]), and the tail $(\qr(\varphi),\omega_1)$ is a club. By (iii-b), the inclusion $\Szero\subseteq S$ fails for general codes; see Thm.~5.4(v).

\smallskip
\noindent Parts (i) and (ii) are quoted from [Sac07, Prop.~5.2 and Thm.~5.3,
p.~14], with proof details from p.~15 and the description of the tree from p.~8.
Part (iii) is proved in the three clauses just given. Parts (iv) and (v) are
proved below, using Theorem~(TV-b) at the two values indicated.
\begin{proof}[Proof of (iv), (v)]
Every model of a branch node models $\varphi$---definitional from statement (p.~8; padding note \S3.7). $\lambda\in C_{5.3}\Rightarrow\lambda\in A(T)$ (iii-a) $\Rightarrow L(\lambda,T)\models\mathrm{KP}\Rightarrow\lambda$ is a multiple of $\omega^2$ (\S2.6); [Mon15, p.~5432]: $\srS=\SRsym$ at multiples of $\omega^2$ (TV-a(i)), so $\SRsym(A_\lambda)=\lambda$; Theorem~(TV-b) at the limit value gives $\AKR(A_\lambda)=\lambda$; [Mon15, p.~5433] (TV-c(ii)): $\AKR=\cSR$ at limits, so $\cSR(A_\lambda)=\lambda$. For $B_\lambda$: $\srS(B_\lambda)=\lambda+1=\omega_1^{B_\lambda}+1$, the top value; TV-a(ii) gives $\SRsym(B_\lambda)=\lambda+1$; Theorem~(TV-b) at the successor value gives $\AKR(B_\lambda)=\lambda+1$; TV-c(ii) gives $\cSR(B_\lambda)=\lambda+1$; with $\omega_1^{B_\lambda}=\lambda$, $B_\lambda\in W_\lambda^{\mathrm{high}}$. For (v): $\cSR(A_\lambda)=\lambda$ and $\lambda\in\Ffix$ put $A_\lambda\in\Ccl_\lambda$ (Lemma~F(iv)); Lemma~L makes it the rank-$\lambda$ member: $A_\lambda\iso K_\lambda$; the chain clause is (ii)'s, transported along the isomorphisms.
\end{proof}

\begin{rlem}[A (assembly)]
For a limit $\lambda\in\Ffix\cap(\qr(\varphi),\omega_1)$: \textup{(i-a)} $W_\lambda$ is countable; \textup{(i-b)} $W_\lambda\ne\emptyset$ whenever $\lambda\in\Adm(t_\varphi)$ (Lemma~W) or $\lambda\in C_{5.3}$ (Lemma~S(iv)); \textup{(ii)} $M\in\mathrm{HIGH}(\lambda)\Rightarrow\omega_1^M\ge\lambda$; \textup{(iii)} $H_=(\lambda)$ countable, $W_\lambda^{\mathrm{high}}=W_\lambda\cap H_=(\lambda)$, $H_>(\lambda)$ uncountable; \textup{(iv)} $\Ccl_\lambda=\{K_\lambda\}/{\iso}\sqcup\mathrm{HIGH}(\lambda)$ and $\Ccl_{\lambda+1}\subseteq\mathrm{HIGH}(\lambda)$; \textup{(v)} $\{K_\lambda\}$ is its own $\equiv_{\lambda+1}$-cell; \textup{(vi)} $A\in\Ccl_{\lambda+1}\iff A\equiv_{\lambda+1}K_{\lambda+1}$.
\end{rlem}
\noindent All parts are routine consequences of Lemmas B, E, L, W and S.
\begin{proof}
(ii) By the Nadel bound (\S2.7), $\cSR(M)\le\omega_1^M+1$; so $\cSR(M)\ge\lambda+1$ forces $\omega_1^M\ge\lambda$. (iv) $\Ccl_\lambda=\{K_\lambda\}/{\iso}\sqcup\mathrm{HIGH}(\lambda)$: Lemma~F(iv) at $\lambda$ gives $\{M:\cSR(M)\ge\lambda\}=\Ccl_\lambda$, and within $\Ccl_\lambda$ the part of rank exactly $\lambda$ is $\{K_\lambda\}$ up to ${\iso}$ (Lemmas E and L), the rest being $\mathrm{HIGH}(\lambda)$. $\Ccl_{\lambda+1}\subseteq\mathrm{HIGH}(\lambda)$: $K_{\lambda+1}\in\Ccl_{\lambda+1}$ has $\cSR=\lambda+1$ exactly (Lemma~E at $\lambda+1$), and Lemma~L allows at most one member of $\Ccl_{\lambda+1}$ of $\cSR\le\lambda+1$, so every other member has $\cSR>\lambda+1$. (iii) Every $M\in H_=(\lambda)$ has $\cSR(M)=\lambda+1$ exactly (the Nadel bound at $\omega_1^M=\lambda$), so Lemma~L injects $H_=(\lambda)$ into the countably many $\equiv_{\lambda+1}$-classes ((H1)): $H_=(\lambda)$ is countable. $W_\lambda^{\mathrm{high}}=W_\lambda\cap H_=(\lambda)$ because on $\omega_1^M=\lambda$ the conditions $\cSR=\lambda+1$ and $\cSR\ge\lambda+1$ coincide. $H_>(\lambda)$ is uncountable because $\mathrm{HIGH}(\lambda)\supseteq\Ccl_{\lambda+1}$ is uncountable (Lemma~B(i) at $\lambda+1$) while $H_=(\lambda)$ is countable. (v) By Karp's theorem on $K_\lambda$'s $\Pin{\lambda+1}$ Scott sentence, $M\equiv_{\lambda+1}K_\lambda\Rightarrow M\iso K_\lambda$. (vi) is definitional: $\Ccl_{\lambda+1}$ is a single $\equiv_{\lambda+1}$-class containing $K_{\lambda+1}$. (i-a) is the Lemma~L injection of $W_\lambda$ into the $\equiv_{\lambda+1}$-classes of models of $\cSR\le\lambda+1$; (i-b) is the named supplies.
\end{proof}

\begin{rlem}[AL (alignment)]
Setting: $\lambda\in\Szero$, so $\lambda\in C_{5.3}\subseteq A(T)$ (Lemma~S(iii-a)), $L(\lambda,T)\models\mathrm{KP}$, $\lambda$ multiplicatively (hence additively) closed; $T_\delta$ ($\delta\le\lambda$) the branch nodes along the $\Delta_1$ branch (increasing, fragments unioned; [Sac07, pp.~8--9]); $n$-types per \S3.7. Let $M$ be a countable model of $T_\lambda$. Then:
\textup{(i) (branch uniqueness)} for every $\delta\le\lambda$: $T(\delta,M)=T_\delta$, and for any node $T'$ on level $\delta$ of $\mathrm{TR}(T)$: $M\models T'\iff T'=T_\delta$;
\textup{(ii) (restriction identity)} $L^M_\lambda\subseteq L_{T_\lambda}$ and $T^M_\lambda=T_\lambda\cap L^M_\lambda$;
\textup{(iii) (type extension)} every $n$-type of $T^M_\lambda$ extends to an $n$-type of $T_\lambda$;
\textup{(iv) (saturation coherence)} if $M$ realizes every $n$-type of $T_\lambda$, then $M$ realizes every $n$-type of $T^M_\lambda$; when $\omega_1^M=\lambda$ this is precisely the $\lambda$-saturation of [Sac07, p.~14];
\textup{(v) (node-homogeneity)} if $\omega_1^M=\lambda$ and $\bar a,\bar b\in M^n$ have the same complete $T^M_\lambda$-type---in particular if $\tp_{L_{T_\lambda}}(\bar a)=\tp_{L_{T_\lambda}}(\bar b)$, by \textup{(ii)}---then some $\sigma\in\Aut(M)$ has $\sigma(\bar a)=\bar b$ coordinatewise;
\textup{(vi) (principality descent)} if $\omega_1^M=\lambda$ and $\bar b\in M$ has principal $T^M_\lambda$-type with generator $g\in L^M_\lambda$, then the $L_{T_\lambda}$-type of $\bar b$ is principal over $T_\lambda$ with the same generator; contrapositively, a realized non-principal $T_\lambda$-type restricts to a realized non-principal $T^M_\lambda$-type;
\textup{(vii-1) (fragment complexity)} every $L_{T_\delta}$-formula ($\delta\le\lambda$) has infinitary complexity $\qr<\lambda$, and hence at limit $\lambda$ every sentence of $T_\lambda$ transfers across $\equiv_\lambda$;
\textup{(vii-2)} every member of $\Ccl_\lambda$ models $T_\lambda$; in particular $K_{\lambda+1}\models T_\lambda$.
\end{rlem}
\noindent The proofs of the individual clauses follow.

\noindent\emph{Proof of \textup{(i)}} (simultaneous induction on $\delta$). $\delta=0$: level-$0$ nodes are finitarily consistent $\omega$-complete extensions of $T$ in $L_0$ (p.~8); if $M\models T'$ then, by completeness in $L_0$ (p.~8, (1)) and truth, $T'=\Th_{L_0}(M)$; the tree analysis sets $T(0,M)=\Th_{L_0}(M)$ with $L_{T(0,M)}=L_0$ (p.~10, (1)); in particular $T_0=T(0,M)$. Successor $\delta+1$: a level-$(\delta+1)$ node $T'$ is an $\omega$-complete finitarily consistent extension of some level-$\delta$ node $S$ in $L'_S$ (p.~9); $M\models T'\Rightarrow M\models S\Rightarrow S=T_\delta=T(\delta,M)$ (IH); then $L'_S=L'_{T(\delta,M)}=L_{T(\delta+1,M)}$ (p.~10, (4)); $T'$ is complete in $L'_{T_\delta}$ and true in $M$, so $T'=\Th_{L_{T(\delta+1,M)}}(M)=T(\delta+1,M)$ (p.~10, (5)). At $T':=T_{\delta+1}$: $T_{\delta+1}=T(\delta+1,M)$. Limit $\mu\le\lambda$: a level-$\mu$ node is a union along a chain (p.~9); $M\models T'\Rightarrow$ each chain member $=T_\beta$ (IH) $\Rightarrow T'=\bigcup_{\beta<\mu}T_\beta=T_\mu$; and $T(\mu,M)=\bigcup_{\beta<\mu}T(\beta,M)=T_\mu$ (p.~10, (2)(3)).

\noindent\emph{Proof of \textup{(ii)}.} The induction in the proof of Prop.~4.5 (p.~10) gives $L^M_\delta\subseteq L_{T(\delta,M)}$ for all $\delta$; by (i), $L_{T(\delta,M)}=L_{T_\delta}$, so $L^M_\lambda\subseteq L_{T_\lambda}$. Then $T^M_\lambda=\Th_{L^M_\lambda}(M)$ (canonical-tower clause~(3), [Sac07, p.~4]) $=\Th_{L_{T_\lambda}}(M)\cap L^M_\lambda=T_\lambda\cap L^M_\lambda$, the last equality by (i) at $\delta=\lambda$.

\noindent\emph{Proof of \textup{(iii)}.} Let $q$ be an $n$-type of $T^M_\lambda$. First, $T_\lambda\cup q(\bar c)$ is finitarily consistent: otherwise finitely many $\gamma_1,\dots,\gamma_m\in q$ yield $T_\lambda\vdash_{\mathrm{fin}}\sigma:=\forall\bar x\,\neg(\gamma_1\wedge\cdots\wedge\gamma_m)$; $\sigma\in L^M_\lambda$ (fragment closure under finitary $\wedge,\neg,\forall$); a finitarily consistent theory complete in its fragment is $\vdash_{\mathrm{fin}}$-closed within it, so $\sigma\in T_\lambda\cap L^M_\lambda=T^M_\lambda$ ((ii)), contradicting $q$'s consistency with $T^M_\lambda$. Second, Lindenbaum in the countable fragment: enumerate $L_{T_\lambda}(\bar x)$ and extend $q$ maximally, keeping finitary consistency with $T_\lambda$ at each stage---an $n$-type of $T_\lambda$ extending $q$.

\noindent\emph{Proof of \textup{(iv)}.} An $n$-type $q$ of $T^M_\lambda$ extends by (iii) to an $n$-type $q^+$ of $T_\lambda$; a realizer of $q^+$ realizes $q$. The p.~14 identification is as stated at $\alpha=\lambda=\omega_1^M$ with $L(\lambda,T)$ $\Sigma_1$-admissible (S(iii-a)). This supplies the step, left implicit in [Sac07], between the construction at [Sac07, p.~15] (``$B$ realizes all the types in $N$''; node terms) and the statement at [Sac07, p.~14] (``$\alpha$-saturated''; tower terms). By Lemma~BF's Step~P below, ``realizes all of $N$'' already implies ``realizes every $n$-type of $T_\lambda$''.

\noindent\emph{Proof of \textup{(v)}.} $T^M_\lambda=T^M_{\omega_1^M}$. By [Sac07, (2.6), p.~5], $M$ is a homogeneous model of $T^M_{\omega_1^M}$; the proof given there (pp.~5--6, (2.9)--(2.11)) establishes the extension form: if $p\subseteq q$ are an $n$- and $(n{+}1)$-type of $T^M_{\omega_1^M}$ with $M\models p(\bar a)\wedge p(\bar b)\wedge\exists y\,q(\bar a,y)$, then $M\models q(\bar b,d)$ for some $d$ (mechanism as stated: $\Sigma_1$-definably chosen refuting witnesses, $\Sigma_1$-admissibility of $L(\omega_1^M,M)$ bounds the refuting levels, the $\forall y\,\neg q_{\delta_1}$-fact transfers along the shared complete top-fragment type; see (2.11) there). The closure of the extension property over the countable $M$ is the usual
iteration, written out here once and cited elsewhere as AL(v): enumerate
$M=\{m_k:k<\omega\}$; set $(\bar a_0,\bar b_0):=(\bar a,\bar b)$; at stage
$2k$, apply the extension form with $p:=\tp_{T^M_\lambda}(\bar a_k)=
\tp_{T^M_\lambda}(\bar b_k)$ and the witness $m_k$ adjoined on the $\bar a$-side
to obtain $d$ with $\tp_{T^M_\lambda}(\bar a_k m_k)=\tp_{T^M_\lambda}(\bar b_k d)$;
at stage $2k+1$, symmetrically absorb $m_k$ on the $\bar b$-side. Types match
at every stage by construction, so the union of the finite maps
$\bar a_k\mapsto\bar b_k$ is total, surjective, and preserves all
$L_{T_\lambda}$-formulas --- in particular atomic formulas and their negations
--- hence is an automorphism $\sigma\in\Aut(M)$ with $\sigma(\bar a)=\bar b$.

\noindent\emph{Proof of \textup{(vi)}.} $g\in L^M_\lambda\subseteq L_{T_\lambda}$ ((ii)). For $F\in\tp_{L_{T_\lambda}}(\bar b)$: any $\bar c$ with $M\models g(\bar c)$ has the same complete $T^M_\lambda$-type as $\bar b$ (generation over $T^M_\lambda$), hence is automorphic to $\bar b$ by (v), hence $M\models F(\bar c)$. So $\forall\bar x(g\to F)\in T_\lambda=\Th_{L_{T_\lambda}}(M)$ ((i)); and $\exists\bar x\,g\in T_\lambda$ likewise. Contrapositive: a realized non-principal $T_\lambda$-type restricts to a realized non-principal $T^M_\lambda$-type.

\noindent\emph{Proof of \textup{(vii-1)}.} $c(\delta):=\sup\{\qr(F)+1:F\in L_{T_\delta}\}$. $c(0)\le\qr(\varphi)+\omega$ (it is the tail clause $\qr(\varphi)<\lambda$ in
the definition of $\Szero$ that keeps these bounds below $\lambda$). Successor: $L'_{T_\delta}$ adds $\bigwedge p$ for every non-principal $n$-type $p$ of $T_\delta$ (eq.~(4.1), p.~9); $\qr(\bigwedge p)\le c(\delta)+1$; finitary closure gives $c(\delta+1)\le c(\delta)+\omega$. Limits: unions (p.~9). By induction $c(\delta)\le\qr(\varphi)+\omega\cdot(1+\delta)<\lambda$ for $\delta<\lambda$ (tail clause of $\Szero$; multiplicative and additive closure). $L_{T_\lambda}=\bigcup_{\delta<\lambda}L_{T_\delta}$; a sentence of $\qr<\lambda$ lies in $\Pin{\gamma}\cup\Sin{\gamma}$ for some $\gamma<\lambda$ and transfers across $\equiv_\gamma\supseteq\equiv_\lambda$ ($\equiv_\lambda=\bigcap_{\gamma<\lambda}\equiv_\gamma$, \S2.1; Karp).
\smallskip
\noindent\emph{Proof of \textup{(vii-2)}.} $B_\lambda\models T_\lambda$ (stated: $B$ is a model of the node $T_{c_\delta}$; [Sac07, p.~15]). $B_\lambda\in\Ccl_\lambda$: $\srS(B_\lambda)=\lambda+1$ ([Sac07, p.~14]) gives $\cSR(B_\lambda)\ge\lambda$, by the lower-bound form of Lemma~TV (this is a rank transport, through Corollary~(TV-b$^-$)); then Lemma~F(iv) at $\lambda\in\Ffix\cap(\qr(\varphi),\omega_1)$. For any $M\in\Ccl_\lambda$: $M\equiv_\lambda B_\lambda$; by (vii-1) every sentence of $T_\lambda$ transfers; $M\models T_\lambda$. $K_{\lambda+1}\in\Ccl_{\lambda+1}\subseteq\Ccl_\lambda$ (Lemma~B(ii)).

\begin{rlem}[BF (a back-and-forth argument on maps preserving fragment types)]
Let $\lambda\in\Szero$; let $M,M'$ be countable models of $T_\lambda$ with $\omega_1^M=\omega_1^{M'}=\lambda$ which realize the same $n$-types of $T_\lambda$ (every $n\ge1$); let $\bar c\in M^{<\omega}$, $\bar d\in M'^{<\omega}$ (possibly empty) with $\tp_{L_{T_\lambda}}(\bar c)=\tp_{L_{T_\lambda}}(\bar d)$. Then there is an isomorphism $\sigma:M\iso M'$ with $\sigma(\bar c)=\bar d$.
\end{rlem}
\noindent The sources used are [Sac07, p.~8, clause (1)], the closure property
at [Sac07, p.~9, eq.~(4.1)], and Lemma~AL(i), (ii), (v).
\begin{proof}
\emph{Step P (principal types are free).} Every principal $n$-type $p$ of $T_\lambda$---generated by an atom $b\in L_{T_\lambda}(\bar x)$ finitarily consistent with $T_\lambda$, with $T_\lambda\vdash_{\mathrm{fin}}\forall\bar x(b\to F)$ for every $F\in p$---is realized in every countable model of $T_\lambda$: $\exists\bar x\,b\in T_\lambda$, since by completeness (p.~8~(1)) the alternative $\neg\exists\bar x\,b\in T_\lambda$ refutes $b$'s finitary consistency with $T_\lambda$; a $b$-realizer realizes all of $p$ by generation. [This also discharges the bracket in AL(iv).]
\emph{Step 0 (the poset).} $\mathcal P:=$ finite partial maps $\bar a\mapsto\bar b$ with $\tp_{L_{T_\lambda}}(\bar a)=\tp_{L_{T_\lambda}}(\bar b)$. The seed $\bar c\mapsto\bar d\in\mathcal P$ by hypothesis; $\emptyset\mapsto\emptyset\in\mathcal P$ since $\Th_{L_{T_\lambda}}(M)=T_\lambda=\Th_{L_{T_\lambda}}(M')$ (AL(i), both sides).
\emph{Step F (forth).} Given $(\bar a\mapsto\bar b)\in\mathcal P$ and $a\in M$, set $q:=\tp_{L_{T_\lambda}}(\bar a a)$---realized in $M$, hence an $n$-type of $T_\lambda$. By the realized-set hypothesis $q$ is realized in $M'$, say by $\bar e=\bar e_0 e$. The restriction of a complete type to an initial subtuple is the complete type of that subtuple, so $\tp_{L_{T_\lambda}}(\bar e_0)=\tp_{L_{T_\lambda}}(\bar b)$. AL(v) in $M'$ [the pin $\omega_1^{M'}=\lambda$ is used exactly here] gives $\tau\in\Aut(M')$ with $\tau(\bar e_0)=\bar b$ coordinatewise; set $b:=\tau(e)$. Fragment formulas are automorphism-invariant, so $(\bar a a\mapsto\bar b b)\in\mathcal P$. \emph{This is the only realization-transfer step; it uses exactly the common realized-type set $P\sqcup(\text{common non-principal part})$ and nothing else (in particular no hidden $\trace=N$ use).}
\emph{Step B (back).} Symmetric, via AL(v) in $M$ [the pin $\omega_1^M=\lambda$].
\emph{Step D (dovetail).} Enumerate both models; from the seed, alternately apply F and B; the union $\sigma$ of the chain is a bijection preserving every $L_{T_\lambda}$-formula---in particular all atomic formulas and negations---hence an isomorphism, with $\sigma(\bar c)=\bar d$.
\end{proof}
\begin{rrem}[attribution]
The same-realized-types $\Rightarrow{\iso}$ half of this mechanism has a stated precedent in Sacks' proof of Theorem~8.1 ([Sac07, pp.~23--24, steps (8.20)--(8.22)]): realized-type-set equality at the top level of $A$'s own canonical tower is established between $A$ and its manufactured companion $A_\beta$ through the single tower's restriction lattice $p_\alpha\subseteq p_\delta\subseteq p_\beta\subseteq p_\gamma$, atomicity supplying the return direction ((8.22)), and (2.6)-homogeneity of both then yields $A_\beta\iso A$ in $V$---tower terms, anchor-vs-companion. Theorem~8.1 carries no saturation hypothesis and asserts no uniqueness, and its comparison never leaves the anchor's own tower; Theorem~5.6's novelty claim (uniqueness among \emph{all} node-saturated $\omega_1=\lambda$ models of the branch node, node terms, seed-parameterized) stands---the precedent is the mechanism, not the theorem.
\end{rrem}

\begin{rlem}[FL (a lower bound on $\srS$ and $\omega_1$)]
If $\lambda\in\Szero$ and $M$ is a countable model of $T_\lambda$, then $\srS(M)\ge\lambda$ and $\omega_1^M\ge\lambda$.
\end{rlem}

\begin{proof}
Fix $\delta<\lambda$. [Sac07, Thm.~4.9(vi), p.~12]: $A_{\delta+1}$ realizes a non-principal $n$-type $p$ of $T_\delta$. Then $\bigwedge p\in L'_{T_\delta}$ (eq.~(4.1), p.~9) and $\sigma_p:=\exists\bar x\bigwedge p\in L'_{T_\delta}$ (closure); $A_{\delta+1}$ is an atomic model of $T_{\delta+1}$ (4.9(v)) realizing $p$, so $A_{\delta+1}\models\sigma_p$; $T_{\delta+1}$ is complete in $L'_{T_\delta}$ (p.~8~(1)) and true in $A_{\delta+1}$, so $\sigma_p\in T_{\delta+1}\subseteq T_\lambda$ (increasing branch). $M\models T_\lambda\Rightarrow M\models\sigma_p$; a witness tuple realizes $p$ (maximality). So $M$ realizes a non-principal type of $T_\delta=T(\delta,M)$ (AL(i)) and is not an atomic model of $T(\delta,M)$; $\delta$ is not a witness for $\tr(M)$ ((4.3), p.~10). All $\delta<\lambda$: $\tr(M)\ge\lambda$; Prop.~4.5 (p.~10): $\srS(M)\ge\tr(M)\ge\lambda$; the Nadel bound (\S2.7; [Sac07, p.~5]): $\srS(M)\le\omega_1^M+1$, and $\lambda$ limit forces $\omega_1^M\ge\lambda$.
\end{proof}

\begin{rlem}[JP (join-pin criterion)]
Let $\lambda\in\Szero$ and let $M$ be a fiber model of $T_\lambda$ (countable, $M\models T_\lambda$, $\omega_1^M=\lambda$); write $\omega_1^{T,M}$ for the relativized invariant of \S2.6 (equivalently,
by Lemma~ID, the least height of an admissible set containing a code of
$\langle T,M\rangle$ --- the locus of [Sac07, p.~23], read per the remark in
\S2.6). Then the hypotheses of [Sac07, Thm.~8.1] hold at $(T,M)\iff\omega_1^{T,M}=\lambda$.
\end{rlem}
\noindent The proof below stays entirely within $\srS$ and the invariants
$\omega_1^{(\cdot)}$; the rank $\cSR$ does not occur in it, so no rank transport
is involved.
\begin{proof}
\emph{(0) (Relativized-Nadel monotonicity.)} $\omega_1^{T,M}\ge\omega_1^M=\lambda$: [Mon13, p.~4] defines $\omega_1^{A,Y}=\min\{\omega_1^X:X\in\Sp(A),X\ge_T Y\}$---a min over a subset of the spectrum's $\omega_1$-values, hence $\ge\omega_1^A$; in the join form, every presentation of $\langle T,M\rangle$ computes a presentation of $M$.
\emph{($\Leftarrow$)} Assume $\omega_1^{T,M}=\lambda$. Hypothesis~1: $\srS(M)\ge\lambda=\omega_1^{T,M}$---Lemma~FL; $M\models T$ by the stated interface of [Sac07, p.~8] (\S3.7). Hypothesis~2: $T$ weakly scattered in $L(\lambda,\langle T,M\rangle)$: (H1) $\Rightarrow$ [Sac07]-scattered clauses (a)$\wedge$(b) by way of [Sac07, Prop.~4.3], as recorded in \S3.7 $\Rightarrow$ clause~(a) alone $\Rightarrow$ ([Sac07, Cor.~3.2, p.~7]) $S_{T'}\in$ every countable $\Sigma_1$-admissible $A$ with $T'\in A$---the localized definition of [Sac07, p.~23], as stated, in the full $T'$-generality that the proof of Theorem~8.1 there uses.
\emph{($\Rightarrow$)} Assume the hypotheses. $\srS(M)\ge\omega_1^{T,M}$ and the Nadel bound (\S2.7), $\srS(M)\le\omega_1^M+1=\lambda+1$, give $\omega_1^{T,M}\le\lambda+1$; with (0), $\lambda\le\omega_1^{T,M}\le\lambda+1$; the height of an admissible set is a limit ordinal $>\omega$, excluding $\lambda+1$: $\omega_1^{T,M}=\lambda$. [Sacks' own first proof line---p.~23: ``Thus $\omega_1^A=\alpha$, since $\omega_1^A+1\ge\mathrm{sr}(A)$''---is the identical computation run in the other direction.]
\end{proof}

\section{Main theorems}

Throughout \S5: (H0)--(H3); $\lambda$ a limit with $\lambda\in\Ffix\cap(\qr(\varphi),\omega_1)$ (admissibility of $\lambda$ is nowhere used by the trichotomy or the identities; it enters only through the supply lemmas~W and~S); ZFC only (Appendix~A). \emph{Where Theorem~\textup{(TV-b)} is used.} Theorems 5.1 and 5.2 involve no rank transport: their inputs are the Nadel bound of \S2.7 (on the $\AKR$--$\cSR$ side, which is TV-c) and Lemmas B, E, L, F and A, and $\srS$ does not occur in them. Theorem~5.3(i) and (ii) use Theorem~(TV-b), through Lemma~S(iv) and (v); Theorem~5.4(ii), (ii$'$) and (iv) inherit that use through Theorem~5.3(i).

\begin{rthm}[5.1 (per-witness trichotomy)]
Let $\lambda$ be as above and $A\in W_\lambda$. Then exactly one of:
\begin{itemize}[leftmargin=2.3em]
\item[\textup{(I)}] $\cSR(A)=\lambda$. Then $A\iso K_\lambda$; consequently $\omega_1^{K_\lambda}=\lambda$ and $\AKR(K_\lambda)=\lambda=\omega_1^{K_\lambda}$.
\item[\textup{(II)}] $\cSR(A)=\lambda+1$ and $A\notin\Ccl_{\lambda+1}$. Then $\mathrm{HIGH}(\lambda)\supsetneq\Ccl_{\lambda+1}$ and $\varphi$ has $\ge2$ non-isomorphic models of $\cSR=\lambda+1$ (namely $A$ and $K_{\lambda+1}$); consequently $\lambda+1\notin\Ffix$, since $\lambda+1\in\Ffix$ would give $\{M\models\varphi:\cSR(M)\ge\lambda+1\}=\Ccl_{\lambda+1}$ by Lemma~F(iv) at $\lambda+1$, against $A\notin\Ccl_{\lambda+1}$.
\item[\textup{(III)}] $\cSR(A)=\lambda+1$ and $A\in\Ccl_{\lambda+1}$. Then $A\iso K_{\lambda+1}$; consequently $\omega_1^{K_{\lambda+1}}=\lambda$ and $\cSR(K_{\lambda+1})=\lambda+1=\omega_1^{K_{\lambda+1}}+1$: $K_{\lambda+1}$ is Nadel-maximal ([Mon13, p.~4]'s ``high Scott rank'', top value).
\end{itemize}
\end{rthm}

\begin{proof}
$\cSR(A)\in\{\lambda,\lambda+1\}$: the Nadel bound of \S2.7 $+$ the definition of $W_\lambda$ ($\cSR\le\omega_1^A+1=\lambda+1$; membership gives $\ge\lambda$)---no appeal to Lemma~W, hence no admissibility used. Membership in $\Ccl_{\lambda+1}$ splits the top case; exhaustive and exclusive. (I): $A\in\Ccl_\lambda$ (Lemma~F(iv)); Lemma~L $+$ Lemma~E give $A\iso K_\lambda$; $\omega_1$ transports by isomorphism-invariance and $\AKR$ by the TV-c limit clause. (II): $K_{\lambda+1}$ exists with $\cSR=\lambda+1$ exactly (Lemma~E at $\lambda+1>\qr(\varphi)$); $A\not\equiv_{\lambda+1}K_{\lambda+1}$ (Lemma~A(vi)), so $A\not\iso K_{\lambda+1}$; both have rank $\lambda+1$. (III): $A\equiv_{\lambda+1}K_{\lambda+1}$, equal ranks, [GRT, Cor.~3.2, p.~12]: $A\iso K_{\lambda+1}$; the pins follow.
\end{proof}
\noindent\emph{Remark.} The proofs of Theorems 5.1 and 5.2 use only Lemmas B, E, L, F and A(vi) together with the Nadel bound of \S2.7; neither Lemma~W nor Lemma~S enters them. The only cited statements used below in relativized rather than literal form are those noted in Lemma~W and in TV-a(ii)/TV-c(ii); the external attributions are those of (H1)/(H2).

\begin{rthm}[5.2 (coordinate identities)]
For a limit $\lambda\in\Ffix\cap(\qr(\varphi),\omega_1)$: $x(\lambda)\iff\omega_1^{K_\lambda}=\lambda$ and $y_g(\lambda)\iff\omega_1^{K_{\lambda+1}}=\lambda$; always $\omega_1^{K_\lambda}\ge\lambda$ and $\omega_1^{K_{\lambda+1}}\ge\lambda$ (Nadel bound, \S2.7). If $y_g(\lambda)$, the glued witnesses form the single isomorphism type $\{K_{\lambda+1}\}$ (by Theorem~5.1(III)).
\end{rthm}
\noindent This is Theorem~5.1 combined with Lemma~E, together with the
existence lemmas where they are invoked.
\begin{proof}
The ($\Leftarrow$) directions place the labels themselves in $W_\lambda/W_\lambda^{\mathrm{high}}$ using only Lemma~E's exact ranks and the assumed $\omega_1$-values; the ($\Rightarrow$) directions are 5.1(I)/(III). No admissibility used.
\end{proof}

\begin{rthm}[5.3 (seeding: $W_\lambda^{\mathrm{high}}\ne\emptyset$ on a club)]
On the club $\Szero$: \textup{(i)} $B_\lambda\in W_\lambda^{\mathrm{high}}$, so $W_\lambda^{\mathrm{high}}\ne\emptyset$ [by Lemma~S(iv), using Theorem~(TV-b) at the successor value], hence $y_g(\lambda)\vee y_s(\lambda)$ for every $\lambda\in\Szero$; \textup{(ii)} $A_\lambda\iso K_\lambda$ [Lemma~S(v), using Theorem~(TV-b) at the limit value] and the labels $\{K_\lambda:\lambda\in\Szero\}$ carry the elementary-chain structure; \textup{(iii)} [Sac07, Thm.~5.3] decides \emph{nothing} about $x(\lambda)$ ($A_\lambda$ carries no $\omega_1$-pin); its operative set is its own $\Sigma_1$-hull club $C_{5.3}$---\emph{not} the $\Sigma_2$-admissibles ([Sac07, Thm.~6.1, p.~15]) and, per Lemma~S(iii-b), \emph{not} necessarily $\subseteq\Adm(t_\varphi)$ ($C_{5.3}\subseteq A(T)$ is the correct inclusion; Lemma~S(iii-a)); and [Sac07, Thm.~5.3] does not decide between $y_g$ and $y_s$ at any $\lambda$ (compare Question~5 of \S6, the corresponding question for $y_B$).
\end{rthm}
\noindent Everything is contained in Lemma~S(iv) and (v).

\begin{rthm}[5.4 (stationarity and the pigeonhole)]
\textup{(i)} $S$ is stationary ($\Ffix$ club; tail club; $\Adm(t_\varphi)$ stationary---for every fixed code, by the $\mathrm{ZF}^-$-collapse argument: a countable $N\prec H_{\omega_2}$ with $t_\varphi\in N$ gives $\lambda=N\cap\omega_1$ with $L_\lambda[t_\varphi]\models\mathrm{KP}$). \textup{(ii)} \emph{Stationary dichotomy:} $\Szero$ is a club and $\Szero=\{\lambda:y_s(\lambda)\}\cup\{\lambda:y_g(\lambda)\}$ [via 5.3(i)], hence at least one of \textup{(a)} $\{\lambda:\text{some witness splits from }\Ccl_{\lambda+1}\}$, \textup{(b)} $\{\lambda:\omega_1^{K_{\lambda+1}}=\lambda\}$ is stationary; (a)-stationary $\Rightarrow$ stationarily many successor levels with $\ge2$ models of $\cSR=\lambda+1$; (b)-stationary $\Rightarrow$ stationarily many Nadel-maximal labels. \textup{(ii$'$) (B-form, strictly finer, exclusive per $\lambda$):} $\Szero=\{\lambda:y_B(\lambda)\}\sqcup\{\lambda:\neg y_B(\lambda)\}$; at least one side stationary; the $y_B$-side gives $B_\lambda\iso K_{\lambda+1}$, a $\lambda$-saturated Nadel-maximal label; the $\neg y_B$-side gives a 5.1(II)-witness, hence the two-models conclusion [both sides use $B_\lambda\in W_\lambda^{\mathrm{high}}={}$5.3(i)]. \textup{(iii)} The pattern map $\lambda\mapsto(x,y_g,y_s)$ takes $\le6$ values on $\Szero$ (of the $8$ boolean patterns, the two with $\neg y_g\wedge\neg y_s$ are excluded by 5.3(i)); one pattern class is stationary. \textup{(iv)} On $\Szero$ there is a high witness at every $\lambda$, so $W_\lambda^{\mathrm{high}}\ne\emptyset$ throughout (Thm.~5.3(i)). \textup{(v) (both-supplies overlap):} $S\cap C_{5.3}=\Ffix\cap\Adm(t_\varphi)\cap C_{5.3}\cap\text{tail}$ is stationary (stationary $\cap$ club $\cap$ club): stationarily many $\lambda$ carry both the Gandy witness and $B_\lambda$ inside $W_\lambda$. Whether they coincide is open (formulated only; \S6).
\end{rthm}
\noindent All parts are routine set theory over Theorems 5.1--5.3.

\begin{rdef}[5.5 (fiber, $N$, $P$, trace)]
For $\lambda\in\Szero$: $N:=$ the set of non-principal $n$-types (all $n\ge1$) of $T_\lambda$---nonempty and countable ([Sac07, p.~15]); $P:=$ the set of principal $n$-types of $T_\lambda$; the \emph{fiber} $:=\{M\text{ countable}:M\models T_\lambda,\ \omega_1^M=\lambda\}$ (by Lemma~FL the pin is equivalently one-sided, $\omega_1^M\le\lambda$). For a fiber $M$: $\trace(M):=\{p\in N:M\text{ realizes }p\}$; by Step~P of Lemma~BF the set of $n$-types of $T_\lambda$ realized in $M$ is exactly $P\sqcup\trace(M)$. Call a countable model $M$ \emph{node-saturated} (at $\lambda$) if $M\models T_\lambda$ and $M$ realizes every $n$-type of $T_\lambda$.
\end{rdef}

\begin{rthm}[5.6 (uniqueness of the node-saturated model)]
$\lambda\in\Szero$. Any two models of $T_\lambda$ with $\omega_1=\lambda$ realizing every $n$-type of $T_\lambda$ are isomorphic. $B_\lambda$ realizes every $n$-type of $T_\lambda$ (all of $N$: [Sac07, p.~15]; all of $P$: Step~P) and $\omega_1^{B_\lambda}=\lambda$ (statement-level at p.~14; explicit at p.~15, ``$\omega_1^B=c_\delta$''); hence $B_\lambda$ is, up to isomorphism, \emph{the} node-saturated $\omega_1=\lambda$ model of $T_\lambda$.
\end{rthm}

\begin{proof}
Both realized-type sets equal the full type set of $T_\lambda$; apply Lemma~BF with the empty seed.
\end{proof}
\noindent\emph{Remark.} The $\omega_1=\lambda$ pins are used only through AL(v) (i.e.\ [Sac07, (2.6)]); uniqueness without the pins is not claimed.

\begin{rlem}[5.7 (U$^+$): the trace classifies the fiber]
$\lambda\in\Szero$; fiber models $M,M'$ with $\trace(M)=\trace(M')\Rightarrow M\iso M'$. Hence $\trace$ is a complete isomorphism invariant on the fiber, and $\Spec_\lambda$ (Definition~5.11; countable by (C1)) enumerates the fiber's isomorphism classes.
\end{rlem}

\begin{proof}
The realized sets are $P\sqcup\trace(\cdot)$ (Definition~5.5), equal by hypothesis; apply BF with the empty seed.
\end{proof}
\noindent\emph{The case of empty trace.} A fiber model of empty trace realizes
only principal types, hence is the atomic model $A_\lambda$ ([Sac07, Thm.~4.9(v)]).
So the class of fiber models of empty trace, if non-empty, is exactly
$\{A_\lambda\}$, and it is non-empty if and only if $\omega_1^{A_\lambda}=\lambda$,
if and only if $x(\lambda)$. Whether this holds is Question~4 of \S6: Lemma~FL
gives $\omega_1^{A_\lambda}\ge\lambda$ at once, but the reverse inequality is not
available.

\begin{rlem}[5.8 (U$^+$p): the pointed form of Lemma 5.7]
$\lambda\in\Szero$; fiber models $M,M'$ with $\trace(M)=\trace(M')$ and $\tp_{L_{T_\lambda}}(\bar c)=\tp_{L_{T_\lambda}}(\bar d)\Rightarrow(M,\bar c)\iso(M',\bar d)$.
\end{rlem}

\begin{proof}
BF with the seed $\bar c\mapsto\bar d$.
\end{proof}
\noindent\emph{Consequence.} The same-cell instance of $\BRctree$ (\S5.15) holds in the strongest form (pointed ${\iso}$); any refutation witness must be cross-cell. Note what the proof of Lemma~BF uses: the seed uses only equality of fragment types, and the single step that transfers realizations uses only $P$ together with the common non-principal part, never $N$ itself; the case of the empty tuple reduces to Lemma~5.7.

\begin{rlem}[5.9 (FIB): the fiber is a single $\equiv_\lambda$-class]
$\lambda\in\Szero$: every fiber model $M$ lies in $\Ccl_\lambda$; equivalently $M\equiv_\lambda B_\lambda$; the fiber is a single $\equiv_\lambda$-class.
\end{rlem}
\noindent The rank transport in this proof is step~(3) below, routed through
Lemma~TV of \S2.5, which records the complete list of such steps.
\begin{proof}
(1) Lemma~FL: $\srS(M)\ge\lambda$. (2) the Nadel bound (\S2.7; [Sac07, p.~5]): $\srS(M)\le\omega_1^M+1=\lambda+1$; so $\srS(M)\in\{\omega_1^M,\omega_1^M+1\}$. (3) \emph{[This is the rank transport; we use the lower-bound form of \S2.5.} TV-a at the value pair (the limit value via the unrelativized multiples-of-$\omega^2$ clause, $\lambda=\omega_1^M$ being admissible hence multiplicatively closed; the successor value via the relativized clause; this step is routine) gives $\SRsym(M)\ge\lambda$; Corollary~(TV-b$^-$)---Lemma~TV's floor corollary---gives $\AKR(M)\ge\lambda$; TV-c(iii) (stated off-by-$1$ $+$ $\lambda$ limit) gives $\cSR(M)\ge\lambda$.\emph{]} (4) $M\models\varphi$ ([Sac07, p.~8], stated interface; \S3.7). (5) $\lambda\in\Ffix\cap(\qr(\varphi),\omega_1)$; Lemma~F(iv): $M\in\Ccl_\lambda$. (6) The same chain at $M:=B_\lambda$ ($\srS(B_\lambda)=\lambda+1$; [Sac07, p.~14]) puts $B_\lambda\in\Ccl_\lambda$; hence $M\equiv_\lambda B_\lambda$.
\end{proof}
\noindent\emph{Remark (avoiding a circularity).} The naive route [$T_\lambda$-theory equality $\Rightarrow\equiv_\lambda$] is the $\emptyset$-tuple instance of $\BRctree$ (\S5.15); the rank route breaks the circle; no unproved transfer between fragment types and $\Pin{\lambda}$-types is used anywhere.

\begin{rcor}[5.10 (fiber ${}=W_\lambda$)]
$\lambda\in\Szero$: $\{M\text{ countable}:M\models T_\lambda,\ \omega_1^M=\lambda\}=\{A\models\varphi:\omega_1^A=\lambda,\ \cSR(A)\ge\lambda\}$; under the counting convention of \S3.4, the fiber is exactly $W_\lambda$.
\end{rcor}

\begin{proof}
$\subseteq$: Lemma~5.9 steps (1)--(4) give $A\models\varphi$ and $\cSR(A)\ge\lambda$; the pin is the fiber's. $\supseteq$: $\cSR(A)\ge\lambda\Rightarrow A\in\Ccl_\lambda$ (F(iv)) $\Rightarrow A\models T_\lambda$ (AL(vii-2)); the pin is $W_\lambda$'s.
\end{proof}
\noindent\emph{Consequences.} $\Spec_\lambda$ coordinatizes $W_\lambda$ up to isomorphism; on the fiber, $\trace(M)=\emptyset\iff M\iso A_\lambda$ (the edge case of Lemma~5.7), and the further identifications $A_\lambda\iso K_\lambda$ / $\cSR=\lambda$ use Theorem~(TV-b) at the limit value, through Lemma~S(v)/S(iv); $\trace(M)\ne\emptyset\iff M\in W_\lambda^{\mathrm{high}}$ uses Theorem~(TV-b) at the successor value, through Lemma~S(iv).

\subsection*{5.11. The trace spectrum $\Spec_\lambda$ and its closure constraints}

\begin{rdef}[5.11 ($\Spec_\lambda$)]
For $\lambda\in\Szero$: $\Spec_\lambda:=\{\trace(M):M\text{ a fiber model}\}\subseteq\mathcal P(N)$ (fiber, $N$, $P$, $\trace$ per Definition~5.5; by Lemma~FL the $\omega_1$-pin in ``fiber'' is equivalently one-sided, $\omega_1^M\le\lambda$).
\end{rdef}
\noindent\emph{Remark (the complexity of the invariant).} Fix $t\in V$ coding the triple $(T_\lambda$, which is $\Delta_1^{L(\lambda,T)}$ via the branch parameter $p$; see [Sac07, p.~15], an enumeration $(p_i)_{i<\omega}$ of $N$, a wellorder of type $\lambda)$. For fixed $S\subseteq N$, membership $S\in\Spec_\lambda$ is $\Sigma^1_2(t\oplus S)$: the witness clause is [$\exists$ countable $M\models T_\lambda$ with $\trace(M)=S$ ($\Delta^1_1$ in codes) and $\omega_1^M\le\lambda$ ($\Sigma^1_1$ in a code: $\forall e[\neg\mathrm{WF}(\Phi_e^x)\vee\ot<\lambda]$)]---the $\ge$-half of the two-sided pin is free by Lemma~FL (a ZFC theorem, uniform in every extension), which collapses the naive $\Sigma^1_2\wedge\Pi^1_2$ shape. Shoenfield: absolute for all set forcing, per instance. $\Spec_\lambda$ as a set gains members only at $S\notin V$; all downstream use is membership-level or existential (``$\exists$ proper nonempty member'' is $\Sigma^1_2$, absolute); frame-level statements naming $K_{\lambda+1}/\Ccl_{\lambda+1}/B_\lambda/A_\lambda$ inherit the $\omega_1$-preserving scope; stationarity is $V$-internal. (See Appendix~A.2; Koerwien's example is excluded here because the countability argument in (C1) uses scatteredness.)

\smallskip
\noindent\textbf{(C0) (node factorization).} Every countable $M\models T_\lambda$ satisfies exactly one node on level $\lambda+1$ of $\mathrm{TR}(T)$, namely $S_M:=\Th_{L'_{T_\lambda}}(M)$; and $\trace(M)=\{p\in N:\sigma_p\in S_M\}$---the $\sigma$-part of $S_M$.
\begin{proof}
The successor fragment $L'_{T_\lambda}$---the least fragment $\supseteq L_{T_\lambda}$ containing $\bigwedge p$ for every non-principal $n$-type $p$ of $T_\lambda$ (eq.~(4.1), [Sac07, p.~9])---exists and is countable: $N$ is non-empty and countable, both stated at [Sac07, p.~15]. \emph{Existence.} $S_M$ is finitarily consistent ($M$ models it); complete in $L'_{T_\lambda}$ ([Sac07, p.~8, clause~(1)], by truth); has the disjunction property (clause~(2) there, semantically); and extends the level-$\lambda$ node: $T_\lambda=\Th_{L_{T_\lambda}}(M)$ [AL(i) at $\delta=\lambda$] with $L_{T_\lambda}\subseteq L'_{T_\lambda}$. So $S_M$ is a finitarily consistent $\omega$-complete extension of $T_\lambda$ in $L'_{T_\lambda}$---a node on level $\lambda+1$. \emph{Uniqueness.} A level-$(\lambda+1)$ node $T'$ with $M\models T'$ extends some level-$\lambda$ node $S$ in $L'_S$; $M\models S$ forces $S=T_\lambda$ [AL(i)], so $L'_S=L'_{T_\lambda}$; completeness of $T'$ in $L'_{T_\lambda}$ plus truth in $M$ gives $T'=S_M$. \emph{The $\sigma$-part.} For $p\in N$: $\sigma_p:=\exists\bar x\bigwedge p\in L'_{T_\lambda}$, and $M\models\sigma_p\iff M$ realizes $p$ (the witness's complete $L_{T_\lambda}$-type is finitarily consistent with $T_\lambda$ and contains the maximal $p$, hence equals it). So $\{p\in N:\sigma_p\in S_M\}=\trace(M)$.
\end{proof}
\noindent The sources used are [Sac07, p.~8, clauses (1) and (2)], [Sac07, p.~9,
eq.~(4.1)] and the successor clause there, the statement about $N$ at
[Sac07, p.~15], and Lemma~AL(i).

\smallskip
\noindent\textbf{(C1) (countability of the spectrum).} $|\Spec_\lambda|\le\aleph_0$.
\begin{proof}
$\varphi$'s scatteredness---(H1), [Mon13, Def.~1.4]---yields [Sac07]-scatteredness, clauses (a)$\wedge$(b) by way of [Sac07, Prop.~4.3], as recorded in \S3.7. In the development terms of [Sac07, p.~13], clauses (a)/(b) refute cases (1)--(4) at every countable $\beta$ (each asserts an uncountable type-set, extension-set, or node-set, giving continuum many models), and case~(5) is refuted by $\varphi$'s $\aleph_1$-many models ((H0)/(H1)). So $\mathrm{vr}(\varphi)=\omega_1$ (in the notation of the development at [Sac07, p.~13]) and the tree develops fully. In particular, case~(3)'s failure at $\beta=\lambda+1$ with $S=T_\lambda$---whose type-sets are countable---forces the negation of its second conjunct: the set of finitarily consistent $\omega$-complete extensions of $T_\lambda$ in $L'_{T_\lambda}$, i.e.\ the successor nodes of $T_\lambda$, is countable. By (C0), every $\trace(M)$ is the $\sigma$-part of a successor node of $T_\lambda$; hence $\Spec_\lambda\subseteq\{\sigma\text{-part}(S):S\text{ a successor node of }T_\lambda\}$, a countable set.
\end{proof}
\noindent The proof uses the failure of case (3) at [Sac07, p.~13] together with (C0);
the inference from (H1) to weak scattering is the one recorded in \S3.7. The
five-case list of p.~13 is not reproduced here; what is consumed from it is
exactly this: were the set of finitarily consistent $\omega$-complete
extensions of $T_\lambda$ in $L'_{T_\lambda}$ uncountable, case~(3) there would
place $T$ outside the weakly scattered case, against \S3.7 --- so that set is
countable, and $\Spec_\lambda$ injects into it through (C0).

\smallskip
\noindent\textbf{(C2) (restriction/extension closure).} For every fiber model $M$: if $p\in\trace(M)$ and $q:=p\restr\bar x'$ (the members of $p$ among the $L_{T_\lambda}(\bar x')$-formulas, $\bar x'$ a subtuple) is non-principal, then $q\in\trace(M)$. Dually, $N\setminus\trace(M)$ is extension-closed.
\begin{proof}
Let $\bar a$ realize $p$; the subtuple $\bar a'$ satisfies every member of $q$. $q$ is an $n$-type of $T_\lambda$: maximal (each $F\in L_{T_\lambda}(\bar x')$ is also an $\bar x$-formula, so $F\in p$ or $\neg F\in p$, and membership descends) and finitarily consistent with $T_\lambda$, witnessed in $M\models T_\lambda$. So $q$ is realized (by $\bar a'$) and non-principal by hypothesis: $q\in N\cap\trace(M)$. The dual clause is the contrapositive: a realizer of $p$ restricts to a realizer of $q$.
\end{proof}

\smallskip
\noindent\textbf{(C3) (top occupied).} $N\in\Spec_\lambda$ at every $\lambda\in\Szero$, with witness $B_\lambda$.
\begin{proof}
All inputs are stated at [Sac07, pp.~14--15]: $B_\lambda\models T_\lambda$; $\omega_1^{B_\lambda}=\lambda$; $B_\lambda$ realizes all the types in $N$. So $B_\lambda$ is a fiber model with $\trace(B_\lambda)\supseteq N$, hence $=N$.
\end{proof}
\noindent All inputs are at [Sac07, pp.~14--15]. (Lemma~AL(iv) identifies
saturation with respect to the node theory with the notion of $\lambda$-saturation
at [Sac07, p.~14]; that identification is not needed here, since the statement
concerns the node theory only.)

\smallskip
\noindent\textbf{(C4) (the empty trace).} \textup{(C4-a)} $\emptyset\in\Spec_\lambda
\iff\omega_1^{A_\lambda}=\lambda$. \textup{(C4-b)} $\omega_1^{A_\lambda}=\lambda\iff
x(\lambda)$; this follows from Lemma~S(v), which gives $A_\lambda\iso K_\lambda$ on
$\Szero$ (and uses Theorem~(TV-b) at the limit value), together with the identity
$x(\lambda)\iff\omega_1^{K_\lambda}=\lambda$ of Thm.~5.2. Whether
$\emptyset\in\Spec_\lambda$ is open; it is Question~4 of \S6. Lemma~FL gives
$\omega_1^{A_\lambda}\ge\lambda$ at once, but the reverse inequality is not
available.
\begin{proof}[Proof of (C4-a)]
($\Leftarrow$) If $\omega_1^{A_\lambda}=\lambda$, then $A_\lambda$ is a fiber model ($A_\lambda\models T_\lambda$: it is the atomic model of $T_\lambda$, [Sac07, Thm.~4.9(v)]); atomicity means every realized complete $L_{T_\lambda}$-type is principal, so $\trace(A_\lambda)=\emptyset$ and $\emptyset\in\Spec_\lambda$. ($\Rightarrow$) If $\trace(M)=\emptyset$ for a fiber $M$, then $M$ realizes only principal types, i.e.\ $M$ is an atomic model of $T_\lambda$; by 4.9's uniqueness $M\iso A_\lambda$, and isomorphism-invariance of $\omega_1$ gives $\omega_1^{A_\lambda}=\omega_1^M=\lambda$.
\end{proof}
\subsection*{5.12. A reduction of $\dq$}

Let $\lambda\in\Szero$. Define $\mathrm{NodeSat}(K_{\lambda+1}):\iff K_{\lambda+1}$ realizes every $n$-type of $T_\lambda$. [Framing note: equivalently, $K_{\lambda+1}$ realizes every member of $N$---principal types are automatic once $K_{\lambda+1}\models T_\lambda$, which is AL(vii-2), via Step~P of Lemma~BF.]

\begin{rthm}[5.12 (reduction of $\dq$)]
Let $\lambda\in\Szero$ and write $K:=K_{\lambda+1}$. Then:
\begin{itemize}[leftmargin=2.6em]
\item[\textup{(1)}] $y_B(\lambda)\iff y_g(\lambda)\wedge\mathrm{NodeSat}(K)$;
\item[\textup{(2)}] $K\le_{\lambda+1}B_\lambda$ implies that $K$ realizes every $p\in N$, hence \textup{(}by the framing note\textup{)} $\mathrm{NodeSat}(K)$; combined with $y_g(\lambda)$ this gives $y_B(\lambda)$ via Clause~1;
\item[\textup{(3)}] $B_\lambda\le_{\lambda+1}K\iff\REALB$, where $\REALB:\iff$ every $\Pin{\lambda}$-type realized in $K$ is realized in $B_\lambda$;
\item[\textup{(4)}] $\dq\iff\NSK(\lambda):=[y_g(\lambda)\Rightarrow\mathrm{NodeSat}(K)]$, pointwise on $\Szero$.
\end{itemize}
\end{rthm}
\noindent Corollary~5.15(e) below characterizes $\REALB$ under $y_g$ as $\Phi$-agreement at the unsupported $q\in\trace(K_{\lambda+1})$. The proofs of the four clauses follow.

\smallskip
\noindent\textbf{Clause 1.} $y_B(\lambda)\iff y_g(\lambda)\wedge\mathrm{NodeSat}(K_{\lambda+1})$.
\begin{proof}
($\Rightarrow$) Assume $y_B(\lambda)$, i.e.\ $B_\lambda\in\Ccl_{\lambda+1}$. Then $B_\lambda\iso K_{\lambda+1}$ [the \S3.6 chain: $B_\lambda\in\Ccl_{\lambda+1}\iff B_\lambda\equiv_{\lambda+1}K_{\lambda+1}\iff B_\lambda\iso K_{\lambda+1}$; Lemma~A(vi) $+$ Lemma~L at equal ranks $\lambda+1$, via [GRT, Prop.~3.1 $+$ Cor.~3.2, p.~12]]. Hence $\omega_1^{K_{\lambda+1}}=\omega_1^{B_\lambda}=\lambda$ (isomorphism-invariance; the pin stated at [Sac07, p.~15]), which is $y_g(\lambda)$ [Thm.~5.2]. And $K_{\lambda+1}$ inherits $B_\lambda$'s realized-type set across the isomorphism (fragment formulas are ${\iso}$-invariant): $B_\lambda$ realizes every $n$-type of $T_\lambda$ (all of $N$ stated at p.~15; all of $P$ by Step~P), so $\mathrm{NodeSat}(K_{\lambda+1})$. [$K_{\lambda+1}\models T_\lambda$ is not separately used in this direction; it comes with the isomorphism.]
($\Leftarrow$) Assume $y_g(\lambda)\wedge\mathrm{NodeSat}(K_{\lambda+1})$; write $K:=K_{\lambda+1}$. (i) $\omega_1^K=\lambda$ [$y_g$, Thm.~5.2]. (ii) $K\models T_\lambda$ [AL(vii-2)]. (iii) $K$ and $B_\lambda$ are then two countable models of $T_\lambda$ with $\omega_1=\lambda$ realizing the same $n$-types of $T_\lambda$---namely all of them (NodeSat for $K$; stated p.~15 $+$ Step~P for $B_\lambda$). (iv) Lemma~BF with the empty seed (equivalently Theorem~5.6): $B_\lambda\iso K$. (v) $K_{\lambda+1}\in\Ccl_{\lambda+1}$ (the label, definitional) and ${\iso}$ preserves $\Ccl_{\lambda+1}$-membership: $B_\lambda\in\Ccl_{\lambda+1}$---$y_B(\lambda)$.
\end{proof}

\smallskip
\noindent\textbf{Clause 2.} $K_{\lambda+1}\le_{\lambda+1}B_\lambda\Rightarrow K_{\lambda+1}$ realizes every $p\in N$---hence $\mathrm{NodeSat}(K_{\lambda+1})$ [modulo the framing note]---and, combined with $y_g(\lambda)$, gives $y_B(\lambda)$ via Clause~1($\Leftarrow$).
\begin{proof}
(i) Karp orientation [GRT, Thm.~1.1(3), p.~6]: $K\le_{\lambda+1}B_\lambda\iff$ every $\Sin{\lambda+1}$ sentence true in $B_\lambda$ is true in $K$ (clause~(3) at the empty tuples). (ii) For $p\in N$: $\bigwedge p$ is a countable conjunction of $L_{T_\lambda}$-formulas, each of $\qr<\lambda$ [AL(vii-1)], so $\bigwedge p\in\Pin{\lambda}$ and $\sigma_p:=\exists\bar x\bigwedge p\in\Sin{\lambda+1}$. (iii) $B_\lambda\models\sigma_p$ ([Sac07, p.~15]). (iv) By (i), $K\models\sigma_p$: a witness tuple satisfies every member of $p$---$K$ realizes $p$. (v) So $K$ realizes every $p\in N$; the NodeSat form adds the principal types via the framing note [AL(vii-2)]; with $y_g$, Clause~1($\Leftarrow$) gives $y_B$.
\end{proof}
\noindent\emph{Remark.} The stronger form ``$K_{\lambda+1}\le_{\lambda+1}B_\lambda
\Rightarrow y_B(\lambda)$'', without the hypothesis $y_g(\lambda)$, is not
available: $\mathrm{NodeSat}(K_{\lambda+1})$ by itself is not known to force
$\omega_1^{K_{\lambda+1}}=\lambda$, and whether it does is open (the question is
recorded within Question~2 of \S6). This is why the
conclusion $y_B$ in Clause~2 is obtained through $y_g$.\\
\noindent Steps (i)--(iv) are self-contained; the passage to $\mathrm{NodeSat}$
and thence to $y_B$ uses Lemma~AL(vii-2), through Clause~1($\Leftarrow$).

\smallskip
\noindent\textbf{Clause 3.} $B_\lambda\le_{\lambda+1}K_{\lambda+1}\iff\REALB$ (every $\Pin{\lambda}$-type realized in $K_{\lambda+1}$ is realized in $B_\lambda$).
\begin{proof}
Write $K:=K_{\lambda+1}$. ($\Leftarrow$) Assume $\REALB$. By Karp's theorem ([GRT, Thm.~1.1(3), p.~6], at the empty tuples) it suffices to transfer every true $\Sin{\lambda+1}$ sentence from $K$ to $B_\lambda$. Such a $\sigma$ is a countable disjunction of formulas $\exists\bar x\,\psi$ with $\psi\in\Pin{\lambda}$ [the $\Sigma_{\lambda+1}$ normal form; hierarchy per [MonP2, Ch.~2], by reference]. $K\models\sigma$ yields a true disjunct $\exists\bar x\,\psi$ with witness $\bar a\models\psi$
[no disjunction-splitting principle is invoked here: a
single true disjunct is extracted semantically, and only that disjunct is used
below]; $\{\psi\}$ is a $\Pin{\lambda}$-type realized in $K$; $\REALB$ supplies $\bar b\models\psi$ in $B_\lambda$; so $B_\lambda\models\exists\bar x\,\psi$ and $B_\lambda\models\sigma$. ($\Rightarrow$) Assume $B_\lambda\le_{\lambda+1}K$. Let $\Gamma$ be a $\Pin{\lambda}$-type realized in $K$, at $\bar a$ say. Unfold [GRT, Def.~1.2, p.~6] at $(B_\lambda,\emptyset)\le_{\lambda+1}(K,\emptyset)$, taking $\beta:=\lambda<\lambda+1$ and $\bar d:=\bar a\in K$: there is $\bar b\in B_\lambda$ with $(K,\bar a)\le_\lambda(B_\lambda,\bar b)$. By Karp's theorem ([GRT, Thm.~1.1(2)] at $\alpha=\lambda$), every $\Pin{\lambda}$ formula true of $\bar a$ in $K$ is true of $\bar b$ in $B_\lambda$; in particular $\bar b$ realizes $\Gamma$.
\end{proof}
\noindent The sources used are [GRT, p.~6, Def.~1.2 and Thm.~1.1(2), (3)], and
the normal form for $\Sin{\lambda+1}$ sentences, for which we refer to
[MonP2, Ch.~2].

\smallskip
\noindent\textbf{Clause 4.} $\dq\iff\NSK(\lambda):=[y_g(\lambda)\Rightarrow\mathrm{NodeSat}(K_{\lambda+1})]$, pointwise on $\Szero$.
\begin{proof}
$\dq(\lambda):=[y_g(\lambda)\Rightarrow y_B(\lambda)]$. ($\Leftarrow$) Given $\NSK$ and $y_g$: NodeSat holds, and Clause~1($\Leftarrow$) gives $y_B$. ($\Rightarrow$) Given $\dq$ and $y_g$: $y_B$ holds, and Clause~1($\Rightarrow$) gives NodeSat.
\end{proof}
\noindent The statement $\dq$, equivalently $\NSK$, is itself open; see \S6. We
note that it is a condition on a single $\lambda$ together with one assertion
about the types realized in one structure: it entails no counting of models by
rank and says nothing about $\SRp$, so neither limb of \S1.3 is in play.

\begin{rlem}[5.13 (NS$^-$): the label realizes a non-principal type]
For every $\lambda\in\Szero$ with $y_g(\lambda)$: $\trace(K_{\lambda+1})\ne\emptyset$.
\end{rlem}
\emph{Preliminaries (common to both routes).} $K:=K_{\lambda+1}$. $y_g$ gives $\omega_1^K=\lambda$ [Thm.~5.2]. $K\models T_\lambda$ [AL(vii-2)]. Hence $K$ is a fiber model, and for $\bar a\in K^n$ the complete type $p_{\bar a}:=\tp_{L_{T_\lambda}}(\bar a)$ is an $n$-type of $T_\lambda$ [AL(i)], realized in $K$; a realized non-principal $p_{\bar a}$ lies in $N$ automatically. We give two proofs: route~(a) is self-contained at the empty tuple; route~(b) is the form that generalizes to Lemma~5.14.

\smallskip
\noindent\emph{Route (a) --- orbit form.} Suppose $K$ omits every $p\in N$. Then every $p_{\bar a}$, being realized and $\notin N$, is principal: some atom $b\in L_{T_\lambda}(\bar x)$ generates it over $T_\lambda$, with $\qr(b)<\lambda$ [AL(vii-1)]. The generation sentences $\forall\bar x(b\to F)$, $F\in p_{\bar a}$, lie in $T_\lambda$, so they hold in $K$: every $\bar c$ with $K\models b(\bar c)$ satisfies all of $p_{\bar a}$, hence $\tp_{L_{T_\lambda}}(\bar c)=p_{\bar a}$. By AL(ii), $\bar c$ and $\bar a$ share the complete $T^K_\lambda$-type; by AL(v) [pin $\omega_1^K=\lambda$] they are automorphic. Conversely automorphic tuples satisfy the same fragment formulas, in particular $b$. So the $b$-realizers are exactly $\Aut(K)\cdot\bar a$: every orbit of $K$ is parameter-free definable by a formula of $\qr<\lambda$, hence $\Sin{\lambda}$---(U1) at $\alpha=\lambda$ [Mon15, Thm.~1.1, pp.~5427--5428], so $\cSR(K)\le\lambda$, contradicting Lemma~E [$\cSR(K_{\lambda+1})=\lambda+1$ exactly].

\smallskip
\noindent\emph{Route (b) --- pointed form.} $\cSR(K)=\lambda+1$ [Lemma~E] means $\neg$(U2 at $\lambda$): no $\Pin{\lambda+1}$ Scott sentence; [Mon15, Thm.~1.1] [(U2)$\iff$(U5), $\emptyset$-parameter instance] gives $\neg$(U5 at $\lambda$): some $\Pin{\lambda}$-type $\Gamma^*(\bar z)$ realized in $K$ is not $\Sin{\lambda}$-supported within $K$ [Mon15, Def.~3.1, p.~5433]. Let $\bar a\models\Gamma^*$. If $p_{\bar a}$ were principal with atom $b$, then---exactly as in route (a), under $y_g$---every $b$-realizer is automorphic to $\bar a$ [AL(v)], hence realizes $\Gamma^*$; and $b$ is realized with $\qr(b)<\lambda$, hence $\Sin{\lambda}$ [AL(vii-1)]: $b$ is a realized $\Sin{\lambda}$-support of $\Gamma^*$ within $K$, a contradiction. So $p_{\bar a}$ is non-principal: $p_{\bar a}\in N\cap\trace(K)\ne\emptyset$.

\smallskip
\noindent\emph{A further consequence of route (b).} The fragment type of any realizer of any unsupported $\Pi_\lambda$-type is a member of $\trace(K)$---a pointed witness; homogeneity (available under $y_g$) substitutes for $\BRctree$ at exactly one type. This is why $y_g$ is needed in Lemma~5.13, whereas the corresponding statement at the empty tuple is not.
\begin{rrem}[Sharpness]
The argument uses the only fact available about the label---the exactness of its rank, Lemma~E---and yields exactly one type: a single non-$\Sigma_\lambda$ orbit already refutes (U1) at $\lambda$; nothing iterates. The companion one-liner in Sacks' terms ($\srS(K)=\lambda+1$ would say $K$ is not atomic over $T^K_\lambda$) is not available to us: lifting it to $T_\lambda$ requires exactly the relation asked for in Question~3 of \S6, and the value of $\srS$ would in addition require the transport of \S2.5.
\end{rrem}
\noindent\emph{Consequence:} under $y_g(\lambda)$, the first of the two weaker forms recorded at Question~2 of \S6 holds: for some $p_0\in N$, every member of $\Ccl_{\lambda+1}$ satisfies the $\Sin{\lambda+1}$ sentence $\sigma_{p_0}$ (Lemma~5.13 gives $K_{\lambda+1}\models\sigma_{p_0}$; every $M\in\Ccl_{\lambda+1}$ is $\equiv_{\lambda+1}K_{\lambda+1}$ by Lemma~A(vi), and $\equiv_{\lambda+1}$ preserves $\Sin{\lambda+1}$ sentences, \S2.1).\\
\noindent The only point at which a formula of complexity $\lambda+1$ enters is
the computation of the complexity of $\sigma_p$, through Lemma~AL(vii-1). The
statement $K\models T_\lambda$ is Lemma~AL(vii-2).

\begin{rlem}[5.14 (EXT): a relative form of Lemma 5.13]
$\lambda\in\Szero$ with $y_g(\lambda)$, $K:=K_{\lambda+1}$: for every $\bar c\in K^{<\omega}$ there is $\bar d$ with $\tp_{L_{T_\lambda}}(\bar c\bar d)\in N\cap\trace(K)$---$\trace(K)$ is cofinal over every realized tuple. Strictly generalizes NS$^-$ route~(b) (the $\bar c=\emptyset$ case).
\end{rlem}
\begin{proof}
Fix $\bar c$. \emph{(1) Unsupported supply over $\bar c$.} $\SRp(K)=\lambda+1$ (Proposition~7.1; no rank transport is involved) gives $\cSR(K,\bar c)\ge\lambda+1$ for every finite $\bar c$ [Alv21, Prop.~1.10(1), p.~1709]. So $(K,\bar c)$ fails (U2) at $\alpha=\lambda$; [Mon15, Thm.~1.1] applied to the countable structure $(K,\bar c)$ gives $\neg$(U5 at $\lambda$) there: some $\Pin{\lambda}$-type $\Gamma^*(\bar z;\bar c)$ realized in $K$ is not $\Sin{\lambda}$-supported over $\bar c$ ([Mon15, Def.~3.1], read in the expansion). \emph{(2)} Let $\bar d\models\Gamma^*(\bar z;\bar c)$ in $K$, and $q:=\tp_{L_{T_\lambda}}(\bar c\bar d)$---realized, an $n$-type of $T_\lambda$ [$K\models T_\lambda$: AL(vii-2)]. \emph{(3) The coordinatewise-AL(v) step (written out here once; Lemma~5.13 and \S5.15 use the same step).} Suppose $q$ had an atom $b(\bar x,\bar z)$ over $T_\lambda$. Take any $\bar d'$ with $K\models b(\bar c,\bar d')$. The generation sentences $\forall\bar x\bar z(b\to F)$, $F\in q$, lie in $T_\lambda$ and hold in $K$, so $\tp_{L_{T_\lambda}}(\bar c\bar d')=q=\tp_{L_{T_\lambda}}(\bar c\bar d)$. Now AL(v) [pin $\omega_1^K=\lambda$, from $y_g$ via Thm.~5.2] applies to the concatenated tuples and matches them coordinatewise: it supplies $\sigma\in\Aut(K)$ with $\sigma(\bar c\bar d)=\bar c\bar d'$ coordinate-by-coordinate---the target's first $|\bar c|$ coordinates are $\bar c$ itself, so $\sigma$ fixes $\bar c$ pointwise and $\sigma(\bar d)=\bar d'$. $\Pin{\lambda}$ formulas are automorphism-invariant, so $\bar d'=\sigma(\bar d)\models\Gamma^*(\bar z;\sigma(\bar c))=\Gamma^*(\bar z;\bar c)$. \emph{(4)} So $b(\bar c,\bar z)$ is realized ($\bar d$ witnesses) and every realizer of it realizes $\Gamma^*(\cdot;\bar c)$; $b$ has $\qr<\lambda$ [AL(vii-1)], so $b(\bar c,\bar z)$ is $\Sin{\lambda}$ over $\bar c$: a realized support of $\Gamma^*$ over $\bar c$ in the exact sense of [Mon15, Def.~3.1], contradicting~(1). \emph{(5)} Hence $q$ is non-principal, so $q\in N$, and $\bar c\bar d$ realizes it: $q\in N\cap\trace(K)$.
\end{proof}
\noindent\emph{Remark.} Homogeneity substitutes for the missing transfer principle $\BRctree$ at every tuple, not just at one type---the difficulty lies entirely in the supply of unsupported types, not in the passage between fragments. To summarize: the fragment types produced this way form a non-empty subfamily of $N$ (Lemma~5.13) that is closed under restriction ((C2)) and cofinal in the sense of Lemma~5.14; whether it exhausts $N$ is Question~2 of \S6; no way of steering the unsupported supply onto a prescribed $p\in N$ is available from the sources cited here.\\
\noindent The statement $K\models T_\lambda$, used at steps (2) and (3), is
Lemma~AL(vii-2). The other sources are [Mon15, Thm.~1.1 and Def.~3.1, pp.~5427--5428
and 5433], [Alv21, Prop.~1.10(1), p.~1709], and Lemma~AL(i), (ii), (v), (vii-1)
together with the closure property established in the proof of AL(iii).

\subsection*{5.15. Support localization, reduction to $B_\lambda$, and a characterization of $\REALB$}

\noindent\emph{The $\Phi$-maps (well-definedness).} For a fiber model $M$ and $q\in P\sqcup\trace(M)$: $\Phi_M(q):=$ the complete $\Pin{\lambda}$-type of any realizer of $q$ in $M$. Well-defined: two realizers of $q$ share the complete $L_{T_\lambda}$-type (both equal $q$); AL(v) [pin $\omega_1^M=\lambda$, the fiber's; the 5.14 step-(3) mechanism] makes them automorphic; $\Pin{\lambda}$ formulas are automorphism-invariant. \emph{Two conventions:} \textup{(i)} at limit $\lambda$, a formula of $\qr<\lambda$ is both $\Pin{\lambda}$- and $\Sin{\lambda}$-expressible; in particular every $L_{T_\lambda}$-formula is [AL(vii-1)], so $q\subseteq\Phi_M(q)$ in the expressible sense, and $\Phi_M(q)$ is $\Pi_\lambda$-complete over $q$. \textup{(ii)} a sentence of $\qr<\lambda$ transfers across $\equiv_\lambda$ in both directions [AL(vii-1)'s closing line].

\begin{rlem}[5.15: support localization]
$\lambda\in\Szero$; $M,M'$ fiber models; $q$ realized in both; $\Gamma:=\Phi_M(q)$. If $\Gamma$ is $\Sin{\lambda}$-supported within $M$ [Mon15, Def.~3.1, p.~5433] by $\theta$, then $\Phi_{M'}(q)=\Gamma$, and the same $\theta$ supports $\Gamma$ within $M'$.
\end{rlem}
\begin{proof}
(1) $M\equiv_\lambda M'$ [Lemma~5.9, used as an input]. (2) \emph{The support data, complexity-computed (the single-disjunct-guard computation, once).} At limit $\lambda$, $\theta\in\Sin{\lambda}$ has the raw form $\theta=\bigvee_{i<\omega}\theta_i$ with $\qr(\theta_i)<\lambda$ (only the form of the definition is used here). The facts about $\theta$ in $M$ from [Mon15, Def.~3.1]: (s1) $M\models\exists\bar x\,\theta$; (s2) $M\models\forall\bar x(\theta\to\gamma)$ for each $\gamma\in\Gamma$. Complexities: $\exists\bar x\,\theta\iff\bigvee_i\exists\bar x\,\theta_i$, each disjunct of $\qr<\lambda$, so $\exists\bar x\,\theta\in\Sin{\lambda}$; for (s2), write $\gamma=\bigwedge_{j<\omega}\gamma_j$ with $\qr(\gamma_j)<\lambda$: $\forall\bar x(\theta\to\gamma)\iff\bigwedge_{i,j}\forall\bar x(\neg\theta_i\vee\gamma_j)$, each conjunct of $\qr<\lambda$, so the countable conjunction is $\Pin{\lambda}$. \emph{Pairing each disjunct $\theta_i$ with each conjunct $\gamma_j$ keeps every constituent strictly below $\lambda+1$}, where the unfactored $\forall\bar x(\Sigma_\lambda\to\Pi_\lambda)$ would naively sit at $\Pi_{\lambda+1}$. (3) \emph{Transfer.} All of (s1)/(s2) transfers across $\equiv_\lambda$, both directions [(ii)]. So $\theta$ is realized in $M'$ and implies, within $M'$, every $\gamma\in\Gamma$. (4) \emph{Cell identification.} Pick $\bar e\models\theta$ in $M'$; by (3) $\bar e$ satisfies every $\gamma\in\Gamma\supseteq q$ [(i)], so $\tp_{L_{T_\lambda}}(\bar e)=q$ [(i)]. AL(v) in $M'$ [pin $\omega_1^{M'}=\lambda$] relocates onto any $q$-realizer $\bar f$: $\tp_{\Pin{\lambda}}(\bar f)\supseteq\Gamma$. Hence $\Phi_{M'}(q)\supseteq\Gamma$. (5) \emph{Equality.} Suppose $\psi\in\Phi_{M'}(q)\setminus\Gamma$, $\psi\in\Pin{\lambda}$. Then $M\models\neg\psi(\bar a)$ at the $M$-realizer $\bar a$ of $q$; $\neg\psi\in\Sin{\lambda}$ is $\bigvee_k\chi_k$ with $\qr(\chi_k)<\lambda$; some true disjunct $\chi$ has $M\models\chi(\bar a)$; $\chi$, of $\qr<\lambda$, is $\Pin{\lambda}$-expressible [(i)], so $\chi\in\Gamma\subseteq\Phi_{M'}(q)$---but then the $M'$-realizer of $q$ satisfies $\chi$, hence $\neg\psi$, contradicting $\psi\in\Phi_{M'}(q)$. So $\Phi_{M'}(q)=\Gamma$. (6) \emph{Same-$\theta$ support in $M'$:} (3)'s transferred sentences are exactly the clauses of [Mon15, Def.~3.1] for $\theta$ and $\Gamma$ in $M'$.
\end{proof}
\noindent Step (1) uses Lemma~5.9.

\smallskip
\noindent\textbf{Corollaries.}
\begin{itemize}[leftmargin=1.7em]
\item[\textup{(a)}] \emph{Principal cells are free.} For $q\in P$ with atom $b$: every $b$-realizer in a fiber $M$ realizes $q$ and all $q$-realizers are automorphic [AL(v)], so all $b$-realizers share $\Phi_M(q)$; $b$ is realized and of $\qr<\lambda$, hence $\Sin{\lambda}$ [AL(vii-1)]: a realized $\Sin{\lambda}$-support of $\Phi_M(q)$ within $M$. Lemma~5.15 then makes $\Phi_\cdot(q)$ constant across the fiber at every principal $q$.
\item[\textup{(b)}] \emph{Fiber-uniformity of support.} For every cell $q$: the support-status of $\Phi_\cdot(q)$, and the supporting formula itself, are fiber-uniform. [Here Lemma~5.9 and AL(v) supply the realization that the weaker status-only transfer would note as non-transferring.]
\item[\textup{(c)}] \emph{Reduction to $B_\lambda$.} Every fragment type realized in any fiber model is realized in $B_\lambda$: $P\sqcup\trace(M)\subseteq P\sqcup N$, and $B_\lambda$ realizes all of $N$ ([Sac07, p.~15]) and all of $P$ (Step~P). Pointwise comparisons of $\Phi$ therefore factor through $B_\lambda$, and $\BRctree$---in the $\Phi$-formulation: $\Phi_M(q)=\Phi_{M'}(q)$ for all fiber $M,M'$ and common realized $q$---$\iff$ its $(M,B_\lambda)$-restricted instance $\iff[\Phi_M=\Phi_{B_\lambda}$ on $P\sqcup\trace(M)$ for every fiber $M]$.
\item[\textup{(d)}] With (a)$+$(b): $\BRctree$ is equivalent to agreement of $\Phi$ with $B_\lambda$ at the \emph{unsupported} cells only---supported cells (principal ones included, by (a)) are free by Lemma~5.15.
\item[\textup{(e)}] \emph{The $\REALB$ characterization (the object of Theorem~5.12, Clause~3).} At $\lambda\in\Szero$ with $y_g(\lambda)$---so $K:=K_{\lambda+1}$ is on the fiber [$\omega_1^K=\lambda$ by Thm.~5.2; $K\models T_\lambda$ by AL(vii-2)]---the following are equivalent: \textup{(e1)} $\REALB$; \textup{(e2)} $\Phi_K(q)=\Phi_{B_\lambda}(q)$ for every $q\in P\sqcup\trace(K)$; \textup{(e3)} $\Phi_K(q)=\Phi_{B_\lambda}(q)$ for every unsupported $q\in\trace(K)$.
\end{itemize}
\noindent\emph{Proof of (e).} (e2)$\iff$(e3): every $q\in P\sqcup\trace(K)$ is realized in $B_\lambda$ [(c)], and supported cells agree for free [Lemma~5.15 $+$ (a)]. (e2)$\Rightarrow$(e1): $\Gamma$ realized in $K$ at $\bar a$; $q:=\tp_{L_{T_\lambda}}(\bar a)\in P\sqcup\trace(K)$; $\Gamma\subseteq\Phi_K(q)=\Phi_{B_\lambda}(q)$, realized in $B_\lambda$. (e1)$\Rightarrow$(e2): fix $q$ with $K$-realizer $\bar a$; $\Phi_K(q)$ is realized in $B_\lambda$ at some $\bar b'$; $\bar b'$'s fragment type is exactly $q$ [(i)]; AL(v) in $B_\lambda$ relocates onto any $q$-realizer: $\Phi_{B_\lambda}(q)\supseteq\Phi_K(q)$; equality as in step~(5) of Lemma~5.15. \hfill$\square$\\
\noindent In (e), the statement that $K$ lies on the fiber uses Lemma~AL(vii-2).

\smallskip
\noindent\textbf{The open case (restated).} $\BRcuns$: at $\lambda\in\Szero$, for every $q\in N$ realized in a fiber model $M$ with $\Phi_M(q)$ not $\Sin{\lambda}$-supported within $M$: $\Phi_M(q)=\Phi_{B_\lambda}(q)$. By (c) and (d) this is equivalent to the full statement $\BRctree$; the instances at the empty tuple, within a single cell, at principal cells and at supported cells are proved above. The general case is open; it is Question~3 of \S6.

\section{Open questions}

Throughout this section $\lambda$ is a limit in $\Szero$, so that we may work on
the fiber above $\lambda$; the objects concerned are the labels $K_\lambda$ and
$K_{\lambda+1}$, the models $A_\lambda$ and $B_\lambda$ of [Sac07, Thm.~5.3], and
the conditions $x,y_g,y_s,y_B$ of \S3.6. The first three questions all reduce to
one difficulty: we have no control, from the side of the label, over the
partial $\Pin{\lambda}$-types that are not supported.

\smallskip
\noindent\textbf{Question 1 ($\dq$).} Does $y_g(\lambda)$ imply $y_B(\lambda)$?
Equivalently, by Clause~4 of Theorem~5.12, does $\NSK(\lambda)$ hold, that is, does
$y_g(\lambda)$ imply $\mathrm{NodeSat}(K_{\lambda+1})$? The implications
$B_\lambda\iso K_{\lambda+1}\iff B_\lambda\equiv_{\lambda+1}K_{\lambda+1}\iff
B_\lambda\in\Ccl_{\lambda+1}$ are proved (Lemma~A(vi) and Lemma~L at $\lambda+1$),
and each of them implies $y_g(\lambda)$, since the isomorphism gives
$\omega_1^{K_{\lambda+1}}=\lambda$ and hence $y_g$ by Thm.~5.2. The converse is
what is at issue. An equivalent formulation is a rigidity statement about
$W_\lambda^{\mathrm{high}}$: $y_g(\lambda)$ asserts only that \emph{some} high
witness at $\lambda$ glues, and different witnesses at the same $\lambda$ might a
priori behave differently.

\smallskip
\noindent\textbf{Question 2 ($\NSKp$).} Write $\NSKp(\lambda)$ for the
implication $y_g(\lambda)\Rightarrow\trace(K_{\lambda+1})=N$. Does it hold? By
(C3) we have $N=\trace(B_\lambda)\in\Spec_\lambda$, and under $y_g$ the label
$K_{\lambda+1}$ lies on the fiber and is determined up to isomorphism by its
trace (Theorem~5.6 together with Cor.~5.10). Hence, under $y_g(\lambda)$, the
three statements $\trace(K_{\lambda+1})=N$, $K_{\lambda+1}\iso B_\lambda$ and
$y_B(\lambda)$ are equivalent, so $\NSKp$ is equivalent to $\NSK$, in agreement
with Theorem~5.12. Two weaker forms may be easier: that every member of
$\Ccl_{\lambda+1}$ satisfies $\sigma_p$ for some particular $p\in N$
(established, under $y_g$, at the Consequence after Lemma~5.13); and that
$S_{K_{\lambda+1}}=S_{B_\lambda}$ under $y_g$, in the notation of (C0). Lemmas 5.13 and 5.14 give
$\trace(K_{\lambda+1})\ne\emptyset$ and show it to be cofinal in the sense of
Lemma~5.14; whether it exhausts $N$ is the question. A related sub-question,
recorded at the Remark after Clause~2 of \S5.12: does
$\mathrm{NodeSat}(K_{\lambda+1})$ by itself force
$\omega_1^{K_{\lambda+1}}=\lambda$ --- equivalently $y_g(\lambda)$, by
Thm.~5.2? A positive answer would remove the hypothesis $y_g$ from Clause~2 of
\S5.12 and make $\mathrm{NodeSat}(K_{\lambda+1})$ alone equivalent to
$y_B(\lambda)$, by Clause~1.

\smallskip
\noindent\textbf{Question 3 ($\BRcuns$).} For $\lambda\in\Szero$ and every $q\in N$
realized in a fiber model $M$ with $\Phi_M(q)$ not $\Sin{\lambda}$-supported
within $M$: is $\Phi_M(q)=\Phi_{B_\lambda}(q)$? By Corollaries 5.15(c) and (d)
this is equivalent to the full statement $\BRctree$, that $\Phi$ agrees across
the fiber; the instances at the empty tuple, within a single cell, at principal
cells, and at supported cells are all proved above. The obstruction is that we
know of no published relation, in either direction, between equality of fragment
types and equality of $\Pin{\lambda}$-types: the coincidence of $\srS$ with
$\SRsym$ at [Mon15, p.~5432] is at the level of ranks only, and the analysis in
[Sac07, \S2] is at the level of isomorphism. The statement is plausibly classical,
but we do not use it without a proof or a reference.

\smallskip
\noindent\textbf{Question 4.} Is $\emptyset\in\Spec_\lambda$ --- equivalently, by
(C4), is $\omega_1^{A_\lambda}=\lambda$, equivalently does $x(\lambda)$ hold?
Lemma~FL gives $\omega_1^{A_\lambda}\ge\lambda$ at once; the reverse inequality is
not available, and [Sac07, Thm.~5.3] asserts nothing about $\omega_1^{A_\lambda}$.
The question mirrors a difficulty recorded in [Sac07] itself: the remark
following (2.12) at [Sac07, p.~6] notes that the model produced there has
Scott rank $\alpha$ or $\alpha+1$, and that forcing the value $\alpha+1$ is a
problem ``addressed in this paper but far from resolved''; Question~4 asks for
the mirror control, forcing the value $\lambda$ at the atomic model.

\smallskip
\noindent\textbf{Question 5.} Which of $y_B(\lambda)$, $\neg y_B(\lambda)$ holds,
at a given $\lambda$ or stationarily often? Theorem~5.4(ii$'$) gives the
dichotomy that one of the two sides is stationary; the pointwise question is
open. We note that $y_B(\lambda)$ is a single question at the level of the
back-and-forth relations, namely whether $B_\lambda\equiv_{\lambda+1}K_{\lambda+1}$,
rather than an isomorphism problem.

\smallskip
\noindent\textbf{Question 6.} Does some code of $\varphi$ lie in $L(c_0,T)$? This
is the antecedent of Lemma~CC(c); equivalently, does the hull construction of
[Sac07, \S5] place $\tc(\{\varphi\})$ below $c_0$? The text there does not settle
it.

\section{Context: Scott complexity of the labels}

[Alv21, Thm.~1.6, p.~1708]: the possible Scott complexities of countable structures are $\Pi_\alpha$ ($\alpha\ge1$), $\Sigma_\alpha$ ($\alpha\ge3$ a successor ordinal), and d-$\Sigma_\alpha$ (differences of $\Sigma_\alpha$ classes; $\alpha\ge1$ a successor ordinal); each is realized. We use the following from that paper: pp.~1706--1707, the convention that (simplest Scott sentence $\Sigma_\alpha$ or $\Pi_{\alpha+1}\Rightarrow$ Scott rank $\alpha$), the Wadge ordering ($\Sigma_\alpha,\Pi_\alpha<\text{d-}\Sigma_\alpha<\Sigma_{\alpha+1},\Pi_{\alpha+1}$), and Miller's [Mil83] both-$\Sigma_\alpha$-and-$\Pi_\alpha\Rightarrow\text{d-}\Sigma_\gamma$ for some $\gamma<\alpha$; p.~1709, Prop.~1.10(1)--(2) ($\Sigma_{\alpha+1}$ Scott sentence $\iff$ parameterized $\Pi_\alpha$; $\SRp=\min_{\bar c}\cSR(M,\bar c)$). The $(\cSR,\SRp)$ fingerprint table ($\Pi_{\alpha+1}\mapsto(\alpha,\alpha)$; $\Pi_\lambda\mapsto(\lambda,\lambda)$; d-$\Sigma_{\alpha+1}\mapsto(\alpha+1,\alpha)$; $\Sigma_{\lambda+1}\mapsto(\lambda+1,\lambda)$; $\Sigma_{\alpha+2}\mapsto(\alpha+2,\alpha)$) is not tabulated in [Alv21], but follows routinely from those four items.

\begin{rprop}[7.1]
$\SC(K_\delta)=\Pin{\delta+1}$ exactly and $\SRp(K_\delta)=\delta$ for every $\delta\ge\qr(\varphi)$; hence $(\cSR,\SRp)(K_{\lambda+1})=(\lambda+1,\lambda+1)$ and $\cSR(K_{\lambda+1},\bar c)\ge\lambda+1$ for every finite $\bar c$ ([Alv21, Prop.~1.10(1)]).
\end{rprop}
\begin{proof}
Fix $\delta\ge\qr(\varphi)$ and write $K:=K_\delta$. By Lemma~E, $\cSR(K)=\delta$
exactly, so $K$ has a $\Pin{\delta+1}$ Scott sentence ([Mon15, Thm.~1.1], (U2) at
$\alpha=\delta$), and by the catalogue quoted above $\SC(K)$ is one of its
entries. Three steps.
\emph{(a) No entry at level $\le\delta$.} Suppose $K$ had a Scott sentence
$\theta$ of complexity $\Sigma_\gamma$, $\Pi_\gamma$ or d-$\Sigma_\gamma$ with
$\gamma\le\delta$. Each such $\theta$ is a Boolean combination of $\Sin{\delta}$
sentences, so its truth value transfers across $\equiv_\delta$ (\S2.1). Every
member of $\Ccl_\delta$ is $\equiv_\delta K$ ($K\in\Ccl_\delta$ by Lemma~E, and
$\Ccl_\delta$ is a single $\equiv_\delta$-class), hence would satisfy $\theta$,
hence be isomorphic to $K$ --- collapsing $\Ccl_\delta$ to one isomorphism type,
against the uncountability of Lemma~B(i) (models counted up to isomorphism,
\S3.4).
\emph{(b) $\SC(K)=\Pin{\delta+1}$ exactly.} A $\Sigma_{\delta+1}$ Scott sentence
alongside the $\Pin{\delta+1}$ one would give, by Miller's theorem quoted above,
a d-$\Sigma_\gamma$ Scott sentence for some $\gamma<\delta+1$, contradicting (a);
and the entries d-$\Sigma_{\delta+1}$ and those at levels $\ge\delta+2$ lie
strictly above $\Pin{\delta+1}$ in the Wadge ordering quoted above, so none of
them is the least complexity while the $\Pin{\delta+1}$ Scott sentence exists.
With (a), the only remaining entry is $\Pin{\delta+1}$.
\emph{(c) $\SRp(K)=\delta$.} By [Mon15, p.~5431], $\cSR(K)\in\{\SRp(K),
\SRp(K)+1\}$. If $\delta$ is a limit, $\SRp(K)+1=\delta$ is impossible, so
$\SRp(K)=\delta$. If $\delta=\beta+1$ and $\SRp(K)=\beta$, then some
$(K,\bar c)$ has $\cSR=\beta$, i.e.\ a $\Pin{\delta}$ Scott sentence, so $K$ has
a $\Sigma_{\delta+1}$ Scott sentence ([Alv21, Prop.~1.10(1)] at $\alpha=\delta$)
--- contradicting (b). So $\SRp(K)=\delta$.
The displayed consequences follow at $\delta=\lambda+1$, the last by
[Alv21, Prop.~1.10(2)] as quoted ($\SRp=\min_{\bar c}\cSR(\cdot,\bar c)$).
\end{proof}
\begin{rrem}
The proof is ZFC (the Wadge ordering is used only as the printed order on the
catalogue's entries; no determinacy is invoked). Within \S5 the proposition is
used only at step~(1) of Lemma~5.14; its proof uses nothing from \S5.
\end{rrem}

\subsection*{7.H. The Harrison linear order}

The Harrison linear order $H$ is not a model of $\varphi$, and no proof in \S5
depends on this subsection (the only backward reference is the sharpness of
Corollary~(SR$^\sim$-Nadel), which no proof uses). We include it because it exhibits, outside the setting
of the rest of the paper, the configuration of branch~(III), and because it shows
that the upper bound in Corollary~(SR$^\sim$-Nadel) is attained. Facts specific
to $H$ are taken from [Mon13, p.~4] and [Alv21, p.~1710]; [Sac07] is used only
through its general theorems about an arbitrary countable structure, instantiated
at $A:=H$.

\emph{Setting.} $X\in2^\omega$; $H:=$ the Harrison linear order of $X$, of order type $\omega_1^X\cdot(1+\mathbb Q)$ ([Mon13, p.~4]); $\lambda:=\omega_1^X$ (admissible, [Mon13, p.~4]'s Sacks characterization); $T^H_\delta,L^H_\delta:=$ the canonical tower of $H$ itself ([Sac07, \S2, clauses (1)--(4)]); the ambient theory is $T_H:=$ the (recursive) linear-order axioms, $H\models T_H$.

\smallskip
\noindent\textbf{(H-i) $\omega_1^H=\omega_1^X=\lambda$.} (Both sides are in the rank $\AKR$, which is the rank written $\mathrm{SR}$ in [Mon13].) \emph{Proof.} $\le$: $H$ has an $X$-computable presentation ([Mon13, p.~4]), so $X\in\Sp(H)$ and $\omega_1^H\le\omega_1^X$. $\ge$: [Mon13, p.~4] states, for general $X$, that the Nadel bound is attained at $H$, that is, $\mathrm{SR}(H)=\omega_1^X+1$; the Nadel bound on the same page gives $\mathrm{SR}(H)\le\omega_1^H+1$; so $\omega_1^X+1\le\omega_1^H+1$.

\smallskip
\noindent\textbf{(H-ii) $H$ is a homogeneous model of $T^H_{\omega_1^H}=T^H_\lambda$.} This is [Sac07, (2.6), p.~5] instantiated at $H$; we shall also use the extension form established in its proof at (2.9)--(2.11).

\smallskip
\noindent\textbf{(H-growth).} Every $L^H_\delta$-formula ($\delta\le\lambda$) has infinitary complexity $\qr<\lambda$. \emph{Proof.} (This is the analogue of Lemma~AL(vii-1) for $H$; that lemma is not applicable, since it concerns the node theory $T_\lambda$ at $\lambda\in\Szero$, whereas here the tower is that of $H$ itself.) $d(\delta):=\sup\{\qr(F)+1:F\in L^H_\delta\}$. Base: $L^H_0$ is the finitary language of linear orders ([Sac07] clause (1)), so $d(0)\le\omega$. Successor: $L^H_{\delta+1}$ adds $\bigwedge p$ for every non-principal $n$-type of $T^H_\delta$ realized in $H$ ([Sac07] clause (4), p.~5), each of $\qr\le d(\delta)+1$, then finitary closure: $d(\delta+1)\le d(\delta)+\omega$. Limits: unions ([Sac07] clause (2)). Induction: $d(\delta)\le\omega\cdot(1+\delta)$; for $\delta<\lambda$ this is $<\lambda$, since $\lambda=\omega_1^X$ is admissible ([Mon13, p.~4]) hence multiplicatively closed. $L^H_\lambda=\bigcup_{\delta<\lambda}L^H_\delta$ closes the claim.

\smallskip
\noindent\textbf{(H-iii) $H$ realizes at least one non-principal $n$-type of $T^H_\lambda$.} \emph{Proof.} Suppose every complete $T^H_\lambda$-type realized in $H$ is principal.
\emph{(1) The lower bound $\cSR(H)>\lambda$.} [Mon13, p.~4] gives $\mathrm{SR}(H)=\omega_1^X+1=\lambda+1$, in the rank $\AKR$ (TV-c(i)). If $\cSR(H)\le\lambda$: either $\cSR(H)=\lambda$, a limit value, at which $\AKR$ and $\cSR$ coincide (TV-c(ii), [Mon15, p.~5433])---forcing $\AKR(H)=\lambda\ne\lambda+1$; or $\cSR(H)<\lambda$, whence $\AKR(H)\le\cSR(H)+1<\lambda+1$ (the two ranks differ by at most $1$, by the same clause). Both contradict $\AKR(H)=\lambda+1$. So $\cSR(H)\ge\lambda+1$. (Only TV-c is used here, so this is not a rank transport.) (Sharper form, noted: [Alv21, p.~1710] states $\SC(H)=\Pi_{\woneCK+2}$; at $X=\emptyset$ this gives $\cSR(H)=\lambda+1$ exactly, via the fingerprint table of \S7; for general $X$ that statement relativizes routinely. The present proof uses only $\cSR(H)>\lambda$.) \emph{(2) Pin and homogeneity:} $\omega_1^H=\lambda$ [(H-i)]; $H$ homogeneous over $T^H_\lambda$ [(H-ii)]; the classical countable back-and-forth closure of the extension property of [Sac07, (2.6)] yields: tuples $\bar a,\bar b$ with equal complete $T^H_\lambda$-types are automorphic, coordinatewise (written $H$-side; the AL(v)-pattern is not cited, its ambient differing). \emph{(3) Orbit computation:} let $\bar a\in H^n$ and $p_{\bar a}:=$ its complete $T^H_\lambda$-type---realized, hence principal by assumption: some atom $b\in L^H_\lambda(\bar x)$ generates it. $T^H_\lambda=\Th_{L^H_\lambda}(H)$ is a truth theory ([Sac07] clause (3), p.~4), so the generation sentences $\forall\bar x(b\to F)$, $F\in p_{\bar a}$, hold in $H$ definitionally. Every $\bar c$ with $H\models b(\bar c)$ satisfies all of $p_{\bar a}$, hence realizes it, hence shares $\bar a$'s complete $T^H_\lambda$-type, hence is automorphic to $\bar a$ [(2)]; conversely automorphic tuples satisfy $b$. So every $\Aut(H)$-orbit is defined by a parameter-free formula $b$ of $\qr<\lambda$ [(H-growth)], hence $\Sin{\lambda}$-expressible (only the form of the definition is used here). \emph{(4) Contradiction:} parameter-free $\Sin{\lambda}$-definability of all orbits is [Mon15, (U1) at $\alpha=\lambda$, Thm.~1.1, pp.~5427--5428] $\Rightarrow$ (U2): a $\Pin{\lambda+1}$ Scott sentence $\Rightarrow\cSR(H)\le\lambda$, contradicting~(1). (Only TV-c is used, in the direction available from [Mon15, p.~5433]; $\srS$ does not appear in this proof, so there is no rank transport.)

\smallskip
\noindent\textbf{(H-sr) $\srS(H)=\lambda+1$.} \emph{Proof.} If $\srS(H)=\delta\le\lambda$: $H$ is the atomic model of $T^H_\delta$ ([Sac07, (2.5)]), so every realized complete $T^H_\delta$-type is principal; then clause (4) adds nothing at $\delta$, $L^H_{\delta+1}=L^H_\delta$ and $T^H_{\delta+1}=T^H_\delta$; by induction (unions of a constant chain at limits) $T^H_\gamma=T^H_\delta$ for all $\gamma\in[\delta,\lambda]$---so every realized $T^H_\lambda$-type is a realized $T^H_\delta$-type, all principal, contradicting (H-iii). Hence $\srS(H)\ge\lambda+1$; the Nadel bound (\S2.7; [Sac07, p.~5]) with (H-i) gives $\srS(H)\le\omega_1^H+1=\lambda+1$. (This proof is entirely in terms of $\srS$; no equation involving $\cSR$ occurs in it.)

\smallskip
\noindent\textbf{The locus fact.} $\omega_1^{T_H,H}=\lambda$: $T_H$ is recursive and $H$ has an $X$-computable presentation, so a presentation of the pair is computable from $X$ and $\omega_1^{\langle T_H,H\rangle}=\omega_1^X=\lambda$ (the join of recursives keeps the least admissible at $\omega_1^X$; [Mon13, p.~4] facts); and $\srS(H)=\lambda+1\ge\lambda=\omega_1^{T_H,H}$ by (H-sr). So the criterion of Lemma~JP is satisfied at $H$, and the condition discussed next is the only remaining hypothesis of [Sac07, Thm.~8.1] at $H$.

\smallskip
\noindent\textbf{(H-iv).} One may ask whether $H$ is $\lambda$-saturated in the sense of
[Sac07, p.~14], that is, whether every $n$-type of $T^H_\lambda$ is realized in
$H$. The natural route to this passes through the condition
$|S_{T^H_\beta}|\le\aleph_0$ for $\beta\le\lambda$, that is, weak scattering of
$T_H$ in $L(\lambda,\langle T_H,H\rangle)$, with the top level $\beta=\lambda$ the
essential case. Proposition~HC-2 below shows that this condition fails already at
its lowest level, since $|S_{T^H_0}|=2^\omega$. Hence $T_H$ is not weakly scattered in $\mathbb A_H:=L(\lambda,\langle T_H,H\rangle)$ and the first hypothesis of [Sac07, Thm.~8.1] \emph{fails} at $(T_H,H)$: there is no genuine instance of that theorem at $H$. The locus fact above is unaffected; the hypothesis it left as the only one still to be verified is now shown to fail.

\emph{Remark (the scope of the hypothesis in [Sac07, Thm.~8.1]; see pp.~21--24 there).} Independently of the refutation, the condition just described is not quite the hypothesis of that theorem. \textup{(D1)} The top level $\beta=\lambda$ is not in the range of the hypothesis as stated ($T^H_\lambda\notin\mathbb A_H$, so the p.~23 quantifier ``for all $T'\in A$'' never reaches it) and is not used by the proof given there (the Thm.~3.3 step is invoked exactly ``for all $\beta$ such that $\beta+1<\mathrm{sr}(A)$'', which at $H$ is $\beta<\lambda$; the only top-level contact is the p.~24 restriction chain (8.20)--(8.22), handling realized types one at a time, using no type \emph{set}); \textup{(D2)} below $\lambda$ the hypothesis is broader than the tower family (it ranges over all $\omega$-complete extensions of $T_H$ in $\mathbb A_H$), a hierarchy-wide use belonging to the enumeration of [Sac07, \S8] ((8.12), p.~21) and to the split apparatus of \S9 there, not to the proof of 8.1; \textup{(D2$'$)} the proof-sufficient narrowing 8.1$'$ (same conclusion from: $S_{T^A_\beta}\in$ locus for all $\beta$ with $\beta+1<\mathrm{sr}(A)$, plus Thm.~3.3's lightface uniformity) is obtained by inspection of the proof given there. We do not claim it as a statement of [Sac07].

\emph{Remark.} The results we know of on back-and-forth relations for linear orders --- [AK00, \S15], in particular Prop.~15.1 there, the work of Ash [Ash86] on well-ordered blocks, and the machinery of $\sim_\alpha$ in [Har18] --- all classify data attached to tuples, whereas what is wanted here is control on the cardinality of the set of consistent $n$-types of the tower theory. Passing between the two is exactly the difficulty recorded as Question~3 of \S6; we have not found a way across it. Proposition~HC-2 below is a \emph{lower} bound by exhibited realizations, the opposite direction from the unavailable upper-bound transfer.

\subsubsection*{Resolution of (H-iv): $|S_{T^H_0}|=2^\omega$}

\emph{Setting and conventions} (with the sourcing convention stated at the start of \S7.H): $X\in2^\omega$; $H=$ the Harrison linear order of $X$, order type $\omega_1^X\cdot(1+\mathbb Q)$ [Mon13, p.~4]; $\lambda:=\omega_1^X=\omega_1^{T_H,H}$ [the locus fact]; $T^H_\beta,L^H_\beta=H$'s canonical Scott tower ([Sac07, \S2, clauses (1)--(4)]), so $T^H_0=\Th_{L^H_0}(H)$ is the complete finitary first-order theory of $H$; $\mathbb A_H:=L(\lambda,\langle T_H,H\rangle)$; we write $\mathrm{sr}$ for $\srS$ throughout. The ambient theory is ZFC; no determinacy hypothesis and no large cardinal is used, and no rank transport occurs (all statements below concern $\srS$ and $\omega_1$-invariants only).

\begin{rlem}[HC-1: the tower inside $\mathbb A_H$, and a dichotomy at each level]
\textup{(i)} For every $\beta<\lambda$: $T^H_\beta\in\mathbb A_H$ and the formula set of $L^H_\beta(\bar x)\in\mathbb A_H$. \textup{(ii)} For every $\beta<\lambda$: either $S_{T^H_\beta}\in\mathbb A_H$ (hence countable) or $|S_{T^H_\beta}|=2^\omega$. At $\beta=\lambda$ the machinery does not apply inside $\mathbb A_H$: $T^H_\lambda$ and the level-$\lambda$ formula set have rank $\lambda=o(\mathbb A_H)$ and are not elements of $\mathbb A_H$.
\end{rlem}
\noindent \textup{(i)} is routine ($\Sigma_1$-recursion inside $\mathbb A_H$ on the tower clauses with $H$ the distinguished element for clause (4)'s realized-in-$H$ restriction; the inside-an-admissible effectivization pattern is stated at [Sac07, (8.9)--(8.11), p.~21]; typehood is $\Delta_0$ per the p.~9 line in Prop.~4.4's proof; fragment satisfaction in $H$ is $\Delta_1$ by the standard KP $\Sigma_1$-recursion). \textup{(ii)} is [Sac07, Thm.~3.1, p.~6]---``If $S_{p,b}\notin A$, then the cardinality of $S_{p,b}$ is $2^\omega$'', with $S_{p,b}$ in the notation of that theorem, the instances used here being the type-sets $S_{T^H_\beta}$ in the form supplied by (i) and $\Delta_0$-typehood---together with [Sac07, Cor.~3.2, p.~7].

\begin{rprop}[HC-2: $|S_{T^H_0}|=2^\omega$]
$|S_{T^H_0}|=2^\omega$.
\end{rprop}
\emph{Detector hierarchy} (FO, uniform): $\ell_1(x):=$ ``$x$ has a predecessor and no immediate predecessor''; $\ell_{j+1}(x):=\ell_1(x)\wedge$ ``the $\ell_j$-points below $x$ are cofinal in the predecessors of $x$''; $\ell_j$ has finitary quantifier depth $\le2j+2$; in an ordinal $\ell_j$ picks exactly the nonzero multiples of $\omega^j$, and the definition is local-structural, so it reads the same profile in ill-founded orders. For $c\in\{1,2\}^\omega$ set $D_c:=$ the reversed-$\omega$-indexed sum $\cdots+\omega^3 c_3+\omega^2 c_2+\omega c_1+c_0$ (an ill-founded order: an infinite descending sequence of block-groups, the finite-rank groups on top); $B_c:=\omega^\omega+D_c$; $M_c:=B_c+1+H$; $a:=$ the distinguished middle point, so $(-\infty,a)_{M_c}=B_c$ and $(a,\infty)_{M_c}\iso H$. [Notation guard: this $B_c$ is a linear order local to this passage, unrelated to the frame's $B_\lambda$.]

\emph{Proof.} \textbf{(i)} (This step uses two classical facts, named below, for which we give references in Appendix~A.6.) For every $n\ge1$: $\omega^\omega+D_{>n}(c)\equiv_n\omega^\omega+\delta_n$ for an ordinal $\delta_n$ independent of $c$; hence by summand congruence $B_c\equiv_n(\omega^\omega+\delta_n)+E_n(c)=:\gamma_n(c)$, an ordinal $<\omega^\omega\cdot2$, where $E_n(c):=\omega^n c_n+\cdots+\omega c_1+c_0<\omega^{n+1}\cdot3$. The two classical facts are (see [Ros82, Ch.~6] for the calculus of $n$-characteristics): \emph{(FV-sum)} [$A\equiv_k A'\Rightarrow A+C\equiv_k A'+C$, ordered sums] and \emph{(Ord-$\equiv$)} [the ordinal $\equiv_k$ calculus: ordinals $\ge\omega^k$ congruent mod $\omega^k$ are $\equiv_k$, with the block-absorption Duplicator strategy for rank-$\ge k$ descending sums]; neither is stated in this form in the sources otherwise cited here (the closest statements in the sources otherwise cited here are [AK00, Lemmas~15.7--15.10, pp.~244--245]; see A.6). \textbf{(ii)} From the order type of $H$: $\gamma+1+H\iso H$ for every $\gamma<\lambda$: $\lambda=\omega_1^X$ is admissible hence additively indecomposable, so $\gamma+1+\lambda=\lambda$ and the $\lambda\cdot\mathbb Q$ tail is unchanged. \textbf{(iii)} $M_c=B_c+(1+H)\equiv_n\gamma_n(c)+1+H\iso H$ for every $n$ [(FV-sum) and (ii)], so $M_c\models T^H_0$. \textbf{(iv)} [type separation; computed] for $i\ge1$ let $\theta_i(x):=$ ``exactly $c_i$ points $y<x$ satisfy: $\ell_i(y)$ and $y$ is strictly above every $z<x$ with $\ell_{i+1}(z)$'' (bounded counting $\le2$; an $\ell_{i+1}$-point below $x$ exists), and $\theta_0(x):=$ ``the greatest $y<x$ is a $1$-limit'' ($c_0=1$) / ``$\ldots$is the successor of a $1$-limit'' ($c_0=2$); $\theta_i$ has finitary quantifier depth $O(i)$. Computation of the count in $B_c$ below $a$: the greatest $\ell_{i+1}$-point of $B_c$ is the start of the first $\omega^i$-copy in the $\omega^i c_i$-group; strictly above it, the $\ell_i$-points are exactly the remaining $(c_i-1)$ starts of the $\omega^i$-group plus the single first start of the $\omega^{i-1}$-group, and nothing else---total exactly $c_i$. So $\tp_{M_c}(a)$ reads off every bit of $c$; the map $c\mapsto\tp_{M_c}(a)$ is injective; each value is a maximal set of $L^H_0(x)$-formulas finitarily consistent with $T^H_0$ (realized in $M_c\models T^H_0$). \hfill$\square$

\smallskip
\noindent The argument is routine modulo the two classical facts (FV-sum) and
(Ord-$\equiv$) named in step (i); see Appendix~A.6 for their sources.

\smallskip
\noindent\textbf{Corollaries.} \textup{(R1)} $T^H_0\in\mathbb A_H$ but $S_{T^H_0}\notin\mathbb A_H$ (uncountable), so by the definition at [Sac07, p.~23] $T_H$ is \emph{not} weakly scattered in $\mathbb A_H$: the first hypothesis of [Sac07, Thm.~8.1] fails at $(T_H,H)$; there is no genuine instance of that theorem at $H$. (H-iv) is thereby resolved negatively; HC-1(ii)'s dichotomy lands on the continuum side at $\beta=0$. \textup{(R2) (ambient-independence):} any ambient theory in the pure linear-order signature satisfied by $H$ has the same level-$0$ truth theory (the complete finitary first-order theory of $H$), so the refutation applies as stated---so re-basing on a scattered recursive ambient theory $T_H'$ does not help, within the linear-order signature (a signature-expanded ambient changes $L^H_0$ and is not covered). \textup{(R3)} nothing further is to be proved toward the condition at $H$, since it is false; the relation asked for in Question~3 of \S6 would have been a route to an upper bound that does not exist, whereas Proposition~HC-2 is a lower bound obtained from explicit realizations. \textup{(R4)} (H-i)--(H-sr), the locus fact and Lemma~JP are all consistent with the refutation just recorded; the reason is that $T_H$ is far from weakly scattered: its truth tower at $H$ already carries continuum many types at the finitary level. \textup{(R5)} the saturation comparison of Lemma~AL(iv) does not arise at $H$. A secondary remark, independent of the above: $q(x,y):=\{$``$y$ is a $1$-limit or the minimum''$\}\cup\{$``no $\ell_1$-point $z$ with $y<z\le x$''$\}\cup\{$``$x>y+n$'': $n<\omega\}$ is finitarily consistent with $T^H_1$ and unrealized in $H$---a consistent-vs-realized separation from level $1$ on, from the order type only.

\begin{rrem}[Boolean-algebra formulation]
Superatomic Boolean algebras make genuine contact: HC-1(ii) is the Stone-dual statement that the fragment Lindenbaum algebra $\mathfrak B_\beta(\bar x)$ over $T^H_\beta$ is superatomic or carries a perfect tree of consistent conditions, and Proposition HC-2 says $\mathfrak B_0$ is \emph{not} superatomic. (The absoluteness of Proposition~HC-2 is discussed in Appendix~A.3.)
\end{rrem}

\begin{rrem}
The list of high Scott complexities at [Alv21, p.~1709] includes an entry stated as
``d--$\Pi_{\omega_1^A+1}$'', there and in Theorem~1.8 on the same page, whereas
the catalogue of [Alv21, Thm.~1.6, p.~1708] and the summary at [Alv21, p.~1712]
admit only $\Sigma_\alpha$, $\Pi_\alpha$ and d--$\Sigma_\alpha$. The two are
reconciled by the fact that a difference of two $\Pi_\alpha$ sets and a
difference of two $\Sigma_\alpha$ sets form the same class, so the entry may be
read either way. We use only the $\Pi$ side of that list, so nothing above
depends on how it is read.
\end{rrem}

\appendix
\section{Verification and sources}

\emph{This appendix records: the conventions on attribution used above (A.1);
consistency checks against a list of known structures (A.2); the absoluteness of
the statements proved (A.3); the fact that no determinacy hypothesis is used
(A.4); the two results proved here by assembly, and the rank transports (A.5);
the classical facts used without a located published proof (A.6); a
concordance between the preprint and published editions of [Sac07] (A.7); and the
citation conventions (A.8).}

\subsection*{A.1. Attribution}
Every result above either cites, at each step, a published source with page, or
is marked explicitly as assembled from cited statements, or is proved in full.
Where a result is described as routine, a proof or sketch is given at the item.
Three results below are proved by assembly from statements in the literature
rather than cited: Theorem~(TV-b) and Proposition~(AK-$\sim$-bridge) of \S2.5,
and the club form of $\Ffix$ in Lemma~F; each is identified as such at the
point where it occurs, and again in A.5.

\subsection*{A.2. Consistency checks}
Each lemma and theorem above was checked against the following structures and
families: well-orders; the Harrison linear order; superatomic Boolean algebras;
$\omega$-sums of Boolean algebras with rank increasing along the sum; the
$1$-transitive linear orders (H.~Friedman's example, described by Sacks
[Sac83]; see also [GRT, p.~3]); Koerwien's $\omega$-stable theory with
non-Borel isomorphism relation [Koe11]; and Newelski's small weakly minimal
theory with a type of $M$-rank $\infty$ [New98]. None of them contradicts any statement above. The
informative cases are these. The Harrison order is treated in \S7.H; it exhibits
the configuration of branch~(III), and the identity
$\SRsym(H)=\AKR(H)=\woneCK+1$, obtained from TV-a together with (H-sr) and
[Mon13, p.~4], is an instance of Theorem~(TV-b) established by a route
independent of the argument given for that theorem. The upper bound
$\SRsym(A)\le\omega_1^A+1$ of Corollary~(SR$^\sim$-Nadel) is attained there, so
the bound is sharp. The $1$-transitive linear orders form a
$\mathrm{PC}_{\omega_1\omega}$ class with exactly one model of each Scott rank
([Sac83]; [GRT, p.~3]), outside the scope of Lemma~S; a member with $\AKR=\omega_1^A+1$ and $\SRsym\le\omega_1^A$ is excluded
outright \emph{by} Theorem~(TV-b); unlike the Harrison row, which is
route-independent, this row is a consequence of the theorem and is not
independent corroboration of it. Superatomic Boolean algebras enter
through \S7.H only: part (ii) of Lemma~HC-1 is, in Stone duality, the statement
that the Lindenbaum algebra of the fragment over $T^H_\beta$ is either
superatomic or carries a perfect tree of consistent conditions, and
Proposition~HC-2 says that at $\beta=0$ it is not superatomic. The remaining
families make no contact: well-orders lie below the values at which the results
above apply; the $\omega$-sums meet only the clauses about limit values;
Koerwien's example is not scattered and is excluded wherever scatteredness is a
hypothesis, in particular at step (C1) of \S5.11; and the $M$-rank of Newelski's
example is a finitary notion unrelated to the ranks used here.

\subsection*{A.3. Absoluteness}
Statements about $\varphi$, $\Ccl_\beta$, $K_\beta$, $W_\lambda$, $B_\lambda$ and
$A_\lambda$ are read in $\omega_1$-preserving extensions. Each theorem above is,
instance by instance, a Boolean combination of $\Sigma^1_2$ predicates of reals:
rank values; the identities $\omega_1^M=\lambda$, where the inequality
$\ge$ is free by Lemma~FL, so that only the $\le$ half has to be expressed; and
data about fragment types, which are Borel in codes once the branch parameter $p$
of \S3.7 is fixed. Shoenfield absoluteness therefore applies to each instance.
Claims of stationarity are internal to the model of ZFC in which they are made,
and are re-derived in each extension. In particular, membership in $\Spec_\lambda$
is $\Sigma^1_2$ in a code, as set out in \S5.11, so it is absolute, and
$\Spec_\lambda$ acquires new members in an extension only at sets $S$ that are
not in the ground model. Lemma~CC is a theorem schema of ZFC, and Theorem~(TV-b)
with its two corollaries is a theorem schema of ZFC about countable structures.
Proposition~HC-2 is $\Sigma^1_2$ in codes, since $c\mapsto\tp_{M_c}(a)$ is
arithmetic in $c$ together with a code for $H$, and is therefore absolute and
upward persistent.

\subsection*{A.4. Determinacy}
No determinacy hypothesis and no large cardinal is used anywhere above; the base
theory is ZFC throughout. Determinacy occurs only in the works cited in \S1.4,
where it is noted by their authors: [Mon13, Thm.~3.2]; the classification of
Scott spectra in [Har18]; and the conditional minimality criteria of
[Bec94, \S\S4, 6], where projective determinacy is named at p.~780.

\subsection*{A.5. The assembled results, and the rank transports}
Theorem~(TV-b) --- that $\SRsym(A)=v\iff\AKR(A)=v$ for
$v\in\{\omega_1^A,\omega_1^A+1\}$ --- is assembled in \S2.5 from the
$\sim$-relations of [Mon15, p.~5432], the back-and-forth relations and Karp's
theorem at [GRT, p.~6], and the clause for $\rhole$ together with the Nadel bound
at [Mon13, pp.~3--4]. The source for $\AKR$ itself is [AK00, \S6.7, p.~98], where
a value identity of the same shape is stated without proof for the symmetric rank
$\SRAK$ of that book; as explained in the remark in \S2.5, that statement does not
yield Theorem~(TV-b), because $\SRAK$ and $\SRsym$ are suprema over different
families of relations, and we do not use it. We have found no published statement
of Theorem~(TV-b) in the form needed here.

Proposition~(AK-$\sim$-bridge) --- that $\SRAK(A)=v\iff\SRsym(A)=v$ for
$v\in\{\omega_1^A,\omega_1^A+1\}$ --- is likewise assembled in \S2.5, from the
clauses of [AK00, \S6.7] as read there (with the footnote on p.~99), the
$\sim$-relations of [Mon15, p.~5432], the Scott clause, and the two Nadel-form
caps. It is consumed in no proof; it is included because the comparison with
[AK00, \S6.7] raises the question and the sources cited settle neither
direction. Composed with the unproved value identity quoted in \S2.5 from
[AK00, \S6.7], it would re-derive Theorem~(TV-b); we rely on this in neither
direction.

The club form of $\Ffix$ (Lemma~F(ii), (iii)) is likewise assembled: [GRT] states
and proves only the conclusion at each fixed point, and asserts no club; the
closure and unboundedness arguments are given in full at Lemma~F, from
ingredients in the two paragraphs of the proof of [GRT, Thm.~3.6, p.~14].

Finally, every proof above that passes between $\srS$ and $\cSR$ does so
through Lemma~TV of \S2.5, which records the complete list of such steps:
step~(3) of the proof of Lemma~5.9; Corollary~5.10; Lemma~AL(vii-2); and
Lemma~S(iv)--(v), hence Theorem~5.3(i)--(ii). Every other proof stays on one
side or the other.

\subsection*{A.6. Classical facts used without a located proof}
Three classical facts are used above for which we give the standard references
but not proofs.
\begin{itemize}[leftmargin=1.4em]
\item \emph{(FV-sum):} $A\equiv_k A'\Rightarrow A+C\equiv_k A'+C$ for ordered
sums of linear orders; and \emph{(Ord-$\equiv$):} the calculus of
quantifier-rank-$k$ equivalence for ordinals, by which ordinals $\ge\omega^k$ that
are congruent modulo $\omega^k$ are $\equiv_k$. Both are used in
Proposition~HC-2. The finitary forms are in Rosenstein [Ros82, Ch.~6]; the
closest statements in the sources otherwise cited here are
[AK00, Lemmas 15.7--15.10, pp.~244--245], which prove the corresponding
congruence for ordered sums and intervals, and the full calculus for well-orders,
in terms of the standard back-and-forth relations rather than
Ehrenfeucht--Fra\"iss\'e games.
\item \emph{(SC-code):} $\alpha$ is $v$-admissible if and only if
$\alpha=\omega_1^{v\oplus y}$ for some $y$. This is stated at
[Bec94, p.~764, \S0.E] and attributed there to Sacks; the proofs are in Sacks,
\emph{Countable admissible ordinals and hyperdegrees}, Adv.\ Math.\ 20 (1976),
213--262, and in Steel, \emph{Forcing with tagged trees}, Ann.\ Math.\ Logic 15
(1978), 55--74. It is not used in any proof above.
\end{itemize}
The background facts about Kripke--Platek set theory used in Lemma~CC are
standard; see [Bar75, Ch.~I, \S\S6--8 and Ch.~II, \S\S1, 5--6] and, for Barwise
compactness and control of the standard part, [Bar75, Ch.~III, 5.6, p.~99 and
7.5, pp.~107--109, together with 3.8, p.~91]. The omitting-types step in
[Sac07, Prop.~4.7] rests on Grilliot [Gri72, pp.~85--88].

\subsection*{A.7. Concordance for {[Sac07]}}
All citations to [Sac07] above use the numbering of the December~9, 2004
preprint. In the published edition (Notre Dame J.\ Formal Logic 48 (2007), no.~1,
5--31) every theorem, proposition, corollary and equation number cited here is
unchanged in \S\S1--8; only the pages differ. The correspondences are:
Prop.~5.1, p.~13 $\to$ p.~16; Prop.~5.2, Thm.~5.3 and the definition of
$\alpha$-saturation, p.~14 $\to$ pp.~16--17; the proof of Thm.~5.3, p.~15 $\to$
p.~17; Prop.~4.3, p.~9 $\to$ p.~12; Props.~4.5 and 4.6, p.~10 $\to$ p.~13;
Thm.~4.9, pp.~12--13 $\to$ pp.~14--15; Thm.~3.1, Cor.~3.2 and Thm.~3.3,
pp.~6--7 $\to$ pp.~10--11; Prop.~2.1, p.~6 $\to$ p.~9; (2.5)--(2.12), pp.~5--6
$\to$ p.~9; Thm.~6.1, p.~15 $\to$ p.~17; Thm.~8.1 and (8.19)--(8.22),
pp.~23--24 $\to$ pp.~24--25; Cor.~8.2, p.~24 $\to$ pp.~25--26; Thm.~9.1, p.~25
$\to$ p.~27; the restatement in \S1 of the bounding result, p.~3 $\to$ p.~7;
Cor.~6.2 and (6.1), p.~16 $\to$ p.~18. The remaining bare page references are to
unnumbered material; since pagination is monotone, the rows above bracket their
published locations: the \S2 recursion clauses (1)--(4) (pp.~4--5 here) lie on
pp.~7--9; the \S4 interface and the review of $\omega$-completeness (pp.~8--9)
on pp.~11--12; Props.~4.7--4.8 (pp.~10--11) on pp.~13--14; and the displays
(8.9)--(8.11) (p.~21) between p.~18 and Thm.~8.1 at p.~24. Section 9 is renumbered below (9.6), but no item of \S9 beyond
Thm.~9.1 is cited here. The published edition labels the review of fragments in
\S4 as \S4.1 and renumbers the reference list; the typographical slip in (2.5)
noted in \S2.4 is present in both editions.

\subsection*{A.8. Citation conventions}
Throughout, citations to Sacks, \emph{Bounds on weak scattering}, are to [Sac07]
and follow the preprint numbering as just described. Page references to
[Mon15] are to the published edition (Proc.\ Amer.\ Math.\ Soc.\ 143.12 (2015),
5427--5436) and have been verified against it; in particular the clause quoted
in the remark in \S2.4 is quoted as printed there. Page references to [Alv21]
are likewise to the published edition (J.\ Symb.\ Log.\ 86.4 (2021),
1706--1720) and have been verified against it. Page references to [Mon13] and
[GM23] follow the publicly posted preprint versions of those two papers.


\end{document}